\documentclass[11pt]{article}
\usepackage{t1enc}
\usepackage{amsmath,amssymb,color,hyperref,enumitem,xcolor}
\usepackage{authblk}

\newtheorem{Thm}{Theorem}
\newtheorem{Def}{Definition}
\newtheorem{Prop}{Proposition}

\newtheorem{Lemma}{Lemma}

\newtheorem{Rem}{Remark}
\newtheorem{Exp}{Example}

\newenvironment{proof}[1][Proof]{\textbf{#1.} }{\hfill $\square$}

\newcommand{\vertiii}[1]{{\left\vert\kern-0.25ex\left\vert\kern-0.25ex\left\vert #1 
    \right\vert\kern-0.25ex\right\vert\kern-0.25ex\right\vert}}
\newcommand{\mart}{{\rm M}^\sharp}

\newcommand{\cad}{c\`adl\`ag }
\newcommand{\what}{\widehat}

\newcommand{\dd}[0]{\,\mathrm{d}}

\def \eps{\varepsilon}
\def \bN{\mathbb{N}}
\def \bR{\mathbb{R}}
\def \bQ{\mathbb{Q}}

\def \fd{\mathfrak{d}}

\def \bM{\mathbb{M}}
\def \bL{\mathbb{L}}
\def \bE{\mathbb{E}}
\def \bF{\mathcal{F}}
\def \bF{\mathbb{F}}

\def \bH{\mathbb{H}}
\def \bD{\mathbb{D}}
\def \bS{\mathbb{S}}
\def \tOm{\widetilde{\Omega}}

\def \cE{\mathcal{E}}

\def \cP{\mathcal{P}}

\def \cA{\mathcal{A}}
\def \cB{\mathcal{B}}

\def \cF{\mathcal{F}}
\def \cH{\mathcal{H}} 
\def \cM{\mathcal{M}}

\def \tP{\widetilde{\mathcal{P}}}
\def \tpi{\widetilde{\pi}}

\def \bP{\mathbb{P}}
\def \1{\mathbf{1}}

\newcommand*{\trace}{\operatorname{Tr}}

\begin{document}

\title{Backward stochastic Volterra integral equations with jumps in a general filtration.}

\author{Alexandre Popier\thanks{
e-mail: {\tt  Alexandre.Popier@univ-lemans.fr}.}}
\affil{\small Laboratoire Manceau de Math\'ematiques, Le Mans Universit\'e, Avenue O. Messiaen, 72085 Le Mans cedex 9, France.  }

\date{\today}

\maketitle

\begin{abstract}
In this paper, we study backward stochastic Volterra integral equations introduced in \cite{lin:02,yong:06} and extend the existence, uniqueness or comparison results for general filtration as in \cite{papa:poss:sapl:18} (not only Brownian-Poisson setting). We also consider $L^p$-data and explore the time regularity of the solution in the It\^o setting, which is also new in this jump setting.
\end{abstract}

\vspace{0.5cm}


\section*{Introduction}

The aim of this paper is to extend or to adapt some results concerning backward stochastic Volterra integral equations (BSVIEs in short). To our best knowledge, \cite{lin:02,yong:06,yong:08} were the first papers dealing with BSVIEs and the authors considered the following class of BSVIEs:
\begin{equation} \label{eq:brownian_BSVIE}
Y(t) =\Phi(t) + \int_t^{T} f(t,s, Y(s), Z(t,s) ,Z(s,t)) \dd s - \int_t^{T} Z(t,s) \dd W_s.
\end{equation}
$W$ is a $k$-dimensional Brownian motion, $f$ is called the {\it generator} or the {\it driver} of the BSVIE and $\Phi$ is the {\it free term} (or sometimes the terminal condition). 
They proved existence and uniqueness of the solution $(Y,Z)$ (M-solution in \cite{yong:08}) under the natural Lipschitz continuity regularity of $f$ and square integrability condition on the data. 

Let us focus on two particular cases. If $f$ and $\Phi$ don't depend on $t$, we obtain a backward stochastic differential equation (BSDE for short):
\begin{equation*} \label{eq:brownian_BSDE}
Y(t) =\xi + \int_t^{T} f(s, Y(s), Z(s) ) \dd s - \int_t^{T} Z(s) \dd W_s.
\end{equation*}
Since the seminal paper \cite{pard:peng:90}, it has been intensively studied (see among many others \cite{delo:13,elka:peng:quen:97,pard:rasc:14,situ:05}). Expanding the paper \cite{elka:huan:97}, Papapantoleon et al. \cite{papa:poss:sapl:18} studied BSDEs of the form:
\begin{eqnarray}  \nonumber
Y(t) &= & \xi + \int_t^{T} f(s, Y(s), Z(s),U(s) ) \dd B_s \\ \label{eq:poss_BSDE}
& -&  \int_t^{T} Z(s) \dd X^\circ_s - \int_t^T \int_{\bR^m} U(s,x) \tpi^\natural(\dd x,\dd s) - \int_t^T \dd M(s).
\end{eqnarray}
The underlying filtration $\bF$ satisfies only the usual hypotheses (completeness and right-continuity). The exact definition of the processes $B$, $X^\circ$, $\tpi^\natural$ is given in Section \ref{sect:setting_not}. Roughly speaking, $X^\circ$ is a square-integrable martingale, $\tpi^\natural$ is integer-valued random measure, such that each component of $\langle X^\circ\rangle$ ia absolutely continuous w.r.t. $B$ and the disintegration property given $B$ holds for the compensator $\nu^\natural$ of $\tpi^\natural$. The martingale $M$ appears naturally in the martingale representation since no additional assumption on the filtration $\bF$ is assumed. Their setting contains the particular case where $X^\circ = W $ and $\tpi^\natural$ is a Poisson random measure (Example \ref{exp:brownian_poisson_setting}), but also many others (see the introduction of \cite{papa:poss:sapl:18}). 

The second particular case of \eqref{eq:brownian_BSVIE} is called Type-I BSVIE:
\begin{equation} \label{eq:type_1brownian_BSVIE}
Y(t) =\Phi(t) + \int_t^{T} f(t,s, Y(s), Z(t,s)) \dd s - \int_t^{T} Z(t,s) \dd W_s.
\end{equation}
The extension to $L^p$-solution ($1<p<2$) for the Type-I BSVIE \eqref{eq:type_1brownian_BSVIE} has been done in \cite{wang:12}. In the four papers \cite{lin:02,wang:12,yong:06,yong:08}, the filtration is generated by the Brownian motion $W$. In \cite{wang:zhan:07}, the authors introduced the jump component $\tpi$. In the filtration generated by $W$ and the Poisson random measure $\tpi$, they consider:
\begin{eqnarray*}
Y(t) & = & \Phi(t) + \int_t^{T} f(t,s, Y(s), Z(t,s) ,U(t,s)) \dd s \\ 
&- &\int_t^{T} Z(t,s) \dd W_s  - \int_t^{T}  \int_{\bR^m} U(t,s,x) \tpi(\dd s,\dd x)
\end{eqnarray*}
and prove existence and uniqueness of the solution in the $L^2$-setting. The result has been extended in \cite{hu:okse:19,over:roed:18,ren:10} (see also \cite{lu:11} for the L\'evy case). 

BSVIEs have been also studied in the Hilbert case \cite{anh:yong:06}, with additive perturbation in the Brownian setting \cite{djor:jank:13,djor:jank:15}, in the quadratic case \cite{wang:sun:yong:19}, as a probabilistic representation (nonlinear Feynman-Kac formula) for PDEs \cite{wang:yong:19}. Their use for optimal control problem is well known since the seminal paper \cite{yong:06}; see for example the recent paper \cite{lin:yong:20}. Let us also mention the survey \cite{yong:13}.

\bigskip
Combining all these papers, here we want to deal with a BSVIE of the following type\footnote{In the whole paper, $\dd M(t,s)$ is the integration w.r.t. the second time parameter $s$ where $t$ is fixed. In particular $\int_u^v \dd M(t,s) = M(t,v) - M(t,u)$ is the increment of $M(t,\cdot)$ between the time $u$ and $v$.}:
\begin{eqnarray} \nonumber
Y(t) &= &\Phi(t) + \int_t^{T} f(t,s, Y(s), Z(t,s) ,Z(s,t),U(t,s),U(s,t)) \dd B_s  \\ \label{eq:general_BSVIE}
&-& \int_t^{T} Z(t,s) \dd X^\circ_s -  \int_t^{T}  \int_{\bR^m} U(t,s,x) \tpi^\natural(\dd x,\dd s) -\int_t^{T} \dd M(t,s).
\end{eqnarray}
The unknown processes are the quadruplet $(Y,Z,U,M)$ valued in $\bR^{d+(d\times k)+d+d}$ such that $Y(\cdot)$ is $\bF$-adapted, and for (almost) all $t\in [0,T]$, $(Z(t,\cdot),U(t,\cdot))$ are such that the stochastic integrals are well-defined and $M(t,\cdot)$ is a martingale. This BSVIE is called {\it of Type-II}. We also consider the Type-I BSVIE:
\begin{eqnarray} \nonumber 
Y(t) &= &\Phi(t) + \int_t^{T} f(t,s, Y(s), Z(t,s) ,U(t,s)) \dd B_s \\  \label{eq:general_BSVIE_type_I}
&-& \int_t^T Z(t,s) \dd X^\circ_s - \int_t^T \int_{\bR^m} U(t,s,x) \tpi^\natural(\dd x,\dd s) -\int_t^T \dd M(t,s).
\end{eqnarray}
The BSDE \eqref{eq:poss_BSDE} becomes a particular case of the preceding BSVIE.

%

\subsubsection*{Main contributions}

Let us outline the main contributions of our paper compared to the existing literature. First we prove {\it existence and uniqueness of the solution of the Type-I BSVIE} \eqref{eq:general_BSVIE_type_I} in the $L^2$-setting (Theorem \ref{thm:type_I_general_BSVIE} and Proposition \ref{prop:type_I_BSVIE_without_Y}). This first result generalizes the preceding results (of course only some of them) since we only assume that the filtration is complete and right-continuous. This is the reason of the presence of the c\`adl\`ag process $B$ and of the additional martingale term $M$ in \eqref{eq:general_BSVIE_type_I}. Our proof is based on a fixed point argument in the suitable space.

In \cite{papa:poss:sapl:18}, the authors consider two different types of BSDEs: the equation \eqref{eq:poss_BSDE} and the following one:
\begin{eqnarray*}  \nonumber
Y(t) &= & \xi + \int_t^{T} f(s, Y(s-), Z(s),U(s) ) \dd B_s \\ 
& -&  \int_t^{T} Z(s) \dd X^\circ_s - \int_t^T \int_{\bR^m} U(s,x) \tpi^\natural(\dd x,\dd s) - \int_t^T \dd M(s).
\end{eqnarray*}
The only difference concerns the dependence of $f$ w.r.t. $Y$. Since $B$ is assumed to be random and c\`adl\`ag, both cases are not equivalent. However the resolution's method in the second case is not adapted for BSVIEs. This case is left for further research. 

A real issue for BSVIE concerns the comparison principle. In the BSDE theory, comparison principle holds under quite general conditions (see e.g. \cite{krus:popi:14,krus:popi:17, papa:poss:sapl:18,pard:rasc:14}). Roughly speaking, the comparison result is proved by a linearization procedure and by an explicit form for the solution of a linear BSDE. However in the setting of \cite{papa:poss:sapl:18}, the comparison is more delicate to handle because of the jumps of the process $B$. For BSVIE, these arguments fail and comparison is a difficult problem. The paper \cite{wang:yong:15} is the most relevant paper on this topic. It provides comparison results and gives several counter-examples where comparison principle fails. Of course all their counter-examples are still valid in our case; thereby we do not have intrinsically better results. 
Our second contribution is to extend the {\it comparison results for the Type-I BSVIE} \eqref{eq:general_BSVIE_type_I} (Propositions \ref{prop:comp_BSVIE_1} and \ref{prop:comp_BSVIE_2}). Somehow we show that the additional martingale terms do not destroy the comparison result. Note that we have to take into account that the driver $f(s,Y(s),Z(s),U(s))$ is a priori optional (compared to $f(s,Y(s-),Z(s),U(s))$). Thus the linearization procedure should be handled carefully.  
\bigskip

In the BSDE theory, many papers deal with $L^p$-solution (instead of the square integrability condition on the data); see in particular \cite{bria:dely:hu:03,krus:popi:14,krus:popi:17,pard:rasc:14} which deal with $L^p$-solution for BSDE. To our best knowledge, such an extension does not exist for the general BSDE \eqref{eq:poss_BSDE}. This is the reason why we consider in Section \ref{sect:Ito_setting_BSVIE} the It\^o setting where $B_t = t$. This denomination comes from \cite{ait:jaco:15,ait:jaco:18}. Thus $X^\circ$ is the Brownian motion $W$ and $\tpi^\natural = \tpi$ is a Poisson random measure with intensity measure $\mu$. The Type-I BSVIE \eqref{eq:general_BSVIE_type_I} becomes
\begin{eqnarray} \nonumber
Y(t) &= &\Phi(t) + \int_t^{T} f(t,s, Y(s), Z(t,s) ,U(t,s)) \dd s - \int_t^{T} Z(t,s) \dd W_s \\ \label{eq:second_special_BSVIE}
&& \qquad - \int_t^{T}  \int_{\bR^m}  U(t,s,x) \tpi(\dd s,\dd x) -\int_t^{T} \dd M(t,s). 
\end{eqnarray}
For this BSVIE, we provide {\it existence and uniqueness of M-solutions in $L^p$-space} of \eqref{eq:second_special_BSVIE} (Theorem \ref{thm:Hp_M_solution}). To the best of our knowledge, there is no existence and uniqueness result for BSVIEs with $L^p$ coefficients in a general filtration. 

Another contribution is the study of the regularity of the map $t\mapsto Y(t)$. For the solution of the BSDE \eqref{eq:poss_BSDE}, from the c\`adl\`ag regularity of all martingales, $Y$ inherits the same time regularity. For BSVIE, we only require that the paths of $Y$ are in $\bL^2(0,T)$ (or in $\bL^p(0,T)$). Essentially because we assume that $\Phi$ and $t\mapsto f(t,\ldots)$ are also only in $\bL^2(0,T)$. In \cite{yong:08}, it is proved that under weak regularity conditions on the data, then the solution of  \eqref{eq:brownian_BSVIE} $t\mapsto Y(t)$ is continuous from $[0,T]$ to $\bL^2(\Omega)$. Let us stress that the Malliavin calculus is used to control the $Z(s,t)$ term in the generator. Similarly we show that the paths $t \in[0,T] \mapsto Y(t) \in \bL^p(\Omega)$ of the solution of the BSVIE \eqref{eq:second_special_BSVIE} are c\`adl\`ag if roughly speaking $\Phi$ and $t\mapsto f(t,\ldots)$ satisfy the same property. However this first property does not give a.s. continuity of the paths of $Y$ in general. Getting an almost sure continuity is a more challenging issue and is proved in \cite{wang:zhan:07} for the BSVIE \eqref{eq:brownian_BSVIE} when $f$ does not depend on $Z(s,t)$, assuming a H\"older continuity property of $t\mapsto f(t,\ldots)$ for a constant $\Phi(t) = \xi$. To understand the difficulty, let us recall that if $f$ does not depend on $y$, the solution $Y$ of the BSVIE \eqref{eq:second_special_BSVIE} is obtained by the formula: $Y(t) = \lambda(t,t)$ where $\lambda(t,\cdot)$ is the solution of the related BSDE parametrized by $t$. In the Brownian setting, a.s. $s\mapsto \lambda(t,s)$ is continuous. Using the Kolmogorov continuity criterion, the authors show that $(t,s) \mapsto \lambda(t,s)$ is bi-continuous, which leads to a continuous version of $Y$. To our best knowledge, there is not equivalent result to the Kolmogorov criterion for c\`adl\`ag paths. Hence we assume that the free term $\Phi$ and the generator $f$ are H\"older continuous. Thus we sketch the arguments of \cite{wang:zhan:07} to obtain that {\it a.s. the paths of $Y$ are c\`adl\`ag} (Theorem \ref{thm:as_regularity_p_2}) if we know that the data $\Phi$ and $f$ are H\"older continuous w.r.t. $t$, meaning that the jumps only come from the martingale parts in the BSVIE. Relaxing the regularity of the data is still an open question.

\bigskip

Up to now, we only provide results for Type-I BSVIEs \eqref{eq:general_BSVIE_type_I} or \eqref{eq:second_special_BSVIE}. In the last part of the paper, we expand some results to the Type-II BSVIEs \eqref{eq:general_BSVIE} or in the It\^o setting: 
\begin{eqnarray} \nonumber
Y(t) &= &\Phi(t) + \int_t^{T} f(t,s, Y(s), Z(t,s) ,Z(s,t),U(t,s),U(s,t)) \dd s  \\ \label{eq:Ito_general_BSVIE}
&-& \int_t^{T} Z(t,s) \dd W_s -  \int_t^{T}  \int_{\bR^m} U(t,s,x) \tpi(\dd s,\dd x) -\int_t^{T} \dd M(t,s).
\end{eqnarray}
As explained in the introduction of \cite{yong:08}, for Type-II BSVIEs, the notion of M-solution is crucial to ensure uniqueness of the solution. To define the terms $Z$ and $U$ on the set $\Delta(R,T)=\{ (t,s) \in [R,T]^2, \ R \leq s \leq t \leq T\}$, the martingale representation is used in \cite{yong:08}: for almost every $t \in [R,T]$, 
\begin{equation} \label{eq:brownian_def_mart_Delta_S_T}
Y(t) = \bE \left[ Y(t)  \bigg| \cF_R \right]  + \int_R^t Z(t,s) \dd W_s. 
\end{equation}
The existence is justified since $Y \in \bL^2(\Omega \times [R,T])$. However for the BSVIE \eqref{eq:general_BSVIE} we only except
$$\bE \left[  \int_S^T e^{\beta A_t}|Y(t)|^2 \alpha_t^2 \dd B_t \right] < +\infty.$$
And $A = \int \alpha_s^2 \dd B_s$ and $B$ are random processes. To get around this issue, we could define
$$\mathcal Y(t) = \int_S^t Y(s) \alpha_s^2 \dd B_s.$$
The Cauchy-Schwarz inequality yields to
\begin{eqnarray*}
|\mathcal Y(t)|^2 & \leq & \frac{1}{\delta} e^{-\delta A_t}  \int_t^T e^{\delta A_s}|Y(s)|^2 \alpha_s^2 \dd B_s.
\end{eqnarray*}
Hence for almost all $t \in [R,T]$, $\mathcal Y(t)$ is in $\bL^2(\Omega)$, thus we could define
\begin{equation} \label{eq:def_mart_Delta_S_T}
\mathcal Y(t)  = \bE \left[  \mathcal Y(t)  \bigg| \cF_R \right] +  \int_R^t Z(t,s) \dd X^\circ_s + \int_R^t \int_{\bR^m} U(t,s,x) \dd \tpi (\dd s ,\dd x)+ \int_R^t \dd M(t,s).
\end{equation}
Integrating the relation \eqref{eq:brownian_def_mart_Delta_S_T} between $R$ and $t$ and using Fubini's theorem leads to
\begin{eqnarray*}
\mathcal Y(t) = \int_R^t Y(r) \dd r & = & \bE \left[  \int_R^t Y(r) \dd r  \bigg| \cF_R \right]  + \int_R^t \int_S^r Z(r,u) \dd W_u \dd r \\
& = &  \bE \left[  \int_R^t Y(r) \dd r  \bigg| \cF_R \right]  + \int_R^t \left( \int_u^t Z(r,u)\dd r \right) \dd W_u \\
& = &  \bE \left[  \mathcal Y(t)  \bigg| \cF_R \right]  + \int_R^t  Z^\sharp(t,u) \dd W_u .
\end{eqnarray*}
Moreover 
\begin{eqnarray*}
\bE \int_R^T \left| \int_R^t Z^\sharp(t,s) \dd W_s \right|^2 \dd t & = & \bE \int_R^T \int_R^t |Z^\sharp(t,s)|^2 \dd s \dd t \\
& \leq & (T-R)^2   \bE  \int_R^T \int_R^r  |Z(r,s)|^2  \dd s \dd r .
\end{eqnarray*}
In other words a M-solution (in the sense of \cite{yong:08}) provides a ``integrated martingale''-solution.

Coming back to \eqref{eq:def_mart_Delta_S_T}, let us try to control the martingale part $\mart$ on $\Delta(R,T)$. We can easily obtain: for any $0 < \delta<  \beta$
$$\bE \int_S^T e^{\delta A_t} \left| \int_S^t \dd \mart(t,r) \right|^2 \dd A_t
\leq 2 \frac{e^{(\beta-\delta)\mathfrak f}}{\delta(\beta-\delta)} \bE \left[  \int_S^T  e^{\beta A_u}  |Y(u)|^2 \dd A_u \right].$$
But this weak norm on $\mart$ is not sufficient to control $Z$ and $U$ in the generator of the BSVIE \eqref{eq:general_BSVIE}. Since the process $A$ is supposed to be only predictable, we cannot claim that 
$$\bE \int_S^T e^{\delta A_t} \left| \int_S^t \dd \mart(t,r) \right|^2 \dd A_t = \bE \int_S^T e^{\delta A_t}  \int_S^t \dd \trace \langle \mart(t,\cdot)\rangle_r  \dd A_t.$$
It leads to a major issue.  Moreover the trick used in \cite{papa:poss:sapl:18}: 
\begin{eqnarray*}
\int_S^t e^{\delta A_r} \dd \trace \langle \mart(t,\cdot)\rangle_r & \leq & \delta \int_S^t e^{\delta A_u} \int_u^t \dd \trace \langle \mart(t,\cdot)\rangle_r \dd A_u \\
&+& e^{\delta A_S} \left( \trace \langle \mart(\cdot,t)\rangle_t - \trace \langle \mart(\cdot,t)\rangle_S \right) .
\end{eqnarray*}
is useless here since $e^{\delta A_u}$ is $\cF_u$-measurable and not $\cF_S$-measurable. {\it Due to this reason, we aren't able to address an existence and uniqueness result for the BSVIE \eqref{eq:general_BSVIE} in this general setting, but only for deterministic processes $B$ and $A$} (Theorem \ref{thm:type_II_general_BSVIE}). 

\bigskip
For Type-I BSVIEs in the It\^o case, $\mathfrak S^p$-solutions were considered with $p>1$. For Type-II BSVIEs, we were not able to provide the same extension. The reason can be understood just by considering the term $Z$. We prove that $Z \in  \bL^p(0,T;\bH^p(0,T))$, that is
$$\bE \int_0^T \left[ \int_0^T |Z(t,r)  |^2 \dd r \right]^{\frac{p}{2}} \dd t < + \infty.$$
 Indeed since $Z$ is integrated w.r.t. the Brownian motion $W$, the previous norm is natural.  But it is symmetric w.r.t. $(t,s)$ only for $p=2$. The two time variables $t$ and $r$ don't play the same role and the integrability property is not the same w.r.t. $t$ or w.r.t. $r$, except if $p=2$. Thereby in the BSVIE \eqref{eq:general_BSVIE}, we can use both $Z(t,s)$ and $Z(s,t)$ if $p=2$ (Proposition \ref{thm:H2_M_solution}). Let us also mention that in the case where the generator depends on the stochastic integrand w.r.t.\ a Poisson random measure, the case when $p < 2$ has to be handled carefully. Indeed in this case Burkholder-Davis-Gundy inequality with $p/2<1$ does not apply and the $L^{p/2}$-norm of the predictable projection cannot be controlled by the $L^{p/2}$-norm of the quadratic variation (see \cite{leng:lepi:prat:80} and the discussion in \cite{krus:popi:17}). The extension to $p\neq 2$ seems difficult to prove and is left for further research. We also point out that $p\geq 2$ implies that $ \bL^p(S,T;\bH^p(R,T)) \subset  \bL^2(S,T;\bH^2(R,T))$ and $ \bL^p(S,T;\bL^p_\pi(R,T)) \subset \bL^2(S,T;\bL^2_\pi(R,T))$:
$$\bE \int_S^T \left(  \int_R^T |Z(t,r)|^2 \dd r \right) \dd t \leq C \left( \bE  \int_S^T \left[ \left( \int_R^T |Z(t,r)|^2 \dd r \right)^{\frac{p}{2}} \right] \dd t \right)^{\frac{2}{p}}$$
and
$$\bE\int_S^T \|U(t,\cdot)|^2_{\bL^2_\pi(R,T)} \ \dd t  \leq C\left[  \bE \int_S^T \left( \|U(t,\cdot)|^2_{\bL^2_\pi(R,T)}\right)^{\frac{p}{2}} \dd t\right]^{\frac{2}{p}}.$$
For $1<p<2$, this property fails. In the proof of Theorem \ref{thm:Hp_M_solution} we point out the generated troubles if $p\neq 2$. 

Finally we prove a duality principle for the BSVIE \eqref{eq:Ito_general_BSVIE} provided we know that the solution $X$ of the forward SVIE is itself c\`adl\`ag (see \cite{prot:85}). Note the importance of the time regularity here. This result is the first step for comparison principle for this kind of BSVIEs.

\subsubsection*{Decomposition of the paper}

The paper is decomposed as follows. In the first section, we give the mathematical setting and recall some results concerning the existence and uniqueness of the solution of a BSDE \eqref{eq:poss_BSDE}. In the second part, we prove existence, uniqueness and comparison principle of the adapted solution of the Type-I BSVIE \eqref{eq:general_BSVIE_type_I}. The proof is essentially based on a fixed point argument as in \cite{papa:poss:sapl:18,yong:06}.  The third section is devoted to some additional properties for the Type-I BSVIE \eqref{eq:second_special_BSVIE}: $L^p$-solution ($p>1$) and time regularity.
In the fourth section, we prove that if the process $B$ is deterministic, the Type-II BSVIE \eqref{eq:general_BSVIE} has a unique M-solution, again by some fixed point argument. The existence and uniqueness of the M-solution for the It\^o case \eqref{eq:Ito_general_BSVIE} is deduced. However we provide a different proof, adapted from the ideas of \cite{yong:08}, and we point out the issues to extend this result to $L^p$-solution ($p\neq 2$). The appendix contains some proofs and auxiliary results.

\subsubsection*{Remaining open questions}

We address here some open problems. First of all, the existence and uniqueness of a M-solution for the BSVIE \eqref{eq:general_BSVIE} is not proved yet, except for a deterministic characteristic $B$. The second question concerns the comparison principle, at least for Type-I BSVIE \eqref{eq:general_BSVIE_type_I}, when $B$ is not continuous. At least the $L^p$-theory for the BSDE \eqref{eq:poss_BSDE} and thus for the BSVIEs \eqref{eq:general_BSVIE} or \eqref{eq:general_BSVIE_type_I} is a natural question.

\section{Setting and notations}  \label{sect:setting_not}

On $\bR^d$, $|.|$ denotes the Euclidean norm and $\bR^{d\times k}$ is identified with the space of real matrices with $d$ rows and $k$ columns. If $z \in  \bR^{d\times k}$, we have $|z|^2 = \trace(zz^*)$. For any metric space $G$, $\cB(G)$ is the Borel $\sigma$-field. 

Our probabilistic setting is the same as the one of Papapantoleon et al. \cite{papa:poss:sapl:18}. We recall the main notations but the details can be found in this paper, especially in Section 2, and are left to the reader. Along this paper, we consider a filtered probability space $(\Omega,\cF,\bP,\bF = (\cF_t)_{t\geq 0})$ such that it is a complete stochastic basis in the sense of Jacod and Shiryaev \cite{jaco:shir:03}. Without loss of generality we suppose that all semimartingales are c\`adl\`ag\footnote{French acronym for right continuous with left limit}, that is they have a.s. right continuous paths with left limits. 
The space $\cH^2(\bR^p)$ denotes the set of $\bR^p$-valued, square-integrable $\bF$-martingales and $\cH^{2,d}(\bR^p)$ is the subspace of $\cH^2(\bR^d)$ consisting of purely discontinuous square-integrable martingales. $\cP$ is the predictable $\sigma$-field on $\Omega \times [0,T]$ and $\tP=\cP \otimes \mathcal{B}(\bR^n)$. On $\tOm = \Omega \times [0,T] \times \bR^n$, a function that is $\tP$-measurable, is called predictable. 

The required notions on stochastic integrals are recalled in \cite[Section 2.2]{papa:poss:sapl:18}. In particular for $X \in \cH^2(\bR^p)$, if
$$\langle X \rangle  = \int_{(0,\cdot]} \dfrac{\dd\langle X \rangle_s }{\dd B_s} \dd B_s $$
for $B$ predictable non-decreasing and c\`adl\`ag,  they defined the stochastic integral of $Z$ w.r.t $X$, denoted $Z\cdot X$ or $\int_0 Z_s \dd X_s$, on the space $\bH^2(X)$ of the predictable processes $Z : \Omega \times \bR_+ \to \bR^{d\times k}$ such that 
$$\bE \int_0^\infty \trace \left( Z_t \dfrac{\dd\langle X \rangle_t }{\dd B_t} Z_t^{\top} \right) \dd B_t < +\infty.$$
Moreover $\pi^{X}$ is the $\bF$-optional integer-valued random measure on $\bR_+\times \bR^m$ defined by 
$$\pi^X(\omega;\dd t,\dd x) = \sum_{s> 0} \mathbf 1_{\Delta X_s(\omega)\neq 0} \delta_{(s,\Delta X_s(\omega))}(\dd t,\dd x).$$
And the stochastic integral of $U$ w.r.t. $\pi^X$, denoted $U\star \pi^X$, the compensator $\nu^X$ of $\pi^X$, the compensated integer-valued random measure $\tpi^X$ and the stochastic integral of $U$ w.r.t. $\tpi^X$, denoted $U\star \tpi^X$, are also defined on the space $\bH^2(X)$ of $\bF$-predictable processes $U :  \Omega \times \bR_+ \times \bR^n \to \bR^{d}$ such that 
$$\bE \left[\int_0^\infty \dd \trace \left[ \langle U\star \tpi^X\rangle_t\right] \right]< +\infty.$$
Recall that here we summarize all definitions; all details can be found in \cite{jaco:shir:03,papa:poss:sapl:18} for the interested reader.

{\bf In the rest of the paper, we fix:}
\begin{itemize}
\item An $\bR^{k+m}$-valued, $\cF \otimes \cB(\bR_+)$-measurable process $\overline X = (X^\circ,X^\natural)$ such that 
$$\overline X^T \in \cH^2(\bR^k) \times \cH^{2,d}(\bR^m)$$
with 
$$M_{\pi^{(X^\natural)^T}} \left[ \Delta ((X^\circ)^T) \big| \tP\right] = 0.$$ 

Here $X^T$ is the process stopped at time $T$: $X^T_t = X_{t\wedge T}$, $t\geq 0$.
$M_\pi\left[ \cdot  \big| \tP\right]$ is the conditional $\bF$-predictable projection on $\pi$ (see \cite[Definition 2.1]{papa:poss:sapl:18}). To simplify the notations, $\pi^\natural = \pi^{(X^\natural)^T}$, $\tpi^\natural =  \tpi^{(X^\natural)^T}$. 
Recall that $\nu^\natural$ is the compensator of $\pi^\natural$, that is the $\bF$-predictable random measure on $\bR_+ \times \bR^m$ for which $\bE [ U\star \mu^\natural] = \bE [U \star \nu^\natural]$ for every non-negative $\bF$-predictable function $U$. 
\item A non-decreasing predictable and c\`adl\`ag $B$ such that each component of $\langle X^\circ \rangle$ is absolutely continuous with respect to $B$ and the disintegration property given $B$ holds for the compensator $\nu^\natural$, that is there exists a transition kernel $K$ such that 
\begin{equation} \label{eq:decomp_kernel_C}
\nu^\natural (\omega;\dd t,\dd x) = K_t(\omega;\dd x) \dd B_t
\end{equation}
(see \cite[Lemma 2.9]{papa:poss:sapl:18}). This property is called Assumption ($\mathbb C$) in \cite{papa:poss:sapl:18}. \item $b$ is $\bF$-predictable process defined in \cite[Remark 2.11]{papa:poss:sapl:18} by:
$$b= \left( \dfrac{\dd \langle X^\circ \rangle}{\dd B} \right)^{\frac{1}{2}}.$$ 
\end{itemize}
The process $B$ is the first component of the characteristics of the semimartingale $\overline{X}$ (see \cite[Definition II.2.6]{jaco:shir:03}). 
 \begin{Exp} \label{exp:brownian_poisson_setting}
In Section \ref{sect:Ito_setting_BSVIE}, $X^\circ$ is the $k$-dimensional Brownian motion $W$, $B_t = t$, $b=1$ and $\tpi^\natural$ is the compensated Poisson random measure $\tpi$, with the compensator $\nu(\dd t,\dd x)=\dd t \mu(\dd x)$, where $\mu$ is $\sigma$-finite on $\bR^m$ such that
$$\int_{\bR^m} (1\wedge |x|^2) \mu(\dd x) <+\infty.$$
In this particular case, $K_t(\omega;\dd x)=\mu(\dd x)$. 
These spaces are classically used for BSDEs with Poisson jumps (see among others \cite{delo:13,krus:popi:14}). 
 
\end{Exp}
\begin{Exp}
The preceding example can be generalized to the case where the compensator of $\pi$ is random and equivalent to the measure $ \dd t \otimes \mu (\omega, \dd x)$ with a bounded density for example (see the introduction of \cite{bech:06}). 
\end{Exp}

\begin{Exp}
In \cite[Section 3.3]{papa:poss:sapl:18}, the authors cite the counterexample of \cite{conf:fuhr:jaco:16}. They also provide two other examples just after their remark 3.19. 
\end{Exp}

The notion of orthogonal decomposition plays a central role here. Inspired by \cite{jaco:shir:03} and defined in \cite[Definition 2.2]{papa:poss:sapl:18}, if $Y \in \cH^2(\bR^d)$, then the decomposition
$$Y = Y_0 + Z \cdot X^\circ + U \star \tpi^\natural + N$$
is called {\it orthogonal decomposition of $Y$} w.r.t. $\overline{X} = (X^\circ,X^\natural)$ if 
\begin{enumerate}
\item $Z \in \bH^2(X^\circ)$ and $U \in \bH^2(\pi^\natural)$,
\item the martingales $Z\cdot X^\circ$ and $U\star \tpi^\natural$ are orthogonal,
\item $N \in \cH^2(\bR^d)$ with $\langle N,X^\circ\rangle = 0$ and $M_{\pi^\natural}\left[ \Delta N \big| \tP\right]=0$.
\end{enumerate}
The statements \cite[Propositions 2.5 and 2.6, Corollary 2.7]{papa:poss:sapl:18} give us the existence and the uniqueness of such orthogonal decomposition. 

\subsection{Setting and known results for the BSDE \eqref{eq:poss_BSDE}} \label{ssect:poss_BSDE}

We consider the BSDE \eqref{eq:poss_BSDE}
\begin{equation*}
Y_t = \xi + \int_t^T f(s,Y_s,Z_s,U_s) \dd B_s -\int_t^T Z_s \dd X^\circ_s - \int_t^T \int_{\bR^m} U_s(x) \tpi^\natural(\dd s,\dd x) - \int_t^T \dd M_s.
\end{equation*}
For some $\beta \in \bR$ and some c\`adl\`ag increasing and measurable process $A$, let us define the spaces used to obtain the solution of the BSDE \eqref{eq:poss_BSDE}. For ease of notation, the dependence on $A$ is suppressed. 
\begin{eqnarray*}
\bL^2_{\beta,\cF_T} & = & \left\{ \xi , \ \bR^d\mbox{-valued,} \ \cF_T\mbox{-measurable,} \ \|\xi\|^2_{\bL^2_\beta} = \bE \left[ e^{\beta A_T} |\xi|^2 \right]< +\infty\right\}, \\
\cH^2_\beta(R,S) & = & \left\{ M \in \cH^2,\ \|M\|^2_{\cH^2_\beta(R,S)} =  \bE \left[ \int_R^S e^{\beta A_t} \dd\trace(\langle M \rangle_t)  \right]< +\infty\right\}, \\
\bH^2_\beta(R,S) & = & \left\{ \phi \  \bR^d\mbox{-valued, }\bF\mbox{-optional semimartingale with c\`adl\`ag paths and } \right.\\
&&\quad  \ \left. \|\phi\|^2_{\bH^2_\beta(R,S)} =  \bE \left[ \int_R^S e^{\beta A_t} |\phi(t)|^2 \dd B_t  \right]< +\infty\right\}, \\
\bH^{2,\circ}_\beta(R,S) & = & \left\{ Z \in \bH^2(X^\circ),\ \|Z\|^2_{\bH^{2,\circ}_\beta(R,S)} =  \bE \left[ \int_R^S e^{\beta A_t} \dd\trace(\langle Z\cdot X^\circ \rangle_t)  \right]< +\infty\right\}, \\
\bH^{2,\natural}_\beta(R,S) & = & \left\{ U \in \bH^2(X^\natural),\ \|U\|^2_{\bH^{2,\natural}_\beta(R,S)} =  \bE \left[ \int_R^S e^{\beta A_t} \dd\trace(\langle U\star \tpi^\natural \rangle_t)  \right]< +\infty\right\}, \\
\cH^{2,\perp}_\beta(R,S) & = & \left\{ M \in \cH^2(\overline{X}^\perp),\ \|M\|^2_{\cH^{2,\perp}_\beta(R,S)} =  \bE \left[ \int_R^S e^{\beta A_t} \dd\trace(\langle M \rangle_t)  \right]< +\infty\right\}, \\
\end{eqnarray*}
Let us introduce some additional notations. For $U$ a $\bF$-predictable function, we define
$$\widehat K_t(U_t(\omega;\cdot))(\omega)=\int_{\bR^m} U_t(\omega;x)K_t(\omega,\dd x),$$
where $K$ is the kernel defined by \eqref{eq:decomp_kernel_C}. Hence
$$\widehat U^\natural_t(\omega) = \int_{\bR^m} U_t(\omega,x)\nu^\natural(w;\{t\}\times \dd x) = \widehat K_t(U_t(\omega;\cdot))(\omega)\Delta B_t(\omega).$$
And 
$$\zeta^\natural_t = \int_{\bR^m} \nu^\natural(w;\{t\}\times \dd x) .$$
The justification of the definition of the norms on the preceding spaces is given by \cite[Lemma 2.12]{papa:poss:sapl:18}.
\begin{Lemma} \label{lem:equiv_norm_poss}
Let $(Z,U) \in  \bH^{2,\circ}_\beta(R,S) \times \bH^{2,\natural}_\beta(R,S)$. Then
$$ \|Z\|^2_{\bH^{2,\circ}_\beta(R,S)} = \bE \left[ \int_R^S e^{\beta A_t} \left\| b_t Z_t\right\|^2 \dd B_t  \right],$$
and
$$ \|U\|^2_{\bH^{2,\natural}_\beta(R,S)} =  \bE \left[ \int_R^S e^{\beta A_t}  \vertiii{ U_t(\cdot) }^2_t \dd B_t  \right],$$
where for every $(t,\omega) \in \bR_+\times \Omega$
$$\left( \vertiii{ U_t(\omega;\cdot) }^2_t(\omega)\right)^2=\widehat K^{\overline{X}}_t(|U_t(\omega;\cdot) - \widehat U_t^\natural(\omega)|^2)(\omega) + (1-\zeta_t^\natural) \Delta B_t(\omega)|\widehat K^{\overline{X}}_t(U_t(\omega;\cdot))(\omega)|^2 \geq 0.$$
Furthermore
$$\| Z\cdot X^\circ + U\star \tpi^\natural\|^2_{\cH^2_\beta} =  \|Z\|^2_{\bH^{2,\circ}_\beta} + \|U\|^2_{\bH^{2,\natural}_\beta}.$$
\end{Lemma}
Finally we define 
\begin{equation} \label{eq:def_space_mathfrak_H}
\mathfrak H = \left\{ U:[0,T]\times \Omega\times \bR^m \to \bR^d, \ U_t(\omega;\cdot) \in \mathfrak H_{t,\omega} \ \mbox{for } \dd B\otimes\dd \bP- \mbox{a.e. } (t,\omega)\right\}
\end{equation}
where
 $$ \mathfrak H_{t,\omega} =  \left\{ u: \bR^m \to \bR^d, \ \vertiii{u(\cdot)}_t < +\infty \right\}.$$

\begin{Rem} 
In the setting of the example \ref{exp:brownian_poisson_setting}, $\widehat U^\natural \equiv 0$ and 
\begin{eqnarray*}
\vertiii{u(\cdot)} &=& \int_{\bR^m} |u(x)|^2 \mu(\dd x) = \|u\|^2_{\bL^2_\mu(\bR^m)} ,\\
 \|U\|^2_{\bH^{2,\natural}_\beta(R,S)} &= & \bE \left[ \int_R^S e^{\beta A_t}  \int_{\bR^m} |U_t(x)|^2 \mu(\dd x) \dd t  \right],\\
 \|Z\|^2_{\bH^{2,\circ}_\beta(R,S)} &=& \bE \left[ \int_R^S e^{\beta A_t} \left\|  Z_t\right\|^2 \dd t  \right].
 \end{eqnarray*}
 Then we can define the generator on $\bL^2_\mu(\bR^m)$ instead of $\mathfrak H$ in Condition {\rm \ref{H2}}.
\end{Rem}
 
We define the product space 
$$\bD^2_\beta(0,T)= \bH^2_\beta(0,T) \times \bH^{2,\circ}_\beta(0,T) \times \bH^{2,\natural}_\beta(0,T) \times \cH^{2,\perp}_\beta(0,T)$$ 
and assuming that $A = \alpha^2 \cdot B$ for a measurable process $\alpha : \Omega \times \bR_+ \to \bR_+$, we search a solution $(Y,Z,U,M)$ of the BSDE \eqref{eq:poss_BSDE} such that 
$$(\alpha Y, Z, U, M) \in  \bD^2_\beta(0,T).$$

Now let us describe the conditions on the parameters $(\xi,f)$ for the BSDE \eqref{eq:poss_BSDE}.
\begin{enumerate}[label=\textbf{(F\arabic*)}]
\item \label{F1} The terminal condition $\xi$ belongs to $\bL^2_{\delta,\cF_T}$ for some $\delta > 0$. 
\item \label{F2} The generator $f : \Omega \times \bR_+ \times \bR^d \times \bR^{d\times m} \times \mathfrak H \to \bR^d $ is such that for any $(y,z,u) \in \bR^d \times \bR^{d\times m} \times \mathfrak H $, the map $(\omega,t)\mapsto f(\omega,t,y,z,u)$ is $\cF_t \otimes \cB([0,t])$-measurable. Moreover there exist
$$\varpi:(\Omega\times \bR_+, \cP) \to \bR_+, \qquad \vartheta = (\theta^\circ,\theta^\natural) : (\Omega\times \bR_+, \cP) \to (\bR_+)^2$$
such that for $\dd B\otimes \dd \bP$-a.e. $(t,\omega)$
\begin{eqnarray*}
&& |f(\omega,t,y,z,u_t(\omega;\cdot)) - f(\omega,t,y',z',u'_t(\omega;\cdot)) |^2 \\
&&\quad \leq \varpi_t(\omega)|y-y'|^2 + \theta^\circ_t(\omega)\|c_t(\omega)(z-z')\|^2 + \theta^\natural_t(\omega)(\vertiii{u_t(\omega;\cdot)-u'_t(\omega;\cdot)}_t(\omega))^2
\end{eqnarray*}
\item \label{F3} Let $\alpha^2 = \max(\sqrt{\varpi},\theta^\circ,\theta^\natural)>0$ and define the increasing, $\bF$-predictable and c\`adl\`ag process
\begin{equation}\label{eq:def_A}
A_t = \int_0^t \alpha_s^2 \dd B_s.
\end{equation}
There exists $\mathfrak f > 0$ such that 
$$\Delta A_t(\omega) \leq \mathfrak f, \ \mbox{for }\dd B\otimes \dd \bP-\ \mbox{a.e. } (t,\omega).$$ 
\item \label{F4} For the same $\delta$ as in \ref{F1}, 
$$\bE \left[ \int_0^T e^{\delta A_t} \frac{|f(t,0,0,\mathbf 0)|^2}{\alpha_t^2} \dd B_t \right] < +\infty.$$
$\mathbf 0$ denotes the null application from $\bR^m$ to $\bR$. 
\end{enumerate}
Let us recall the existence and uniqueness result of \cite{papa:poss:sapl:18}. 
\begin{Thm}{\cite[Theorem 3.5, Corollary 3.6]{papa:poss:sapl:18}} \label{thm:exis_sol_poss}
Assume that the paramater $(\xi,f)$ verifies all conditions {\rm \ref{F1}} to {\rm \ref{F4}}, and suppose that 
\begin{equation}\label{eq:cond_exist_beta_poss}
\kappa^\mathfrak f(\delta)=\frac{9}{\delta} + \dfrac{\mathfrak f^2(2+9\delta)}{\sqrt{\delta^2 \mathfrak f^2+4}-2} \exp \left( \dfrac{\delta \mathfrak f + 2 -\sqrt{\delta^2 \mathfrak f^2+4} }{2} \right) < \dfrac{1}{2}.
\end{equation}
Then the BSDE \eqref{eq:poss_BSDE} has a unique solution $(Y,Z,U,M)$ such that $(\alpha Y, Z, U,M) \in \bD_\delta^2(0,T)$. In particular if $\delta$ is sufficient large and $18 e \mathfrak f < 1$, then the BSDE has a unique solution in $\bD_\delta^2(0,T)$. 
\end{Thm}

Let us define for $\delta < \gamma \leq \beta$ the quantity
$$\Pi^{\mathfrak f}(\gamma,\delta) =   \frac{11}{\delta} + 9 \frac{e^{(\gamma-\delta)\mathfrak f}}{(\gamma-\delta)}.$$ 
The next technical lemma is equivalent to \cite[Lemma 3.4]{papa:poss:sapl:18}. Its proof is postponed in the Appendix. 
\begin{Lemma} \label{lem:infimum}
The infimum of $ \Pi^{\mathfrak f}(\gamma,\delta) $ over all $\delta < \gamma \leq \beta$ is given by $M^{\mathfrak f}(\beta)= \Pi^{\mathfrak f}(\beta,\delta^*(\beta)),$
where $\delta^*(\beta)$ is the unique solution on $(0,\beta)$ of the equation:
$$11(\beta-x)^2 - 9 e^{(\beta-x)\mathfrak f} x^2 (\mathfrak f(\beta-x)-1)=0.$$
Moreover $\displaystyle \lim_{\beta\to +\infty} M^{\mathfrak f}(\beta) = 9e\mathfrak f.$
\end{Lemma}
For $M^{\mathfrak f}(\beta) < 1/2$, we will also consider the quantities
\begin{equation} \label{eq:op_constants_beta_f}
\Sigma^\mathfrak f(\beta) = \dfrac{2 M^{\mathfrak f}(\beta)}{1-2M^{\mathfrak f}(\beta)},\qquad \widetilde \Sigma^\mathfrak f(\beta) = \Sigma^\mathfrak f(\beta) \dfrac{e^{(\beta-\delta^*(\beta))\mathfrak f}}{\beta-\delta^*(\beta)}.
\end{equation}

\subsection{The It\^o setting and related BSDEs} \label{ssect:Ito_setting}

In Section \ref{sect:Ito_setting_BSVIE}, we assume that the processes are It\^o semimartingales in the sense of \cite[Definition 1.16]{ait:jaco:15}. Hence the process $B$ is now deterministic and equal to $B_t=t$.  The semimartingale $\bar X$ can be represented by the Grigelionis form\footnote{In general only on a very good filtered extension of the original probability space. But using the remarks of \cite[section 1.4.3]{ait:jaco:15}, we can assume that this form of $\bar X$ holds on the original filtered probability space.}. Up to some modifications in the generator\footnote{In general we should take into account a possible degeneracy of the coefficients. This possibility leads to a non Lipschitz continuous new version of the generator $f$. We ignore this trouble here.}, the BSDE \eqref{eq:poss_BSDE} takes the next form:
\begin{eqnarray}  \nonumber
Y(t) &= & \xi + \int_t^{T} f(s, Y(s), Z(s),U(s,\cdot) ) \dd s \\ \label{eq:Ito_BSDE}
& -&  \int_t^{T} Z(s)  \dd W_s - \int_t^T \int_{\bR^m} U(s,x)  \tpi(\dd s,\dd x) - \int_t^T \dd M(s),
\end{eqnarray}
where
\begin{itemize}
\item $W$ is a Brownian motion
\item $\tpi$ is a Poisson random measure on $[0,T]\times \bR^m$, with intensity $\dd t \otimes \mu(\dd x)$.
\end{itemize}
In this setting, the BSVIE \eqref{eq:general_BSVIE} becomes the BSVIE \eqref{eq:Ito_general_BSVIE}:
\begin{eqnarray*} \nonumber
Y(t) &= &\Phi(t) + \int_t^{T} f(t,s, Y(s), Z(t,s) ,Z(s,t),U(t,s),U(s,t)) \dd s  \\ 
&-& \int_t^{T} Z(t,s) \dd W_s -  \int_t^{T}  \int_{\bR^m} U(t,s,x) \tpi(\dd s,\dd x) -\int_t^{T} \dd M(t,s).
\end{eqnarray*}
Let us first recall some standard notations.
\begin{itemize}
\item $\bS^p(0,T)$ is the space of all $\bF$-adapted \cad processes $X$ such that
$\bE \left(  \sup_{t\in [0,T]} |X_t|^p \right) < +\infty.$
\item $\bH^p(0,T)$ is the subspace of all predictable processes $X$ such that
$\bE \left[ \left( \int_0^T |X_t|^2 \dd t\right)^{p/2} \right] < +\infty.$
In this setting, $\bH^{2,\circ}_\beta(0,T)$ (with $\beta=0$) is equal to $\bH^2(0,T)$. 
\item $\cH^p(0,T)$ is the space of all martingales such that
$\bE \left[ \left( \langle M \rangle_T \right)^{p/2}\right] < +\infty.$
The space $\cH^{2,\perp}_\beta(0,T)$, with $\beta=0$, becomes now $\cH^{2,\perp}(0,T)$ and we define in a similar way $\cH^{p,\perp}(0,T)$ as a subspace of $\cH^p(0,T)$. 
\item $\bL^p_\pi(0,T) = \bL^p_{\pi}(\Omega\times [0,T] \times \bR^m)$ is the set of processes $\psi$ such that
$$\bE \left[ \left(  \int_0^T \int_{\bR^m} |\psi_s(x)|^2 \pi(\dd s,\dd x) \right)^{p/2} \right] < +\infty .$$
For $p=2$, it corresponds to $\bH^{2,\natural}_\beta(0,T)$ with $\beta=0$ (see Lemma \ref{lem:equiv_norm_poss}). 
\item $\bL^p_\mu=\bL^p(\bR^m,\mu;\bR^d)$ is the set of measurable functions $\psi : \bR^m \to \bR^d$ such that
$\| \psi \|^p_{\bL^p_\mu} = \int_{\bR^m} |\psi(x)|^p \mu(\dd x)  < +\infty .$
\item $\bD^p(0,T) = \bS^p(0,T) \times \bH^p(0,T) \times \bL^p_\pi(0,T) \times \cH^{p,\perp}(0,T)$.
\end{itemize}
For the BSDE \eqref{eq:poss_BSDE}, the $L^p$-theory has not been developed yet. But in the It\^o case the next result is proved in \cite{krus:popi:17}. Let us reinforce the condition \ref{H3}:
\begin{enumerate}[label=\textbf{(H\arabic**)}]
\setcounter{enumi}{2}
\item \label{H3star} 
There exists a constant $K$ such that a.s. for any $s\in [0,T]$ and $t \in [0,s]$,  
$$K^2 \geq \max(\sqrt{\varpi_t(\omega)},\theta^\circ_{t,s}(\omega),\theta^\natural_{t,s}(\omega)).$$ 
\end{enumerate}
\begin{Prop}\label{prop:BSDE_Lp_sol_exist}
Assume that for any $(y,z,\psi)$, $f(\cdot,y,z,\psi)$ is progressively measurable and that {\rm \ref{H2}} and {\rm \ref{H3star}} hold. 
If 
$$ \bE \left(|\xi|^p +  \int_0^T |f(t,0,0,\mathbf{0})|^p dt \right) < +\infty,$$
there exists a unique solution $(Y,Z,U,M)$ in $\bD^p(0,T)$ to the BSDE \eqref{eq:Ito_BSDE}. Moreover for some constant $C=C_{p,K,T}$
\begin{eqnarray*}
&& \bE \left[\sup_{t\in[0,T]}   |Y_t|^p +  \left( \int_0^T  |Z_t|^2 \dd t \right)^{p/2} +  \left(  \int_0^T  \int_{\bR^m} |U_s(x)|^2 \pi(\dd s, \dd x) \right)^{p/2} + \left(  \langle M \rangle_T \right)^{p/2}\right] \\
&& \qquad \qquad \leq C \bE \left[|\xi|^p  + \left( \int_0^T |f(r,0,0,\mathbf{0})| \dd r \right)^p \right].
\end{eqnarray*}
\end{Prop}
Condition \ref{H3star} is certainly too strong for this result and could be relaxed. Thereby our results concerning BSVIEs in the It\^o setting could be extended to non Lipschitz-continuous w.r.t. $y$ driver $f$ as in \cite{wang:zhan:07} (using a concave function $\rho$) or if $K$ becomes a function of $(\omega,s)$ with a suitable integrability condition (see \cite{yong:08} Condition (3.13)). 
Nonetheless such extensions would increase the length of the paper and are left for further research.

\section{Existence, uniqueness and comparison principle for the Type-I BSVIE \eqref{eq:general_BSVIE_type_I}.} \label{sect:general_Type_I_BSVIE}

%
%
%

Concerning BSVIE, we take the same notations as in \cite{yong:08}, we only adapt them to our setting and thus we skip some details (see \cite[Section 2.1]{yong:08} for interesting readers). For $0\leq R\leq S\leq T$ we denote
\begin{eqnarray*}
\Delta[R,S] & = & \{ (t,s) \in [R,S]^2 , \ R \leq s \leq t \leq S\},\\
\Delta^c[R,S] & = & \{ (t,s) \in [R,S]^2 , \ R \leq t < s  \leq S\}= [R,S]^2 \setminus \Delta[R,S] .
\end{eqnarray*}
This notation should not be confused with the jump of a process. 

We fix some $\beta \in \bR_+$, some $0\leq \delta \leq \beta$ and some c\`adl\`ag increasing and measurable process $A$. Again to simplify the notations, the dependance on $A$ is suppressed. 
First
\begin{eqnarray*}
\bL^2_{\beta,\cF_T}(0,T) &= &\Bigg\{ \phi : (0,T)\times \Omega \to \bR^d , \ \cB([0,T])\otimes\cF_T-\mbox{measurable with } \\
&&\qquad \qquad \left. \bE \left[ \int_0^T e^{\beta A_t} \left(  e^{\beta A_T} |\phi(t)|^2 \right) \dd B_t\right]< +\infty \right\}.
\end{eqnarray*}
The above space is for the free term $\Phi(\cdot)$ (for which the $\bF$-adaptiveness is not required). When $\bF$-adaptiveness is required, that is for $Y(\cdot)$, we define
\begin{eqnarray*}
\bL^2_{\beta,\bF}(0,T) &= &\Bigg\{ \phi : (0,T)\times \Omega \to \bR^d , \ \cB([0,T])\otimes\cF_T-\mbox{measurable and } \bF-\mbox{adapted with } \\
&&\qquad \qquad \left. \bE \left[ \int_0^T e^{\beta A_t} \left(  |\phi(t)|^2 \right) \dd B_t\right] < +\infty \right\}.
\end{eqnarray*}

To control the martingale terms in the Type-I BSVIE, we need other spaces. 
We define 
$$\cH_{\delta\leq \beta}^2(\Delta^c(R,T))$$
the set of processes $M(\cdot,\cdot)$ such that for $t\in [R,T]$, $M(t,\cdot)=\{M(t,s), \ s \geq t\}$ belongs to $\cH^2(\bR^d)$ and 
$$ \bE \left[ \int_R^T e^{(\beta-\delta) A_t} \int_t^T e^{\delta A_s} \dd  \trace \langle M(t,\cdot)\rangle_s    \dd A_t \right] <+\infty.$$
In the particular case where $M(t,\cdot) = \int_S^\cdot Z(t,s) \dd X^\circ_s$, then $M \in \cH_{\gamma,\delta}^2(\Delta^c(R,T))$
 is equivalent to $Z \in \bH^{2,\circ}_{\delta\leq \beta}(\Delta^c(R,T))$ if $Z(t,\cdot)$ belongs to $\bH^{2,\circ}_\delta(t,T)$ and from Lemma \ref{lem:equiv_norm_poss}
$$  \bE \left[ \int_R^T e^{(\beta-\delta) A_t} \int_t^T e^{\delta A_s} \dd  \trace \langle M(t,\cdot)\rangle_s    \dd A_t \right] = \bE \left[ \int_R^T e^{(\beta-\delta) A_t} \int_t^T e^{\delta A_s} \left\| b_s Z(t,s)\right\|^2 \dd B_s \dd A_t  \right].$$
Note that we strongly use that we work on $\Delta^c(R,T)$ and that $A_t$ is $\cF_t$-measurable to obtain this equality. For Type-II BSVIEs this point is an issue. Similarly for $U(t,\cdot) \in \bH^{2,\natural}_\delta(t,T)$, the martingale $M$ defined by $M(t,\cdot) = U(t,\cdot) \star \tpi^\natural$ is in $\cH_{\delta \leq \beta}^2(\Delta^c(R,T))$ if
\begin{eqnarray*}
&& \bE \left[ \int_R^T e^{(\beta-\delta) A_t} \int_t^T e^{\delta A_s} \dd  \trace \langle M(t,\cdot)\rangle_s    \dd A_t \right] \\
&& = \bE \left[ \int_R^T e^{(\beta-\delta) A_t} \int_t^T e^{\delta A_s} \vertiii{U(t,\cdot)}_s^2 \dd B_s \dd A_t  \right] < +\infty.
\end{eqnarray*}
In this case $U \in \bH^{2,\natural}_{\delta\leq \beta}(\Delta^c(R,T))$. To lighten the notations, we define the martingale for $t \leq u \leq T$
\begin{equation} \label{eq:def_mart_sharp}
\mart(t,u) = \int_t^u Z(t,s) \dd X^\circ_s + \int_t^u \int_{\bR^m} U(t,s,x) \tpi^\natural(\dd s,\dd x) + M(t,u), 
\end{equation}
such that the BSVIE \eqref{eq:general_BSVIE_type_I} becomes
\begin{eqnarray}\label{eq:BSVIE_gene_filtr_type_I}
Y(t) &= &\Phi(t) + \int_t^{T} f(t,s, Y(s), Z(t,s) ,U(t,s)) \dd B_s - \left(\mart(t,T) - \mart(t,t)\right). 
\end{eqnarray}
Using \cite[Corollary 2.7]{papa:poss:sapl:18}, $\mart$ belongs to $\cH_{\delta \leq \beta}^2(\Delta^c(R,T))$ if and only if the triplet $(Z,U,M)$ belongs to $\bH^{2,\circ}_{\delta\leq \beta}(\Delta^c(R,T)) \times \bH^{2,\circ}_{\delta\leq \beta}(\Delta^c(R,T)) \times \cH^{2,\perp}_{\delta\leq \beta}(\Delta^c(R,T))$:
\begin{eqnarray} \nonumber
&& \bE \left[  \int_0^Te^{(\beta-\delta) A_t } \left(  \int_t^T e^{\delta A_r} \dd \trace [\langle Z(t,\cdot)\cdot X^\circ\rangle_r ] \right) \dd A_t \right. \\ \nonumber
&&\qquad + \int_0^T e^{(\beta-\delta) A_t }\left( \int_t^T e^{\delta A_r} \dd \trace [\langle U(t,\cdot)\star \tpi^\natural \rangle_r ]  \right) \dd A_t\\ \nonumber
&&\qquad + \left. \int_0^T  e^{(\beta-\delta) A_t } \left( \int_t^T e^{\delta A_r} \dd \trace [\langle M(t,\cdot) \rangle_r ] \right) \dd A_t  \right] < +\infty.
\end{eqnarray}
Finally we define
\begin{eqnarray*}
\mathfrak{S}^2_{\delta\leq \beta}(\Delta^c(0,T))&=& \bL^2_{\beta,\bF}(0,T) \times\bH^{2,\circ}_{\delta\leq \beta}(\Delta^c(0,T)) \times \bH^{2,\circ}_{\delta\leq \beta}(\Delta^c(0,T)) \times \cH^{2,\perp}_{\delta\leq \beta}(\Delta^c(0,T))
\end{eqnarray*}
with the naturally induced norm. If $\delta=\delta(\beta)$ is a known function of $\beta$, we denote $\mathfrak{S}^2_{\delta\leq \beta}(\Delta^c(0,T))$  by $\mathfrak{S}^2_{\beta}(\Delta^c(0,T))$.

The notion of solution of a Type-I BSVIE is the following.
\begin{Def}[Adapted solution] \label{def:BSVI_adapted_sol}
A quadruple $(Y,Z,U,M)$ is called an adapted solution of the Type-I BSVIE \eqref{eq:general_BSVIE_type_I} if $(Y,Z,U,M)$ belongs to $\mathfrak{S}^2_{\delta\leq \beta}(\Delta^c(0,T))$ for some $\delta \leq \beta$ and if the equation is satisfied a.s. for almost all $t\in [0,T]$. 
\end{Def}
Note that if $(Y,Z,U,M)$ is a solution of the BSDE \eqref{eq:poss_BSDE} in $\bD_\beta^2(0,T)$, then $Y \in  \bL^2_{\beta,\bF}(0,T)$ and  
\begin{eqnarray*}
&&  \bE \left[ \int_0^T e^{(\beta-\delta) A_t} \int_t^T e^{\delta A_s} \left\| b_s Z(s)\right\|^2 \dd B_s \dd A_t  \right] \\
&& \quad = \bE \left[ \int_0^T \int_0^s e^{(\beta-\delta) A_t} e^{\delta A_s} \left\| b_s Z(s)\right\|^2  \dd A_t \dd B_s \right]  \\
&& \quad \leq \dfrac{e^{(\beta-\delta)\mathfrak f}}{\beta-\delta} \bE \left[ \int_0^T e^{\beta A_s} \left\| b_s Z(s)\right\|^2 \dd B_s \right] < +\infty.
\end{eqnarray*}
In other words $(Z,U,M)$ is in $\bH^{2,\circ}_{\delta\leq \beta}(\Delta^c(0,T)) \times \bH^{2,\circ}_{\delta\leq \beta}(\Delta^c(0,T)) \times \cH^{2,\perp}_{\delta\leq \beta}(\Delta^c(0,T))$, and thus $(Y,Z,U,M)$ is an adapted solution of the  BSVIE \eqref{eq:general_BSVIE_type_I} where $\Phi(t) = \xi$ and $f$ does not depend on $t$.

For the Type-II BSVIE \eqref{eq:general_BSVIE}, as it pointed out in \cite{yong:08}, uniqueness of an adapted solution fails. Roughly speaking, there is an additional freedom on $\Delta(0,T)$. To avoid this problem, the next definition is formulated in \cite{yong:08}.
\begin{Def}[M-solution]\label{def:BSVI_M_sol}
Let $S \in [0,T)$. A quadruple $(Y,Z,U,M)$ is called an adapted M-solution of \eqref{eq:general_BSVIE} on $[S,T]$ if \eqref{eq:general_BSVIE} holds in the usual It\^o sense for almost all $t \in [S,T]$ and, in addition, the following holds: for a.e. $t \in [S,T]$
\begin{equation} \label{eq:M_sol_def}
Y(t) = \bE \left[ Y(t) |\cF_S\right] + \int_S^t Z(t,s) \dd X^\circ_s + \int_S^t \int_{\bR^m} U(t,s,x)\tpi^\natural(\dd s,\dd x) + \int_S^t \dd M(t,s). 
\end{equation}
\end{Def}
Note that we keep the notion of M-solution of \cite{yong:08}, where the letter M stands for ``a martingale representation'' for $Y(t)$ to determine $Z(\cdot,\cdot)$, $U(\cdot,\cdot)$ and $M(\cdot,\cdot)$ on $\Delta[S,T]$. It should not be confused with the orthogonal martingale part $M$. As in \cite{yong:08}, any M-solution on $[S,T]$ is also a M-solution on $[\bar S,T]$ with $\bar S \in (S,T)$.

\subsection{Assumptions and main result}

Let us precise the assumptions on the free term $\Phi$ and on the generator $f$ of the BSVIE \eqref{eq:general_BSVIE_type_I}. 
\begin{enumerate}[label=\textbf{(H\arabic*)}]
\item \label{H1} $\Phi \in \bL^2_{\beta,\cF_T}(0,T)$.

\item \label{H2} The driver $f$ is defined on $\Omega \times \Delta^c(0,T) \times \bR^{d} \times \bR^{d\times m} \times \mathfrak H \to \bR^d$ and we assume that for any fixed $(t,y,z,u)$ the process $f(t,\cdot,y,z,u)$ is progressively measurable. Moreover there exist
$$\varpi:(\Omega\times \Delta^c(0,T), \cP) \to \bR_+, \qquad \vartheta = (\theta^\circ,\theta^\natural) : (\Omega\times \Delta^c(0,T), \cP) \to (\bR_+)^2$$
such that for $\dd B \otimes \dd B \otimes \dd \bP$-a.e. $(t,s,\omega)$
\begin{eqnarray*}
&& |f(\omega,t,s,y,z,u_s(\omega;\cdot)) - f(\omega,t,s,y',z',u'_s(\omega;\cdot)) |^2 \\
&&\quad \leq \varpi(\omega,t,s)|y-y'|^2 + \theta^\circ(\omega,t,s)\|c_s(\omega)(z-z')\|^2 \\
&&\qquad + \theta^\natural(\omega,t,s)(\vertiii{u_s(\omega;\cdot)-u_s(\omega;\cdot)}_s(\omega))^2.
\end{eqnarray*}
To simplify the notation in the sequel: $f^0(t,s) = f(t,s,0,0,\mathbf 0).$
\item \label{H3} The hypothesis \ref{F3} holds. Namely we assume that there exists $\alpha(\omega,s)>0$ defined on $\Omega \times [0,T]$ such that $\alpha^2_s(\omega) \geq \max(\sqrt{\varpi}(\omega,t,s),\theta^\circ(\omega,t,s),\theta^\natural(\omega,t,s))$ for $(\omega,t,s) \in \Omega \times \Delta^c(0,T)$. We define the process $A$ by \eqref{eq:def_A}.
For any fixed $t \in [0,T]$, $(\omega,r) \in \Omega \times [0,T] \mapsto A_{r}(\omega)$ is increasing, $\bF$-predictable and c\`adl\`ag. 
There exists $\mathfrak f > 0$ such that for any $t$
$$\Delta A_{r}(\omega) \leq \mathfrak f, \ \mbox{for }\dd B\otimes \dd \bP-\ \mbox{a.e. } (r,\omega).$$

\item \label{H4} For the same $\beta$ as in \ref{H1}, $\delta$ denotes the constant $\delta^*(\beta)$ of the above lemma \ref{lem:infimum}. We assume that
$$
\bE \left[ \int_0^T e^{(\beta-\delta) A_t} \left( \int_t^T e^{\delta A_s} \dfrac{|f(t,s,0,0,\mathbf 0)|^2}{\alpha_s^2} \dd B_s \right) \dd B_t \right] < +\infty.
$$
\item \label{H5} Let us define the following set
\begin{equation} \label{eq:subset_times_square_int}
\mathfrak T^{\Phi,f}_{\delta} = \left\{ t \in [0,T]: \bE  \left[ e^{\delta A_T} |\Phi(t)|^2 + \int_t^T e^{\delta A_s} \dfrac{|f(t,s,0,0,\mathbf 0)|^2}{\alpha_s^2} \dd B_s \right] < +\infty \right\}.
\end{equation}
We assume that the set $\mathfrak T^{\Phi,f}_{\delta}$ is equal to $[0,T]$. 
\end{enumerate}
Now we come back to the BSVIE \eqref{eq:general_BSVIE_type_I} and state our main result on this BSVIE.
\begin{Thm} \label{thm:type_I_general_BSVIE}
Assume that the conditions {\rm \ref{H1}} to {\rm \ref{H5}} hold. We also suppose that the constants $\kappa^\mathfrak f(\delta)$, $M^{\mathfrak f}(\beta)$ and $\widetilde \Sigma^\mathfrak f(\beta)$ defined by \eqref{eq:cond_exist_beta_poss}, in Lemma \ref{lem:infimum} and  by \eqref{eq:op_constants_beta_f} verify
\begin{equation} \label{eq:cond_beta_Type_I_BSVIE}
\kappa^\mathfrak f(\delta) < \dfrac{1}{2},\qquad M^{\mathfrak f}(\beta) < \dfrac{1}{2}, \qquad \widetilde \Sigma^\mathfrak f(\beta) < 1,
\end{equation}
where $\delta=\delta^*(\beta)$ is defined in Lemma \ref{lem:infimum}. Then the BSVIE \eqref{eq:general_BSVIE_type_I} has a unique adapted solution $(Y,Z,U,M)$ in $ \mathfrak{S}^2_{\beta}(\Delta^c(0,T))$. Moreover there exists a constant $\mathfrak C^\mathfrak f(\beta)$ such that 
\begin{eqnarray} \nonumber
&& \bE \left[ \int_0^T e^{ \beta A_t} |Y(t)|^2 \dd A_t +  \int_0^Te^{(\beta-\delta)A_t } \left(  \int_t^T e^{\delta A_r} \dd \trace [\langle Z(t,\cdot)\cdot X^\circ\rangle_r ] \right) \dd A_t \right. \\ \nonumber
&&\qquad + \int_0^T e^{(\beta-\delta)A_t }\left( \int_t^T e^{\delta A_r} \dd \trace [\langle U(t,\cdot)\star \tpi^\natural \rangle_r ]  \right) \dd A_t\\ \nonumber
&&\qquad + \left. \int_0^T  e^{(\beta-\delta)A_t } \left( \int_t^T e^{\delta A_r} \dd \trace [\langle M(t,\cdot) \rangle_r ] \right) \dd A_t  \right]\\ \label{eq:a_priori_estim_general_BSVIE_type_I}
&& \quad  \leq  \mathfrak C^\mathfrak f(\beta) \bE \left[ \int_0^T e^{\beta A_t} |\Phi(t)|^2  \dd A_t +  \int_0^T  e^{(\beta-\delta) A_t}  \int_t^T e^{\delta A_s}\dfrac{ |f^0(t,s)|^2}{\alpha_s^2} \dd B_s  \dd A_t \right] .
\end{eqnarray}
 Let $(\bar \Phi,\bar h)$ be a couple of data each satisfying the above assumptions {\rm \ref{H1}} to {\rm \ref{H5}}. Let $(\bar Y,\bar Z,\bar U,\bar M)$ a solution of the BSVIE \eqref{eq:general_BSVIE_type_I} with data $(\bar \Phi, \bar f)$. Define
$$(\fd Y,\fd  Z, \fd  U, \fd  M) = (Y-\bar Y, Z-\bar Z,U-\bar U,M-\bar M).$$
Then using \eqref{eq:def_mart_sharp} we have the following stability result:
\begin{eqnarray*} \nonumber
&& \bE \left[ \int_0^T e^{ \beta A_t} |\fd  Y(t)|^2 \dd A_t +  \int_0^T e^{(\beta-\delta) A_t}  \int_t^T e^{\delta A_s} \dd  \trace \langle \fd  \mart(t,\cdot)\rangle_s    \dd A_t \right] \\  \nonumber
&& \quad  \leq \mathfrak C^\mathfrak f(\beta) \bE \left[ \int_0^T e^{\beta A_t} |\fd  \Phi(t)|^2  \dd A_t\right] \\
&&\qquad +\mathfrak C^\mathfrak f(\beta) \bE \left[ \int_0^T  e^{(\beta-\delta) A_t}  \int_t^T e^{\delta A_s}\dfrac{ |\fd  f(t,s)|^2}{\alpha_s^2} \dd B_s  \dd A_t \right] 
\end{eqnarray*}
with
$$\fd  f(t,r)= f(t,r,Y(r),Z(t,r),U(t,r))-\bar f(t,r,Y(r),Z(t,r),U(t,r)).$$
\end{Thm}

Let us comment the condition \eqref{eq:cond_beta_Type_I_BSVIE}. The first part $\kappa^\mathfrak f(\delta)<1/2$ comes from (\cite[Theorem 3.5]{papa:poss:sapl:18} (see Theorem \ref{thm:exis_sol_poss}). The second condition is sufficient to have existence and uniqueness of the solution of the Type-I BSVIE if $f$ does not depend on $Y$ (see the BSVIE \eqref{eq:special_BSVIE} and Proposition \ref{prop:type_I_BSVIE_without_Y} below). The last part $\widetilde \Sigma^\mathfrak f(\beta) < 1$ ensures existence and uniqueness thanks to a fixed point argument. 

\begin{Rem}[Large values of $\beta$] \label{rem:large_values_beta}
Note that for $\beta$ large, $\kappa^\mathfrak f(\delta)<1/2$ and $M^{\mathfrak f}(\beta) < 1/2$ if and only if $18 e\mathfrak f < 1$ (as in \cite{papa:poss:sapl:18}). And since
$$\lim_{\beta \to + \infty}  \widetilde \Sigma^\mathfrak f(\beta)  = \dfrac{18(e\mathfrak f)^2}{1-18e\mathfrak f} ,$$
 $\widetilde \Sigma^\mathfrak f(\beta)  < 1$ if and only if $18e\mathfrak f < 3 (\sqrt{11}-3) < 1.$
In other words, if $18e\mathfrak f < 3 (\sqrt{11}-3) < 1$ (and $3 (\sqrt{11}-3) \approx 0.95$), then \eqref{eq:cond_beta_Type_I_BSVIE} holds for large values of $\beta$. 
\end{Rem}

\begin{Rem}[On Condition \ref{H5}]
If the assumptions {\rm \ref{H1}} and {\rm \ref{H4}} hold, then 
$$
\bE \int_0^T  \left( e^{\delta A_T} |\Phi(t)|^2 \right) \dd B_t < +\infty, \ \mbox{and }
\bE \int_0^T  \left( \int_t^T e^{\delta A_s} \dfrac{|f(t,s,0,0,\mathbf 0)|^2}{\alpha_s^2} \dd B_s \right) \dd B_t < +\infty.
$$
For a random process $B$, we cannot switch the time integral and the expectation and a precise description of the set $\mathfrak T^{\Phi,f}_\delta$ is not easy.
But if $B$ is deterministic, we deduce that $\dd B$-almost every $t \in [0,T]$ belongs to $\mathfrak T^{\Phi,f}_\delta$. In the Brownian-Poisson setting (Example \ref{exp:brownian_poisson_setting}), $\dd B$ is the Lebesgue measure. If $B$ is piecewise-constant as in the second example after \cite[Remark 3.19]{papa:poss:sapl:18} with deterministic jump times, we have a similar $\dd B$-a.e. property. In other words {\rm \ref{H5}} is too strong in a lot of cases, but it makes the presentation of our results correct and easier in the general setting. 
\end{Rem}

\subsection{Preliminary results}

First we consider for any $R$ and $S$ in $[0,T)$ a driver $h:\Omega \times [S,T]\times [R,T] \times \bR^{d\times m}\times \mathfrak H \to \bR^d$ such that \ref{H2} holds and:
\begin{equation}\label{eq:int_cond_h_0}
\bE \left[ \int_S^T e^{(\beta-\delta) A_t} \left( \int_R^T e^{\delta A_s} \dfrac{|h(t,s,0,\mathbf 0)|^2}{\alpha_s^2} \dd B_s \right) \dd B_t \right] < +\infty.
\end{equation}
Let us recall that the constant $\delta=\delta^*(\beta)$ is defined in Lemma \ref{lem:infimum}. For ease of notation, $h^0(t,s) = h(t,s,0,\mathbf 0)$. Like \ref{H5}, we also assume that 
\begin{equation} \label{eq:subset_times_square_int_h}
\left\{ t \in [S,T]: \bE  \left[ e^{\delta A_T} |\Phi(t)|^2 + \int_R^T e^{\delta A_s} \dfrac{|h(t,s,0,\mathbf 0)|^2}{\alpha_s^2} \dd B_s \right] < +\infty \right\} = [S,T].
\end{equation}
Then let us define the BSDE on $[R,T]$ parameterized by $t \in [S,T] $: for $r \in [R,T]$
\begin{eqnarray} \nonumber
\lambda(t,r) &= &\Phi(t) + \int_r^{T} h(t,s, z(t,s) ,u(t,s)) \dd B_s - \int_r^{T} z(t,s) \dd X^\circ_s  \\ \label{eq:parametrized_BSDE}
&& \qquad - \int_r^{T}  \int_{\bR^m} u(t,s,x) \tpi^\natural(\dd s,\dd x) -\int_r^{T} \dd m(t,s).
\end{eqnarray}
From Theorem \ref{thm:exis_sol_poss}, if \eqref{eq:cond_exist_beta_poss} holds, that is $\kappa^\mathfrak f(\delta) < 1/2$, the previous BSDE has a unique solution 
$$(\lambda(t,\cdot),z(t,\cdot),u(t,\cdot),m(t,\cdot)).$$ 
Let us fix $R=S$ and define $Y(t) = \lambda(t,t)$, $t \in [S,T]$, $Z(t,s) = z(t,s)$, $U(t,s,e) = u(t,s,e)$, $M(t,s) = m(t,s)$ for $(t,s) \in \Delta^c[S,T]$. Equation \eqref{eq:parametrized_BSDE} becomes
\begin{eqnarray} \nonumber
Y(t) &= &\Phi(t) + \int_t^{T} h(t,s, Z(t,s) ,U(t,s)) \dd B_s - \int_t^{T} Z(t,s) \dd X^\circ_s  \\ \label{eq:special_BSVIE}
&& \qquad - \int_t^{T}  \int_{\bR^m} U(t,s,x) \tpi^\natural(\dd s,\dd x) -\int_t^{T} \dd M(t,s),
\end{eqnarray}
which a special case of \eqref{eq:general_BSVIE_type_I} where $f$ does not depend on $y$. Using \eqref{eq:def_mart_sharp}, the BSVIE becomes for any $t\in [S,T]$
\begin{eqnarray} \nonumber 
Y(t) &= &\Phi(t) + \int_t^{T} h(t,s, Z(s,t),U(t,s)) \dd B_s - \int_t^T \dd \mart(t,s).
\end{eqnarray}
Let us prove the next result. 
\begin{Lemma} \label{lem:param_BSDE->BSVIE}
Assume that {\rm \ref{H1}} holds for $\Phi$, that $h$ satisfies {\rm \ref{H2}} and {\rm \ref{H3}}, together with the conditions \eqref{eq:int_cond_h_0} and \eqref{eq:subset_times_square_int_h}. If \eqref{eq:cond_exist_beta_poss} holds and if the constant $M^{\mathfrak f}(\beta)$ defined in Lemma \ref{lem:infimum} verifies: $M^{\mathfrak f}(\beta) <1/2$, then the solution $(Y,Z,U,M)$ of the BSVIE \eqref{eq:special_BSVIE} admits the upper bound:
\begin{eqnarray} \nonumber
&& \bE \left[ \int_S^T e^{ \beta A_t} |Y(t)|^2 \dd A_t +  \int_S^T e^{(\beta-\delta) A_t}  \int_t^T e^{\delta A_s} \dd  \trace \langle \mart(t,\cdot)\rangle_s    \dd A_t \right] \\  \nonumber
&& \quad  \leq\dfrac{\delta}{2}  \Sigma^\mathfrak f( \beta) \bE \left[ \int_S^T e^{\beta A_t} |\Phi(t)|^2  \dd A_t\right] \\ \label{eq:H2-estimate}
&&\qquad + \Sigma^\mathfrak f(\beta)  \bE \left[ \int_S^T  e^{(\beta-\delta) A_t}  \int_t^T e^{\delta A_s}\dfrac{ |h^0(t,s)|^2}{\alpha_s^2} \dd B_s  \dd A_t \right] ,
\end{eqnarray}
where 
$ \Sigma^\mathfrak f(\beta) $ is defined by \eqref{eq:op_constants_beta_f}.
\end{Lemma}
\begin{proof}
Note that for any $t \in [0,T]$
$$Y(t) = \Phi(t) + H(t) - \int_t^T \dd \mart(t,s),$$
where for $t\in [S,T]$, $H(t) = \int_t^T  h(t,s, Z(t,s) ,U(t,s)) \dd B_s.$
Therefore
$$\bE \left[ \int_t^T \dd \trace\langle \mart(t,\cdot)\rangle_s \bigg| \cF_t \right] = \bE \left[ |\Phi(t) - Y(t) - H(t)|^2 \bigg| \cF_t \right] $$
For any $\delta >0$, by the Cauchy-Schwarz inequality and using \cite[Lemma B.1]{papa:poss:sapl:18}:
\begin{eqnarray*}
|H(t)|^2 & \leq & \int_t^T e^{-\delta A_s} \dd A_s  \int_t^T e^{\delta A_s}\dfrac{ |h(t,s, Z(t,s) ,U(t,s))|^2}{\alpha_s^2} \dd B_s \\
& \leq & \frac{1}{\delta} e^{-\delta A_t}  \int_t^T e^{\delta A_s}\dfrac{ |h(t,s, Z(t,s) ,U(t,s))|^2}{\alpha_s^2} \dd B_s.
\end{eqnarray*}
Arguing as in the proof of  \cite[Lemma 3.8]{papa:poss:sapl:18}, we deduce that for any $\gamma > 0$ and $\delta > 0$
\begin{equation*} 
\int_S^T e^{\gamma A_t}|H(t)|^2 \dd A_t \leq \int_S^T  \frac{1}{\delta} e^{(\gamma-\delta) A_t}  \int_t^T e^{\delta A_s}\dfrac{ |h(t,s, Z(t,s) ,U(t,s))|^2}{\alpha_s^2} \dd B_s  \dd A_t .
\end{equation*}
From our assumption \ref{H2} on $h$, we have:
$$ |h(t,r,Z(t,r) ,U(t,r))|^2  \leq 2|h^0(t,r)|^2  + 2\theta^\circ_r |c_r Z(t,r)|^2 + 2 \theta^\natural_r \vertiii{ U(t,r) }^2_r,$$
thus from the definition of $\alpha$ (hypothesis \ref{H3})
\begin{eqnarray*}
&& e^{\delta A_r} \dfrac{|h(t,r,Z(t,r) ,U(t,r))|^2}{\alpha_r^2}  \leq  2e^{\delta A_r} \dfrac{|h^0(t,r)|^2}{\alpha_r^2}  + 2e^{\delta A_r}\dfrac{\theta^\circ_r}{\alpha_r^2}   |c_r Z(t,r)|^2 \\
&& \qquad  +  2e^{\delta A_r}\dfrac{\theta^\natural_r}{\alpha_r^2}   \vertiii{ U(t,r) }^2_r \\
&&\quad  \leq  2e^{\delta A_r} \dfrac{|h^0(t,r)|^2}{\alpha_r^2}  + 2e^{\delta A_r}  |c_r Z(t,r)|^2  +  2e^{\delta A_r}  \vertiii{ U(t,r) }^2_r.
\end{eqnarray*}
Thereby for any $t\in [S,T]$
\begin{eqnarray*}
&&  \int_t^T e^{\delta A_s}\dfrac{ |h(t,s, Z(t,s) ,U(t,s))|^2}{\alpha_s^2} \dd B_s  \leq  2   \int_t^T e^{\delta A_r} \dfrac{|h^0(t,r)|^2}{\alpha_r^2} \dd B_s \\
&& \qquad + 2   \int_t^T e^{\delta A_r}  |c_r Z(t,r)|^2 \dd B_r +  2   \int_t^T e^{\delta A_r}  \vertiii{ U(t,r) }^2_r \dd B_r.
\end{eqnarray*}
Hence we deduce that 
\begin{eqnarray*}
\bE \left[\int_S^T e^{\gamma A_t}|H(t)|^2 \dd A_t \right] & \leq &\bE \left[ \int_S^T  \frac{1}{\delta} e^{(\gamma - \delta) A_t}  \int_t^T e^{\delta A_s}\dfrac{ |h(t,s, Z(t,s) ,U(t,s))|^2}{\alpha_s^2} \dd B_s  \dd A_t \right] \\
&\leq & 2 \bE \left[ \int_S^T  \frac{1}{\delta} e^{(\gamma-\delta) A_t}  \int_t^T e^{\delta A_s}\dfrac{ |h^0(t,s)|^2}{\alpha_s^2} \dd B_s  \dd A_t \right] \\
&+ & 2 \bE \left[ \int_S^T  \frac{1}{\delta} e^{(\gamma-\delta) A_t}  \int_t^T e^{\delta A_s} \dd  \trace [\langle Z(t,\cdot). X^\circ\rangle_s]   \dd A_t \right] \\
& + & 2 \bE \left[ \int_S^T  \frac{1}{\delta} e^{(\gamma- \delta) A_t}  \int_t^T e^{\delta A_s} \dd \trace [\langle U(t,\cdot)\star \tpi^\natural \rangle_s]   \dd A_t \right].
\end{eqnarray*}
This leads to the next estimate on $H$:
\begin{eqnarray} \nonumber
\bE \left[\int_S^T e^{\gamma A_t}|H(t)|^2 \dd A_t \right] 
& \leq & \frac{2}{\delta}  \bE \left[ \int_S^T  e^{(\gamma-\delta) A_t}  \int_t^T e^{\delta A_s}\dfrac{ |h^0(t,s)|^2}{\alpha_s^2} \dd B_s  \dd A_t \right] \\ \label{eq:estim_H_first}
&+ & \frac{2}{\delta}  \bE \left[ \int_S^T   e^{(\gamma - \delta) A_t}  \int_t^T e^{\delta A_s} \dd  \trace \langle \mart(t,\cdot)\rangle_s    \dd A_t \right] .
\end{eqnarray}


Now remark that
\begin{equation} \label{eq:expr_Y_1}
Y(t) = \bE \left[\Phi(t) + \int_t^{T} h(t,s, Z(t,s) ,U(t,s)) \dd B_s  \bigg| \cF_t\right]=  \bE \left[\Phi(t) + H(t)  \bigg| \cF_t\right].
\end{equation}
 Thus for $\gamma > 0$ and $\delta > 0$
\begin{eqnarray} \nonumber
\bE \left[ \int_S^T e^{\gamma A_t} |Y(t)|^2 \dd A_t \right] &\leq & 2 \bE \left[ \int_S^T \bE \left[ e^{\gamma A_t} |\Phi(t)|^2 + e^{\gamma A_t} |H(t)|^2  \bigg| \cF_t\right] \dd A_t\right] \\ \label{eq:estim_Y_1}
& = &  2 \bE \left[ \int_S^T \left( e^{\gamma A_t} |\Phi(t)|^2 + e^{\gamma A_t} |H(t)|^2  \right) \dd A_t\right].
\end{eqnarray}
Following the proof of \cite[Lemma 3.8]{papa:poss:sapl:18}, we obtain
\begin{eqnarray}\nonumber 
&&
\int_t^Te^{\delta A_s} \dd \trace \langle \mart(t,\cdot)\rangle_s  
\\ \label{eq:estim_mart_1}
&&\quad
 \leq  \delta \int_t^T e^{\delta A_u}\int_u^T \dd \trace \langle \mart(t,\cdot)\rangle_s \dd A_u + e^{\delta A_t} \left( \trace \langle \mart(t,\cdot)\rangle_T -\trace  \langle \mart(t,\cdot)\rangle_t \right). 
\end{eqnarray}
Then 
\begin{eqnarray*}
&& \bE \left[  \int_S^T e^{(\gamma-\delta)A_t} \int_t^T e^{\delta A_u}\int_u^T \dd \trace \langle \mart(t,\cdot)\rangle_s \dd A_u \dd A_t \right] \\
&& \quad =  \bE \left[  \int_S^T e^{(\gamma-\delta)A_t} \int_t^T e^{\delta A_u} \bE \left[\int_u^T \dd \trace \langle \mart(t,\cdot)\rangle_s \bigg| \cF_u\right] \dd A_u \dd A_t \right] \\
&&\quad =  \bE \left[  \int_S^T e^{(\gamma-\delta)A_t}   \int_t^T e^{\delta A_u} \bE \left[ |\Phi(u) - Y(u) - H(u)|^2 \bigg| \cF_u\right] \dd A_u \dd A_t \right] \\
&&\quad \leq 3  \bE \left[  \int_S^T e^{(\gamma-\delta)A_t}   \int_t^T e^{\delta A_u} \bE \left[ |\Phi(u)|^2 +| Y(u)|^2 + |H(u)|^2 \bigg| \cF_u\right] \dd A_u \dd A_t \right] \\
&&\quad  \leq 9  \bE \left[  \int_S^T e^{(\gamma-\delta)A_t}  \int_t^T e^{\delta A_u}  |\Phi(u)|^2 \dd A_u \dd A_t \right] \\
&&\qquad + 9  \bE \left[  \int_S^T e^{(\gamma-\delta)A_t} \int_t^T e^{\delta A_u}  |H(u)|^2 \dd A_u \dd A_t \right] .
\end{eqnarray*}
Note that we use several times that $A_t$ is $\cF_u$-measurable and \cite[Corollary D.1]{papa:poss:sapl:18}, together with \eqref{eq:expr_Y_1} for the last inequality. For $\gamma > \delta$,
$$ \bE \left[  \int_S^T e^{(\gamma-\delta)A_t}  \int_t^T e^{\delta A_u}  |\Phi(u)|^2 \dd A_u \dd A_t \right] \leq \frac{e^{(\gamma-\delta)\mathfrak f}}{(\gamma-\delta)} \bE \left[  \int_S^T  e^{\gamma A_u}  |\Phi(u)|^2 \dd A_u \right].$$
The same holds for $H(u)$, thus
\begin{eqnarray*}
&& \bE \left[  \int_S^T e^{(\gamma-\delta)A_t}  \int_t^T e^{\delta A_u}\int_u^T \dd \trace \langle \mart(t,\cdot)\rangle_s \dd A_u \dd A_t \right] \\
&& \quad  \leq 9 \frac{e^{(\gamma-\delta)\mathfrak f}}{(\gamma-\delta)}  \bE \left[  \int_S^T e^{\gamma A_u}  |\Phi(u)|^2 \dd A_u \right] + 9  \frac{e^{(\gamma-\delta)\mathfrak f}}{(\gamma-\delta)} \bE \left[  \int_S^T  e^{\gamma A_u}  |H(u)|^2 \dd A_u \right] .
\end{eqnarray*}
For the second term in \eqref{eq:estim_mart_1}, we have
\begin{eqnarray*}
&& \bE \left[  \int_S^T e^{(\gamma-\delta)A_t} e^{\delta A_t} \left( \trace \langle \mart(t,\cdot)\rangle_T -\trace  \langle \mart(t,\cdot)\rangle_t \right) \dd A_t  \right] \\
&&\quad \leq \bE \left[  \int_S^T e^{\gamma A_t} \left( \trace \langle \mart(t,\cdot)\rangle_T -\trace  \langle \mart(t,\cdot)\rangle_t \right) \dd A_t  \right] \\
&&\quad \leq 9  \bE \left[  \int_S^T e^{\gamma A_t}  \left( |\Phi(t)|^2 + |H(t)|^2 \right) \dd A_t  \right].
\end{eqnarray*}
Therefore we have for $\delta < \gamma \leq \beta$
\begin{eqnarray} \nonumber 
&& \bE \left[  \int_S^T e^{(\gamma-\delta)A_t}  \int_t^Te^{\delta A_s} \dd \trace \langle \mart(t,\cdot)\rangle_s  \dd A_t \right]\\ \label{eq:estim_mart_2}
&&\quad \leq \left(9 + 9 \delta \frac{e^{(\gamma-\delta)\mathfrak f}}{(\gamma-\delta)}\right)  \bE \left[  \int_S^T e^{\gamma A_t} \left( |\Phi(t)|^2 + |H(t)|^2 \right) \dd A_t  \right].
\end{eqnarray}
Combining \eqref{eq:estim_H_first}, \eqref{eq:estim_Y_1} and the preceding estimate, we have for any $\delta < \gamma \leq \beta$: 
\begin{eqnarray*}
&& \bE \left[ \int_S^T e^{\gamma A_t} |Y(t)|^2 \dd A_t \right] +\bE \left[ \int_S^T e^{(\gamma-\delta) A_t}  \int_t^T e^{\delta A_s} \dd  \trace \langle \mart(t,\cdot)\rangle_s    \dd A_t \right]  \\
& &\quad \leq \left( 2 + 9 + 9\delta \frac{e^{(\gamma-\delta)\mathfrak f}}{(\gamma-\delta)} \right) \bE \left[ \int_S^T e^{\gamma A_t} |\Phi(t)|^2  \dd A_t\right] \\
&&\qquad +\left( 2 + 9 + 9 \delta \frac{e^{(\gamma-\delta)\mathfrak f}}{(\gamma-\delta)} \right) \bE \left[ \int_S^T e^{\gamma A_t} |H(t)|^2  \dd A_t\right] \\
& &\quad \leq \left( 11 + 9\delta \frac{e^{(\gamma-\delta)\mathfrak f}}{(\gamma-\delta)} \right) \bE \left[ \int_S^T e^{\gamma A_t} |\Phi(t)|^2  \dd A_t\right] \\
&&\qquad +2 \left(  \frac{11}{\delta} + 9 \frac{e^{(\gamma-\delta)\mathfrak f}}{(\gamma-\delta)} \right)
 \bE \left[ \int_S^T  e^{(\gamma-\delta) A_t}  \int_t^T e^{\delta A_s}\dfrac{ |h^0(t,s)|^2}{\alpha_s^2} \dd B_s  \dd A_t \right] \\
 && \qquad + 2 \left(  \frac{11}{\delta} + 9 \frac{e^{(\gamma-\delta)\mathfrak f}}{(\gamma-\delta)} \right) \bE \left[ \int_S^T   e^{(\gamma - \delta) A_t}  \int_t^T e^{\delta A_s} \dd  \trace \langle \mart(t,\cdot)\rangle_s    \dd A_t \right].
\end{eqnarray*}
From Lemma \ref{lem:infimum}, the values $\gamma=\beta$ and $\delta = \delta^*(\beta)$ lead to the infimum $M^{\mathfrak f}(\beta)$ of 
$$\Pi^{\mathfrak f}(\gamma,\delta) =   \frac{11}{\delta} + 9 \frac{e^{(\gamma-\delta)\mathfrak f}}{(\gamma-\delta)} .$$
If $\mathfrak f$ and $\beta$ are such that $M^{\mathfrak f}(\beta) < 1/2$, then the conclusion of the lemma follows. 
\end{proof}

Moreover we have a stability result for this BSVIE \eqref{eq:special_BSVIE}.
\begin{Lemma} \label{lem:stab_BSVIE_h_poss}
 Let $(\bar \Phi,\bar h)$ be a couple of data each satisfying the above assumptions {\rm \ref{H1}}, {\rm \ref{H2}} and {\rm \ref{H3}}. Let $(\bar Y,\bar Z,\bar U,\bar M)$ a solution of the BSVIE \eqref{eq:special_BSVIE} with data $(\bar \Phi, \bar h)$. Define
$$(\fd Y,\fd Z, \fd U, \fd M) = (Y-\bar Y, Z-\bar Z,U-\bar U,M-\bar M).$$
Under the conditions of Lemma \ref{lem:param_BSDE->BSVIE} on $\mathfrak f$ and $\beta$, with the same $\delta \in (0,\beta)$, we have the following stability result:
\begin{eqnarray} \nonumber
&& \bE \left[ \int_S^T e^{ \beta A_t} |\fd Y(t)|^2 \dd A_t +  \int_S^T e^{(\beta-\delta) A_t}  \int_t^T e^{\delta A_s} \dd  \trace \langle \fd \mart(t,\cdot)\rangle_s    \dd A_t \right] \\  \nonumber
&& \quad  \leq \dfrac{\delta}{2} \Sigma^\mathfrak f( \beta) \bE \left[ \int_S^T e^{\beta A_t} |\fd \Phi(t)|^2  \dd A_t\right] \\ \label{eq:stability_H2-estimate}
&&\qquad +\Sigma^\mathfrak f(\beta)  \bE \left[ \int_S^T  e^{(\beta-\delta) A_t}  \int_t^T e^{\delta A_s}\dfrac{ |\fd h(t,s)|^2}{\alpha_s^2} \dd B_s  \dd A_t \right] 
\end{eqnarray}
where 
$$\fd h(t,r) = h(t,r,Z(t,r),U(t,r))-\bar h(t,r,Z(t,r),U(t,r)).$$
\end{Lemma}
\begin{proof}
We sketch the proof of \cite[Proposition 3.13]{papa:poss:sapl:18}. Note that 
\begin{eqnarray*}
\fd Y(t) &= &\fd \Phi(t) + \int_t^T \left( h(t,r,Z(t,r),U(t,r))-\bar h(t,r,\bar Z(t,r),\bar U(t,r))\right) \dd B_r - \int_t ^T \dd \fd \mart(t,r) \\
& = & \fd \Phi(t) +\fd H(t) - \int_t ^T \dd \fd \mart(t,r).
\end{eqnarray*}
We denote 
$$\fd h(t,r) = h(t,r,Z(t,r),U(t,r))-\bar h(t,r,Z(t,r),U(t,r)),$$
and from the assumptions on $\bar h$ we have:
\begin{eqnarray*}
&&\dfrac{1}{\alpha_r^2} |h(t,r,Z(t,r),U(t,r))-\bar h(t,r,\bar Z(t,r),\bar U(t,r)) |^2 \\
&& \quad \leq 2 \dfrac{\theta^\circ_r}{\alpha_r^2} |c_r \fd Z(t,r)|^2 + 2\dfrac{ \theta^\natural_r}{\alpha_r^2} \vertiii{ \fd U(t,r) }^2_r + 2\dfrac{ |\fd h(t,r)|^2}{\alpha_r^2} \\
&&\quad  \leq 2  |c_r \fd Z(t,r)|^2 + 2 \vertiii{ \fd U(t,r) }^2_r + 2\dfrac{ |\fd h(t,r)|^2}{\alpha_r^2}.
\end{eqnarray*}
As for the preceding lemma, we have:
\begin{eqnarray*} \nonumber
\bE \left[\int_S^T e^{\gamma A_t}|\fd H(t)|^2 \dd A_t \right] 
& \leq & \frac{2}{\delta}  \bE \left[ \int_S^T  e^{(\gamma-\delta) A_t}  \int_t^T e^{\delta A_s}\dfrac{ |\fd h(t,s)|^2}{\alpha_s^2} \dd B_s  \dd A_t \right] \\ \label{eq:estim_H_1}
&+ & \frac{2}{\delta}  \bE \left[ \int_S^T   e^{(\gamma - \delta) A_t}  \int_t^T e^{\delta A_s} \dd  \trace \langle \fd \mart(t,\cdot)\rangle_s    \dd A_t \right] .
\end{eqnarray*}
Thereby for any $\delta < \gamma \leq \beta$: 
\begin{eqnarray*}
&& \bE \left[ \int_S^T e^{\gamma A_t} |\fd Y(t)|^2 \dd A_t \right] +\bE \left[ \int_S^T e^{(\gamma-\delta) A_t}  \int_t^T e^{\delta A_s} \dd  \trace \langle \fd \mart(t,\cdot)\rangle_s    \dd A_t \right]  \\
& &\quad \leq \left( 11 + 9\delta \frac{e^{(\gamma-\delta)\mathfrak f}}{(\gamma-\delta)} \right) \bE \left[ \int_S^T e^{\gamma A_t} |\fd \Phi(t)|^2  \dd A_t\right] \\
&&\qquad +\left( 11 + 9 \delta \frac{e^{(\gamma-\delta)\mathfrak f}}{(\gamma-\delta)} \right) \bE \left[ \int_S^T e^{\gamma A_t} |\fd H(t)|^2  \dd A_t\right] \\
& &\quad \leq \left( 11 + 9\delta \frac{e^{(\gamma-\delta)\mathfrak f}}{(\gamma-\delta)} \right) \bE \left[ \int_S^T e^{\gamma A_t} |\fd \Phi(t)|^2  \dd A_t\right] \\
&&\qquad +2 \left(  \frac{11}{\delta} + 9 \frac{e^{(\gamma-\delta)\mathfrak f}}{(\gamma-\delta)} \right)
 \bE \left[ \int_S^T  e^{(\gamma-\delta) A_t}  \int_t^T e^{\delta A_s}\dfrac{ |\fd h(t,s)|^2}{\alpha_s^2} \dd B_s  \dd A_t \right] \\
 && \qquad + 2 \left(  \frac{11}{\delta} + 9 \frac{e^{(\gamma-\delta)\mathfrak f}}{(\gamma-\delta)} \right) \bE \left[ \int_S^T   e^{(\gamma - \delta) A_t}  \int_t^T e^{\delta A_s} \dd  \trace \langle \fd \mart(t,\cdot)\rangle_s    \dd A_t \right].
\end{eqnarray*}
The conclusion of the lemma follows as for the preceding lemma.
\end{proof}

Note that this stability result leads to uniqueness of the solution of the BSVIE \eqref{eq:special_BSVIE} in the space $\mathfrak{S}^2_{\delta^*(\beta)\leq \beta}(\Delta^c(0,T))$. We then deduce:
\begin{Prop} \label{prop:type_I_BSVIE_without_Y}
Under the conditions {\rm \ref{H1}} to {\rm \ref{H5}}, if the driver $f$ does not depend on $y$, and if the constants $\kappa^\mathfrak f(\delta)$ and $M^{\mathfrak f}(\beta)$ defined by \eqref{eq:cond_exist_beta_poss} and in Lemma \ref{lem:infimum} verify
\begin{equation} \label{eq:cond_beta_Type_I_BSVIE_special_case}
\kappa^\mathfrak f(\delta) < \dfrac{1}{2},\qquad M^{\mathfrak f}(\beta) < \dfrac{1}{2}, 
\end{equation}
then the BSVIE \eqref{eq:general_BSVIE_type_I} has a unique adapted solution $(Y,Z,U,M)$ in $ \mathfrak{S}^2_{\beta}(\Delta^c(0,T))$.
\end{Prop}


\subsection{Proof of Theorem \ref{thm:type_I_general_BSVIE} } 

Since we only consider a Type-I BSVIE, our arguments are close to the ones used in \cite{yong:06}. Fix $(y,\zeta,\nu,m) \in  \mathfrak{S}^2_{\beta}(\Delta^c(0,T))$ and consider the BSVIE on $[S,T]$
\begin{eqnarray} \nonumber
Y(t) &= &\Phi(t) + \int_t^{T} f(t,s, y(s),Z(t,s) , U(t,s)) \dd B_s - \int_t^{T} Z(t,s) \dd X^\circ_s  \\ \label{eq:fix_point_BSVIE}
&& \qquad - \int_t^{T}  \int_{\bR^m} U(t,s,x) \tpi^\natural(\dd s,\dd x) -\int_t^{T} \dd M(t,s),
\end{eqnarray}
We want to apply Lemma \ref{lem:param_BSDE->BSVIE}. If 
$$h(t,s,z,\psi)= f(t,s, y(s),z, \psi),$$
we need to check that 
$$  \bE \left[ \int_0^T  e^{(\beta-\delta) A_t}  \int_t^T e^{\delta A_s}\dfrac{ |h(t,s,0,\mathbf 0)|^2}{\alpha_s^2} \dd B_s  \dd A_t \right]  <+\infty.$$
The Lipschitz property \ref{H2} of $f$ leads to
$$ |h(t,s,0,\mathbf 0)|^2  \leq 2|f^0(t,s)|^2 + 2\varpi_s |y(s)|^2 .$$
Using \ref{H3} 
\begin{eqnarray*}
e^{\delta A_r} \dfrac{|h(t,r,0,\mathbf 0)|^2}{\alpha_r^2} & \leq & 2e^{\delta A_r} \dfrac{|f^0(t,r)|^2}{\alpha_r^2}  + 2 e^{\delta A_r}\dfrac{\varpi_r}{\alpha_r^2} |y(r)|^2 \\
& \leq & 2e^{\delta A_r} \dfrac{|f^0(t,r)|^2}{\alpha_r^2}  + 2 e^{\delta A_r} \alpha_r^2 |y(r)|^2. 
\end{eqnarray*}
We obtain for a.e. $t\in [S,T]$
\begin{eqnarray*}
&& \int_t^T e^{\delta A_r} \dfrac{|h(t,r,0,\mathbf 0)|^2}{\alpha_r^2} \dd B_r  \leq 2   \int_t^T e^{\delta A_r} \dfrac{|f^0(t,r)|^2}{\alpha_r^2} \dd B_r + 2  \int_t^T e^{\delta A_r} \alpha_r^2 |y(r)|^2  \dd B_r . \\
 \end{eqnarray*}
Note that from \ref{H5}, the assumption \eqref{eq:subset_times_square_int_h} holds. 
Thus the BSVIE \eqref{eq:fix_point_BSVIE} has a unique adapted solution $(Y,Z,U,M) \in \mathfrak{S}^2_{\beta}(\Delta^c(0,T))$ and from \eqref{eq:H2-estimate}
\begin{eqnarray*} \nonumber
&& \bE \left[ \int_0^T e^{ \beta A_t} |Y(t)|^2 \dd A_t +  \int_0^T e^{(\beta-\delta) A_t}  \int_t^T e^{\delta A_s} \dd  \trace \langle \mart(t,\cdot)\rangle_s    \dd A_t \right] \\  \nonumber
&& \quad  \leq \dfrac{\delta}{2} \Sigma^\mathfrak f( \beta) \bE \left[ \int_0^T e^{\beta A_t} |\Phi(t)|^2  \dd A_t\right] \\ 
&&\qquad + \Sigma^\mathfrak f(\beta)  \bE \left[ \int_0^T  e^{(\beta-\delta) A_t}  \int_t^T e^{\delta A_s}\dfrac{ |h^0(t,s)|^2}{\alpha_s^2} \dd B_s  \dd A_t \right].
\end{eqnarray*}
Using our estimate on $h$, we deduce that 
\begin{eqnarray*} \nonumber
&& \bE \left[ \int_0^T e^{ \beta A_t} |Y(t)|^2 \dd A_t +  \int_0^T e^{(\beta-\delta) A_t}  \int_t^T e^{\delta A_s} \dd  \trace \langle \mart(t,\cdot)\rangle_s    \dd A_t \right] \\ 
&&\quad \leq \dfrac{\delta}{2} \Sigma^\mathfrak f( \beta) \bE \left[ \int_0^T e^{\beta A_t} |\Phi(t)|^2  \dd A_t\right]  +2  \Sigma^\mathfrak f(\beta)  \bE \left[ \int_0^T  e^{(\beta-\delta) A_t}  \int_t^T e^{\delta A_s}\dfrac{ |f^0(t,s)|^2}{\alpha_s^2} \dd B_s  \dd A_t \right] \\
&&\qquad +2  \Sigma^\mathfrak f(\beta)  \bE \left[ \int_0^T  e^{(\beta-\delta) A_t}  \int_t^T e^{\delta A_s} |y(s)|^2  \dd A_s  \dd A_t \right] \\
&&\quad \leq \dfrac{\delta}{2} \Sigma^\mathfrak f( \beta) \bE \left[ \int_0^T e^{\beta A_t} |\Phi(t)|^2  \dd A_t\right]  +2  \Sigma^\mathfrak f(\beta)  \bE \left[ \int_0^T  e^{(\beta-\delta) A_t}  \int_t^T e^{\delta A_s}\dfrac{ |f^0(t,s)|^2}{\alpha_s^2} \dd B_s  \dd A_t \right] \\
&&\qquad +2  \Sigma^\mathfrak f(\beta) \dfrac{e^{(\beta-\delta)\mathfrak f}}{\beta-\delta} \bE \left[ \int_0^T   e^{\beta A_s} |y(s)|^2  \dd A_s \right].
\end{eqnarray*}
In other words we have defined a map $\Theta$ by $\Theta(y,\zeta,\nu,m)=(Y,Z,U,M)$, from $\mathfrak{S}^2_{\beta}(\Delta^c(0,T))$ to $\mathfrak{S}^2_{\beta}(\Delta^c(0,T))$.

Now let us consider $(y,\zeta,\nu,m)$ and $(\bar y,\bar \zeta,\bar \nu, \bar m)$ in $\mathfrak{S}^2_{\beta}(\Delta^c(0,T))$ and define $(Y,Z,U,M)$ and $(\bar Y ,\bar Z, \bar U, \bar M)$ as the solutions of the BSVIE \eqref{eq:fix_point_BSVIE} and 
$$(\fd Y,\fd Z, \fd U, \fd M) = (Y-\bar Y, Z-\bar Z,U-\bar U,M-\bar M).$$
Then from Lemma \ref{lem:stab_BSVIE_h_poss}, we know that:
\begin{eqnarray*} \nonumber
&& \bE \left[ \int_0^T e^{ \beta A_t} |\fd Y(t)|^2 \dd A_t +  \int_0^T e^{(\beta-\delta) A_t}  \int_t^T e^{\delta A_s} \dd  \trace \langle \fd \mart(t,\cdot)\rangle_s    \dd A_t \right] \\  \nonumber
&& \quad  \leq \Sigma^\mathfrak f(\beta)  \bE \left[ \int_0^T  e^{(\beta-\delta) A_t}  \int_t^T e^{\delta A_s}\dfrac{ |\fd h(t,s)|^2}{\alpha_s^2} \dd B_s  \dd A_t \right] ,
\end{eqnarray*}
where 
$$\fd h(t,r) = h(t,r,y(r),Z(t,r),U(t,r))-h(t,r,\bar y(r), Z(t,r),U(t,r)).$$
Again \ref{H2} leads to:
$$|\fd h(t,r)|^2 \leq \varpi_r(y(r)-\bar y(r))^2.$$
Thereby
\begin{eqnarray*} \nonumber
&& \bE \left[ \int_0^T e^{ \beta A_t} |\fd Y(t)|^2 \dd A_t +  \int_0^T e^{(\beta-\delta) A_t}  \int_t^T e^{\delta A_s} \dd  \trace \langle \fd \mart(t,\cdot)\rangle_s    \dd A_t \right] \\  \nonumber
&& \quad  \leq \Sigma^\mathfrak f(\beta) \dfrac{e^{(\beta-\delta)\mathfrak f}}{\beta-\delta}   \bE \left[ \int_0^T   e^{\beta A_s}(y(s)-\bar y(s))^2 \dd A_s \right] . 
\end{eqnarray*}
Since we suppose that 
$$ \widetilde\Sigma^\mathfrak f(\beta) =  \Sigma^\mathfrak f(\beta) \dfrac{e^{(\beta-\delta)\mathfrak f}}{\beta-\delta}  <1,$$
this map $\Theta$ is a contraction and thus it admits a unique fixed point $(Y,Z,U,M) \in \mathfrak{S}^2_{\beta}(\Delta^c(0,T))$ which is the unique adapted solution of \eqref{eq:general_BSVIE_type_I} on $[0,T]$. 

To prove the upper bound \eqref{eq:a_priori_estim_general_BSVIE_type_I}, let apply Lemma \ref{lem:param_BSDE->BSVIE} to the driver $h(t,s,z,u)=f(t,s,Y(s),z,u).$
Note that with \ref{H2}:
$$ |h(t,s,0,\mathbf 0)|^2  \leq \left(1+ \dfrac{1}{\eps}\right) |f^0(t,s)|^2 + (1+\eps) \varpi_s |Y(s)|^2 $$
for any $\eps > 0$. Thus
\begin{eqnarray*} \nonumber
&& \bE \left[ \int_0^T e^{ \beta A_t} |Y(t)|^2 \dd A_t +  \int_0^T e^{(\beta-\delta) A_t}  \int_t^T e^{\delta A_s} \dd  \trace \langle \mart(t,\cdot)\rangle_s    \dd A_t \right] \\  \nonumber
&& \quad  \leq \dfrac{\delta}{2} \Sigma^\mathfrak f( \beta) \bE \left[ \int_0^T e^{\beta A_t} |\Phi(t)|^2  \dd A_t\right] \\ 
&&\qquad +  \Sigma^\mathfrak f(\beta)  \bE \left[ \int_0^T  e^{(\beta-\delta) A_t}  \int_t^T e^{\delta A_s}\dfrac{ |h^0(t,s)|^2}{\alpha_s^2} \dd B_s  \dd A_t \right] \\
&& \quad  \leq \dfrac{\delta}{2} \Sigma^\mathfrak f( \beta) \bE \left[ \int_0^T e^{\beta A_t} |\Phi(t)|^2  \dd A_t\right] \\ 
&&\qquad +\left(1+ \dfrac{1}{\eps}\right) \Sigma^\mathfrak f(\beta)  \bE \left[ \int_0^T  e^{(\beta-\delta) A_t}  \int_t^T e^{\delta A_s}\dfrac{ |f^0(t,s)|^2}{\alpha_s^2} \dd B_s  \dd A_t \right] \\
&&\qquad + (1+\eps )\Sigma^\mathfrak f(\beta)  \bE \left[ \int_0^T  e^{(\beta-\delta) A_t}  \int_t^T e^{\delta A_s} |Y(s)|^2 \dd A_s  \dd A_t \right] \\
&& \quad= \dfrac{\delta}{2} \Sigma^\mathfrak f( \beta) \bE \left[ \int_0^T e^{\beta A_t} |\Phi(t)|^2  \dd A_t\right] + (1+\eps)  \widetilde \Sigma^\mathfrak f(\beta)   \bE \left[ \int_0^T   e^{\delta A_s} |Y(s)|^2\dd A_s \right]\\ 
&&\qquad +\left(1+ \dfrac{1}{\eps}\right)  \Sigma^\mathfrak f(\beta)  \bE \left[ \int_0^T  e^{(\beta-\delta) A_t}  \int_t^T e^{\delta A_s}\dfrac{ |f^0(t,s)|^2}{\alpha_s^2} \dd B_s  \dd A_t \right]  .
\end{eqnarray*}
Since $\Sigma^\mathfrak f(\beta)  \dfrac{e^{(\beta-\delta)\mathfrak f}}{\beta-\delta} <1$, choosing $\eps$ sufficiently small yields to the wanted result. 

To finish the proof, let us prove the stability estimate. Using again Lemma \ref{lem:stab_BSVIE_h_poss} with 
$$h(t,s,z,u)=f(t,s,Y(s),z,u),\quad \bar h(t,s,z,u)=\bar f(t,s,\bar Y(s),z,u),$$
yields to:
\begin{eqnarray*} \nonumber
&& \bE \left[ \int_0^T e^{ \beta A_t} |\fd Y(t)|^2 \dd A_t +  \int_0^T e^{(\beta-\delta) A_t}  \int_t^T e^{\delta A_s} \dd  \trace \langle \fd \mart(t,\cdot)\rangle_s    \dd A_t \right] \\  \nonumber
&& \quad  \leq \dfrac{\delta}{2} \Sigma^\mathfrak f( \beta) \bE \left[ \int_0^T e^{\beta A_t} |\fd \Phi(t)|^2  \dd A_t\right] \\
&&\qquad +\Sigma^\mathfrak f(\beta)  \bE \left[ \int_0^T  e^{(\beta-\delta) A_t}  \int_t^T e^{\delta A_s}\dfrac{ |\fd h(t,s)|^2}{\alpha_s^2} \dd B_s  \dd A_t \right] 
\end{eqnarray*}
where 
\begin{eqnarray*}
\fd h(t,r) & =&  h(t,r,Z(t,r),U(t,r))-\bar h(t,r,Z(t,r),U(t,r)) \\
& = & f(t,r,Y(r),Z(t,r),U(t,r))-\bar f(t,r,Y(r),Z(t,r),U(t,r)) \\
&+&  \bar f(t,r,Y(r),Z(t,r),U(t,r)) -  \bar f(t,r,\bar Y(r),Z(t,r),U(t,r)) \\
& = & \fd f(t,r) +  \bar f(t,r,Y(r),Z(t,r),U(t,r)) -  \bar f(t,r,\bar Y(r),Z(t,r),U(t,r)).
\end{eqnarray*}
Since $\bar f$ satisfies \ref{H2}, 
$$|\bar f(t,r,Y(r),Z(t,r),U(t,r)) -  \bar f(t,r,\bar Y(r),Z(t,r),U(t,r)) |^2 \leq \varpi_r |\fd Y(r)|^2.$$
Hence we obtain 
\begin{eqnarray*} \nonumber
&& \bE \left[ \int_0^T e^{ \beta A_t} |\fd Y(t)|^2 \dd A_t +  \int_0^T e^{(\beta-\delta) A_t}  \int_t^T e^{\delta A_s} \dd  \trace \langle \fd \mart(t,\cdot)\rangle_s    \dd A_t \right] \\  \nonumber
&& \quad  \leq \dfrac{\delta}{2} \Sigma^\mathfrak f( \beta) \bE \left[ \int_0^T e^{\beta A_t} |\fd \Phi(t)|^2  \dd A_t\right] \\
&&\qquad +\left( 1 + \dfrac{1}{\eps} \right) \Sigma^\mathfrak f(\beta)  \bE \left[ \int_0^T  e^{(\beta-\delta) A_t}  \int_t^T e^{\delta A_s}\dfrac{ |\fd f(t,s)|^2}{\alpha_s^2} \dd B_s  \dd A_t \right] \\
&&\qquad + (1+\eps)  \widetilde \Sigma^\mathfrak f(\beta)  \bE \left[ \int_0^T e^{\beta A_s}|\fd Y(s)|^2\dd A_s \right] .
\end{eqnarray*}
The same arguments used before lead to the conclusion.


\subsection{Comparison principle.}  \label{ssect:comp_prin_BSVIE}

In this section, the dimension $d$ is equal to one. Our goal is to extend some results contained in \cite{wang:yong:15}. Note that the comparison principle for BSDEs has been proved in \cite{krus:popi:14,krus:popi:17} in the quasi left-continuous case (see also \cite[Theorem 3.2.1]{delo:13} or \cite[Proposition 5.32]{pard:rasc:14}). In \cite[Theorem 3.25]{papa:poss:sapl:18}, the comparison principle is established for the BSDE: 
$$Y_t = \xi + \int_t^T f(s,Y_{s-},Z_s,U_s(\cdot)) \dd B_s -\int_t^T Z_s \dd X^\circ_s - \int_t^T U_s(x) \tpi^\natural(\dd s,\dd x) - \int_t^T \dd M_s.$$
Compared to the BSDE \eqref{eq:poss_BSDE}, the difference is that $f$ depends on $Y_{s-}$, instead of $Y_s$. This property is crucial in \cite{papa:poss:sapl:18}, since they have to take into account the discontinuity of $B$. Before stating the comparison principle, let us recall that a generator $f$ can be ``linearized'' as follows:
\begin{eqnarray*}
&& f(\omega,s,y,z,u_s(\omega;\cdot) - f(\omega,s,y',z',u'_s(\omega;\cdot) \\
&& \quad = \lambda_s(\omega)(y-y') + \eta_s(\omega)b_s(\omega)(z-z')^{\top} +  f(\omega,s,y,z,u(\cdot)) - f(\omega,s,y,z,u'(\cdot)).
\end{eqnarray*}
See \cite[Remark 3.24]{papa:poss:sapl:18}. Let us emphasize that $\lambda$ and $\eta$ depend also on $y,y',z,z',u,u',c$. In particular $\lambda$ and $\eta$ are not predictable if they depend on $Y_s$. This is the reason why the preceding BSDE (and not the BSDE \eqref{eq:poss_BSDE}) is studied for the comparison property in \cite{papa:poss:sapl:18}. Nonetheless to simplify the notations and when no confusion may arise, we omit this dependence. If \ref{F2} holds, $|\lambda_s(\omega)|^2 \leq \varpi_s(\omega)$ and $|\eta_s(\omega)|^2 \leq \theta^\circ_s(\omega)$ $\dd \bP \otimes \dd B$-a.e. on $\Omega\times [0,T]$. The comparison principle is the following (similar to \cite[Theorem 3.25]{papa:poss:sapl:18}). 
\begin{Prop} \label{prop:comp_pinc_BSDE}
For $i=1,2$, let $(Y^i,Z^i,U^i,M^i)$ be solutions of the BSDE \eqref{eq:poss_BSDE} with standard data $(\bar X,B,\xi^i,f^i)$, that is {\rm \ref{F1}} to {\rm \ref{F4}} hold. Assume that 
\begin{enumerate}[label={\rm \textbf{(P\arabic*)}}]
\item\label{P1} $X^\circ$ and $B$ are continuous. 
\item\label{P2} The generator $f^1$ is such that for any $(s,y,z,u,u')$ in $\bR_+\times \bR \times \bR^k \times \mathfrak H \times \mathfrak H$, there is some map $\kappa \in \bH^{2,\natural}$ with $\Delta (\kappa \star \tpi^\natural) > -1$ on $[0,T]$ such that $\dd \bP \otimes \dd B$-a.e. $(\omega,s)$, denoting $\fd u = u-u'$
$$f^1(\omega,s,y,z,u(\cdot)) - f^1(\omega,s,y,z,u'(\cdot)) \leq \widehat K_s (\fd u_s(\cdot) \kappa_s(\cdot)).$$ 
\end{enumerate}
If $\xi^1\leq \xi^2$ a.s., and $f^1(s,Y^2_s,Z_s^2,U^2_s) \leq f^2(s,Y^2_s,Z_s^2,U^2_s)$ $\dd \bP \otimes \dd B$-a.e., and if the stochastic exponential $\mathcal E(\eta \cdot X^\circ + \kappa \star \tpi^\natural)$ is a uniformly integrable martingale, then we have $\bP$-a.s.: $Y^1_t \leq Y^2_t$ for any $t \in [0,T]$.
\end{Prop}
\begin{proof}
Postponed in the appendix. 
\end{proof}

Note that the continuity of $X^\circ$ is also supposed in \cite[Theorem 3.25]{papa:poss:sapl:18}. 
Hence if we consider the BSVIE \eqref{eq:special_BSVIE} where the generator $f$ does not depend on $y$, $\zeta$ and $\theta$:
 \begin{eqnarray*} \nonumber
Y(t) &= &\Phi(t) + \int_t^{T} f(t,s, Z(t,s) ,U(t,s)) \dd B_s - \int_t^{T} Z(t,s) \dd X^\circ_s \\ 
&& \qquad - \int_t^{T}  \int_{\bR^m}  U(t,s,x) \tpi(\dd x, \dd s) -\int_t^{T} \dd M(t,s).
\end{eqnarray*}
the comparison principle holds. 
\begin{Prop} \label{prop:comp_BSVIE_without_Y}
For $i=1,2$, let $f^i:\Omega \times \Delta^c(0,T) \times \bR^{k} \times \mathfrak H \to \bR$ satisfy {\rm \ref{H2}} and {\rm \ref{H3}}. Moreover a.s. 
$$f^1(t,s,z,u)\leq f^2(t,s,z,u), \quad \forall (t,s,z,u) \in \Delta^c(0,T) \times \bR^{k} \times  \mathfrak H.$$
We assume that {\rm \ref{P1}} and {\rm \ref{P2}} hold. 
Then for any $\Phi^i \in \bL^p_{\cF_T}(0,T)$ with $\Phi^0(t) \leq \Phi^1(t)$ a.s., $t\in[0,T]$, the solutions $(Y^i,Z^i,\psi^i,M^i) $ of \eqref{eq:special_BSVIE} verify
$$Y^1(t) \leq Y^2(t), \quad \mbox{a.s., } t\in [0,T].$$
\end{Prop}
\begin{proof}
Let us consider $\lambda^i(t,\cdot)$ solution of the parametrized BSDE \eqref{eq:parametrized_BSDE} with data $(\Phi^i,f^i)$. From Proposition \ref{prop:comp_pinc_BSDE}, we obtain that a.s. for any $s\in [t,T]$, $\lambda^1(t,s) \leq \lambda^2(t,s)$. Sending $s$ to $t$, since $Y^i(t) = \lambda^i(t,t)$, we obtain the desired result.
\end{proof}

Nevertheless to extend this result for generators depending also on $y$, in \cite[Theorem 3.4]{wang:yong:15}, $f$ is supposed to be bounded (from above or from below) by a non decreasing w.r.t. $y$ generator. The next proposition extends this result to our setting. 
\begin{Prop} \label{prop:comp_BSVIE_1}
We assume that the setting of Theorem \ref{thm:type_I_general_BSVIE} holds and we consider two generators $f^i : \Omega \times \Delta^c(0,T) \times \bR \to \bR$, $i=1,2$, satisfying {\rm \ref{H2}} and {\rm \ref{H3}}. 
We suppose that a.s. for a.e. $s \in [0,T]$ and for any $0\leq t  \leq s$ and any $(y,z,u) \in \bR \times \bR^k \times \mathfrak H$: 
\begin{equation} \label{eq:cond_comp_BSVIE_1_bis}
f^1(t,s,y,z,u)  \leq \bar f(t,s,y,z,u)  \leq f^2(t,s,y,z,u) .
\end{equation}
The driver $\bar f$ also verifies {\rm \ref{H2}}-{\rm \ref{H3}} and $y\mapsto \bar f(t,s,y,z,u)$ is nondecreasing. 
If a.s. for $0\leq t \leq T$, $\Phi^2(t) \geq \Phi^1(t)$, then the corresponding solutions of the BSVIEs \eqref{eq:general_BSVIE_type_I} with generator $f^i$, verify for any $t\in[0,T]$:
$$Y^2(t) \geq  Y^1(t) ,\quad  a.s.$$
\end{Prop}
\begin{proof}
Since the arguments of the proof are almost the same as in \cite{wang:yong:15}, we postpone the details in the appendix. 
\end{proof}

If we cannot ``separate'' the drivers $f^1$ and $f^2$ by a nondecreasing generator $\bar f$, we restrict ourselves to half-linear generators as in \cite[Theorem 3.9]{wang:yong:15}. Hence we suppose that the generator $f$ is linear w.r.t. $z$ and $\psi$:
\begin{equation}\label{eq:BSVIE_linear_gene}
f(t,s,y,z,u) = g(t,s,y) + h(s)b_s z +  \widehat K_s(u(\cdot)\kappa_s(\cdot)),
\end{equation}
where $h$ is a process bounded by $\theta^\circ$ and $\kappa : \Omega \times [0,T] \times\bR^m \to \bR$ is progressively measurable and such that \ref{P2} holds. 
%
%
%
%
Our comparison result is a extension to the jump case of \cite[Theorems 3.8 and 3.9]{wang:yong:15} for BSVIE \eqref{eq:general_BSVIE_type_I} where $f$ is given by \eqref{eq:BSVIE_linear_gene}. The main difference comes from the free terms. Indeed in \cite{wang:yong:15} where $B_t= t$, the free term $\Phi$ is supposed to be in $C_{\cF_T}([0,T],\bL^2(\Omega))$ (see the functional spaces defined in Section \ref{ssect:time_regularity}). Hence for any partition $\Pi= \{ t_k, \ 0\leq k \leq N\}$ of $[0,T]$, if
\begin{equation} \label{eq:approx_Phi_partition}
\Phi^\Pi(t) = \sum_{k=1}^N \Phi(t_{k-1})  \mathbf 1_{(t_{k-1},t_k]}(t),
\end{equation}
then by uniform continuity, there exists a modulus of continuity $\rho$ such that 
$$\sup_{t\in [0,T]} \bE \left[ \left| \Phi^\Pi(t)-\Phi(t)\right|^2 \right] \leq \sup_{|t-t'|\leq \|\Pi\|} \bE \left[ \left| \Phi(t')-\Phi(t)\right|^2 \right] \leq \rho(\|\Pi\|).$$
$\|\Pi\|$ is the mesh of the partition. In our setting, we cannot separate $t$ and $\Omega$, since $B$ is random. 


\begin{Prop} \label{prop:comp_BSVIE_2}
Consider two drivers $g^i : \Omega \times \Delta^c(0,T) \times \bR \to \bR$ satisfying {\rm \ref{H2}}. 
We suppose that for $\dd \bP \otimes \dd B$-a.e. $(\omega,s) \in \Omega \times [0,T]$ and for any $0\leq t \leq  \tau \leq s$ and any $y\in \bR$: 
\begin{equation} \label{eq:cond_comp_BSVIE_1}
g^2(t,s,y)-g^1(t,s,y)  \geq g^2(\tau,s,y)- g^1(\tau,s,y) \geq 0.
\end{equation}
Moreover for either $i=1$ or $i=2$
\begin{equation} \label{eq:cond_comp_BSVIE_2}
(g^i(t,s,y) - g^i(t,s,y'))(y-y') \geq (g^i(\tau,s,y) -g^i (\tau,s,y'))(y-y')
\end{equation}
again for $\dd \bP \otimes \dd B$-a.e. $(\omega,s) \in \Omega \times [0,T]$ and for any $0\leq t \leq  \tau \leq s$ and any $y,y'$ in $\bR$. Furthermore there exists a continuous nondecreasing function $\rho:[0,T] \to [0,+\infty)$ with $\rho(0)=0$ such that a.s. for a.e. $s \in [0,T]$ and for any $0\leq t,t' \leq s$
\begin{equation} \label{eq:cond_comp_BSVIE_3}
|g^i(t,s,y)-g^i(t,s,y') - g^i(t',s,y)+ g^i (t',s,y')|\leq \rho(|t-t'|)\times |y-y'|.
\end{equation}
We suppose that the difference of the free terms $\fd \Phi = \Phi^2-\Phi^1$ satisfies: 
\begin{equation}\label{eq:cond_comp_BSVIE_5}
\lim_{\|\Pi\| \to 0} \bE \left[ \int_0^T e^{\beta A_t} \left( \fd\Phi^\Pi(t) - \fd \Phi(t)\right) \dd A_t \right]=0,
\end{equation}
where $\fd \Phi^\Pi$ is defined by \eqref{eq:approx_Phi_partition}. 
If a.s. for $0\leq t \leq \tau \leq T$, 
\begin{equation} \label{eq:cond_comp_BSVIE_4}
\Phi^2(t) - \Phi^1(t) \geq \Phi^2(\tau)-\Phi^1(\tau) \geq 0,
\end{equation}
then the corresponding solutions of the BSVIEs \eqref{eq:general_BSVIE_type_I} with generator $f^i$ given by \eqref{eq:BSVIE_linear_gene} with $g^i$ instead of $g$, verify for $\dd \bP \otimes \dd B$-a.e. $(\omega,t) \in \Omega \times [0,T]$
$$Y^2(t) \geq Y^1(t).$$
\end{Prop}
\begin{proof}
Let us first copy the arguments of the proof of \cite[Theorem 3.9]{wang:yong:15}.
Suppose that $g^1$ is differentiable and \eqref{eq:cond_comp_BSVIE_2} holds for $i=1$. Then we have
\begin{eqnarray} \nonumber
&&Y^2(t) - Y^1(t) = \Phi^2(t) - \Phi^1(t) + \int_t^{T}  \left[ g^2(t,s, Y^2(s)) -  g^1(t,s,Y^2(s)) \right] \dd B_s \\ \nonumber
&&\quad + \int_t^{T} L(t,s) (Y^2(s)-Y^1(s)) \dd B_s +  \int_t^{T}  h(s)b_s (Z^2(t,s)-Z^1(t,s))  \dd X^\circ_s  \\ \nonumber
&& \quad +\int_t^{T}  \int_{\bR^m}  \kappa(s,x) (U^2(t,s,x)- U^1(t,s,x))K_s (\dd x) \dd B_s\\ \nonumber
&&\quad - \int_t^{T}  (Z^2(t,s)-Z^1(t,s)) \dd X^\circ_s- \int_t^{T}  \int_{\bR^m}  (\psi^2(t,s,x)- \psi^1(t,s,x)) \tpi^\natural(\dd s,\dd x)\\  \label{eq:comp_BSVIE_proof}
&&\quad - \int_t^{T}  \dd (M^2(t,s)- M^1(t,s))
\end{eqnarray}
where
$$L(t,s) = \frac{g^1(t,s,Y^2(s))-g^1(t,s,Y^1(s))}{Y^2(s)- Y^1(s)} \mathbf 1_{Y^2(s) \neq Y^1(s)} 
, \quad 0\leq t \leq s \leq T.$$
From \ref{H2}, $L$ is bounded by $\varpi(\omega, t,s)$. 
In other words we want to prove that the solution of the BSVIE
\begin{eqnarray*} \nonumber
&&\fd Y(t)  = \fd \Phi(t) + \int_t^{T} L(t,s) \fd Y(s) \dd B_s - \int_t^{T}  \dd \fd M(t,s)\\ \nonumber
&&\quad  +  \int_t^{T}  h(s) b_s \fd Z(t,s)  \dd B_s - \int_t^{T}  \fd Z(t,s) \dd X^\circ_s \\ 
&& \quad +\int_t^{T} \widehat K_s(\kappa(s,\cdot)  \fd U(t,s,\cdot)) \dd B_s - \int_t^{T}  \int_{\bR^m} \fd U(t,s,x) \tpi^\natural(\dd s,\dd x)
\end{eqnarray*}
satisfies: for any $t\in [0,T]$, a.s. $\fd Y(t) \geq 0$. Here 
$$\fd \Phi(t) = \Phi^2(t) - \Phi^1(t) + \int_t^{T}  \left[ g^2(t,s, Y^2(s)) -  g^1(t,s,Y^2(s)) \right] \dd B_s.$$
From our assumptions \eqref{eq:cond_comp_BSVIE_1}, \eqref{eq:cond_comp_BSVIE_2} and \eqref{eq:cond_comp_BSVIE_4}, for $0\leq t \leq \tau \leq T$, $\fd \Phi(t) \geq \fd \Phi(\tau) \geq 0$, $L(t,s) - L(\tau,s) \geq 0$ and from \eqref{eq:cond_comp_BSVIE_3}
$$|L(t,s)-L(t',s)| \leq  \rho(|t-t'|).$$

Now we can follow the proof of \cite[Theorem 3.8]{wang:yong:15}. We consider a partition $\Pi = \{ t_k, \ 0\leq k \leq N\}$ of $[0,T]$ and assume first that 
$$L^\Pi(t,s) = \sum_{k=1}^N L(t_{k-1},s) \mathbf 1_{(t_{k-1},t_k\wedge s]}(t), \quad \Phi^\Pi(t) = \sum_{k=1}^N \phi_k  \mathbf 1_{(t_{k-1},t_k]}(t)$$
where $L^\Pi$ still satisfies $L^\Pi(t,s) - L^\Pi(\tau,s) \geq 0$ and $\phi_k$ are $\cF_T$-measurable r.v. such that 
$$\phi_1 \geq \phi_2\geq \ldots \geq \phi_{N-1} \geq \phi_N \geq 0.$$
Let $ (Y^\Pi(\cdot),Z^\Pi(\cdot,\cdot),U^\Pi(\cdot,\cdot),M^\Pi(\cdot,\cdot))$ be the solution of the BSVIE:
\begin{eqnarray} \nonumber
Y^\Pi(t)  &=&  \Phi^\Pi(t) + \int_t^{T} \left[ L^\Pi(t,s) Y^\Pi(s) + h(s) b_s Z^\Pi(t,s)    +  \widehat K_s(\kappa(s,\cdot) U^\Pi(t,s,\cdot) ) \right] \dd B_s \\ \label{eq:comp_BSVIE_proof_2}
& -& \int_t^{T} Z^\Pi(t,s) \dd X^\circ_s  - \int_t^{T}  \int_{\bR^m} U^\Pi(t,s,x)   \tpi^\natural(\dd s,\dd x) - \int_t^{T}  \dd M^\Pi(t,s).
\end{eqnarray}
If we consider the BSDE 
\begin{eqnarray*}
Y_N(t) &=& \phi_N + \int_t^T \left[ L(t_{N-1},s) Y_N(s)  + h(s) b_s Z_N(s) +  \widehat K_s(\kappa(s,\cdot) U_N(s,\cdot) ) \right] \dd B_s \\
& - & \int_t^{T} Z_N(s) \dd X^\circ_s -  \int_t^{T}  \int_{\bR^m} U_N(s,x)   \tpi^\natural(\dd s,\dd x)  - \int_t^T \dd M_N(s),
\end{eqnarray*}
then for $t_{N-1} < t \leq s \leq T$, 
$$(Y_N(s),Z_N(s),U_N(s,e),M_N(s)) = (Y^\Pi(s),Z^\Pi(t,s),U^\Pi(t,s,e),M^\Pi(t,s))$$
solves the BSVIE \eqref{eq:comp_BSVIE_proof_2} on the interval $(t_{N-1},t_N]$. By uniqueness of the solution and the comparison principle for BSDE, we have a.s. for any $s \in (t_{N-1},t_N]$, $Y^\Pi(s) = Y_N(s) \geq 0$. Since all martingales are c\`adl\`ag processes and since $B$ is continuous (Hypothesis \ref{P1}), 
\begin{eqnarray*}
Y^\Pi(t_{N-1}^+) &=& \phi_N + \int_{t_{N-1}}^T \left[ L(t_{N-1},s) Y^\Pi(s)  + h(s)b_s Z_N(s) +  \widehat K_s(\kappa(s,\cdot) U_N(s,\cdot) ) \right] \dd B_s \\
& - & \int_{t_{N-1}}^{T} Z_N(s)  \dd X^\circ_s -  \int_{t_{N-1}}^{T}  \int_{\bR^m} U_N(s,x)   \tpi^\natural(\dd s,\dd x)  - \int_{t_{N-1}}^T \dd  M_N(s) \geq 0.
\end{eqnarray*}
Now the BSVIE on $(t_{N-2},t_{N-1}]$ can be written as follows:
\begin{eqnarray*} \nonumber
&&Y^\Pi(t)  = \phi_{N-1}  + \int_t^{T}  \left[ L(t_{N-2},s) Y^\Pi(s) + h(s) b_s Z^\Pi(t,s) + \widehat K_s(\kappa(s,\cdot) U^\Pi(t,s,\cdot) ) \right]  \dd B_s\\
&& \quad - \int_t^{T}  Z^\Pi(t,s) \dd X^\circ_s    - \int_t^{T} \int_{\bR^m} U^\Pi(t,s,x)   \tpi^\natural(\dd s,\dd x)- \int_t^T \dd M^\Pi(t,s) \\
&& =  \phi_{N-1} -\phi_N + Y^\Pi(t_{N-1}^+) + \int_{t_{N-1}}^{T} \left[ L(t_{N-2},s)- L(t_{N-1},s)\right] Y^\Pi(s) \dd B_s  \\
&&\quad +  \int_{t_{N-1}}^T  h(s) b_s \left[ Z^\Pi(t,s)-Z_N(s)\right]  \dd B_s - \int_{t_{N-1}}^T (Z^\Pi(t,s)-Z_N(s)) \dd X^\circ_s  \\
&& \quad +\int_{t_{N-1}}^T \widehat K_s\left( \kappa(s,\cdot) (U^\Pi(t,s,\cdot) -U_N(s,\cdot)) \right) \dd B_s \\
&&\quad - \int_{t_{N-1}}^T \int_{\bR^m}  \left[ U^\Pi(t,s,x)     -U_N(s,x)\right]   \tpi^\natural(\dd s,\dd x)- \int_{t_{N-1}}^T \dd (M^\Pi(t,s)-M_N(s)) \\
&&\quad +\int_t^{t_{N-1}} \left[ L(t_{N-2},s) Y^\Pi(s) +  h(s) b_s Z^\Pi(t,s)+ \widehat K_s(\kappa(s,\cdot) U^\Pi(t,s,\cdot) ) \right] \dd B_s   \\
&& \quad - \int_t^{t_{N-1}} Z^\Pi(t,s) \dd X^\circ_s - \int_t^{t_{N-1}} \int_{\bR^m} U^\Pi(t,s,x)   \tpi^\natural(\dd s,\dd x)- \int_t^{t_{N-1}} \dd M^\Pi(t,s) .
\end{eqnarray*}
We consider the terminal condition
$$\xi_N =  \phi_{N-1} -\phi_N + Y^\Pi(t_{N-1}^+) + \int_{t_{N-1}}^{T} \left[ L(t_{N-2},s)- L(t_{N-1},s)\right] Y^\Pi(s) \dd B_s $$
and the solution $(\widetilde Y_N,\widetilde Z_N,\widetilde U_N,\widetilde M_N)$ of the linear BSDE on $[t_{N-1},T]$:
\begin{eqnarray*} \nonumber
&&\widetilde Y_N(t)  =  \xi_N  +  \int_{t}^T\left[  h(s)b_s \widetilde Z_N(s) + \widehat K_s(\kappa(s,\cdot)  \widetilde U_N(t,s,\cdot) ) \right] \dd B_s \\
&& \quad  - \int_{t}^T\widetilde Z_N(s) \dd X^\circ_s  - \int_{t}^T   \int_{\bR^m} \widetilde U_N(t,s,x)   \tpi^\natural(\dd s,\dd x)- \int_{t}^T \dd\widetilde M_N(s).
\end{eqnarray*}
By our conditions, $\xi_N$ is non-negative and thus a.s. $\widetilde Y_N(t) \geq 0$ on $[t_{N-1},T]$. By uniqueness of adapted solutions to the BSVIE, we have 
$$Z^\Pi(t,s) = Z_N(s) + \widetilde Z_N(s), \quad U^\Pi(t,s) = U_N(s) + \widetilde U_N(s), \quad M^\Pi(t,s) = M_N(s) + \widetilde M_N(s)$$
for $(t,s) \in (t_{N-2},t_{N-1}] \times (t_{N-1},t_{N}]$ and our previous BSVIE becomes
\begin{eqnarray*} \nonumber
&&Y^\Pi(t)  = \widetilde Y_N(t_{N-1}) +\int_t^{t_{N-1}}  \left[ L(t_{N-2},s) Y^\Pi(s) +  h(s) b_s Z^\Pi(t,s)+ \widehat K_s(\kappa(s,\cdot) U^\Pi(t,s,\cdot) ) \right]  \dd B_s \\
&&\quad - \int_t^{t_{N-1}} Z^\Pi(t,s) \dd X^\circ_s - \int_t^{t_{N-1}} \int_{\bR^m} U^\Pi(t,s,x)   \tpi^\natural(\dd s,\dd x)- \int_t^{t_{N-1}} \dd M^\Pi(t,s).
\end{eqnarray*}
Again we solve the BSDE:
\begin{eqnarray*}
&&Y_{N-1}(t) = \widetilde Y_N(t_{N-1})\\
&&\quad + \int_t^{t_{N-1}} \left[ L(t_{N-2},s) Y_{N-1}(s) +  h(s) b_s Z_{N-1}(t,s)+ \widehat K_s(\kappa(s,\cdot) U_{N-1}(t,s,\cdot) ) \right]  \dd B_s  \\
&&\quad -  \int_t^{t_{N-1}} Z_{N-1}(t,s) \dd X^\circ_s - \int_t^{t_{N-1}} \int_{\bR^m} U_{N-1}(t,s,x)   \tpi^\natural(\dd s,\dd x)- \int_t^{t_{N-1}} \dd M_{N-1}(t,s)
\end{eqnarray*}
on $[t_{N-2},t_{N-1}]$ and by uniqueness and the comparison principle for BSDE, we have 
$$Y^\Pi(t) = Y_{N-1}(t) \geq 0, \quad t\in [t_{N-2},t_{N-1}].$$
By induction we obtain that $Y^\Pi(t) \geq 0$, $t\in [0,T]$. 

From Theorem \ref{thm:type_I_general_BSVIE}, the stability estimate for BSVIE yields to:
\begin{eqnarray*} \nonumber
&& \bE \left[ \int_0^T e^{ \beta A_t} |Y^\Pi(t)-Y(t)|^2 \dd A_t \right] \\  \nonumber
&& \quad  \leq \mathfrak C^\mathfrak f(\beta) \bE \left[ \int_0^T e^{\beta A_t} |\Phi^\Pi(t)-\Phi(t)|^2  \dd A_t\right] \\
&&\qquad +\mathfrak C^\mathfrak f(\beta) \bE \left[ \int_0^T  e^{(\beta-\delta) A_t}  \int_t^T e^{\delta A_s}\dfrac{ |L^\Pi(t,s)-L(t,s)|^2|Y(s)|^2}{\alpha_s^2} \dd B_s  \dd A_t \right] . 
\end{eqnarray*}
The time regularity condition on $L$ implies that: $ |L^\Pi(t,s)-L(t,s)|^2 \leq  \rho(\|\Pi\|)^2.$ From our condition \eqref{eq:cond_comp_BSVIE_5}, we deduce that 
$$\lim_{\|\Pi\| \to 0}  \bE \left[ \int_0^T e^{ \beta A_t} |Y^\Pi(t)-Y(t)|^2 \dd A_t \right] = 0.$$
The conclusion of the proposition follows and this achieves the proof. 
\end{proof}

\section{More properties in the It\^o setting} \label{sect:Ito_setting_BSVIE}

In this section, we use the setting developed in the subsection \ref{ssect:Ito_setting} and the goal is to adapt some results of \cite{yong:08}. 

In this It\^o setting, the space $\mathfrak S^2(\Delta^c(0,T))$ can be easier defined. We adapt the notations of the sections \ref{ssect:poss_BSDE}, \ref{ssect:Ito_setting} and \ref{sect:general_Type_I_BSVIE} (essentially $\beta=0$), which leads to the same notations as in \cite{yong:08}. Hence we skip some details (see \cite[Section 2.1]{yong:08} for interesting readers). For any $p,q$ in $[0,+\infty)$, $H=\bR^d$ or $\bR^{d\times k}$, and $S \in [0,T]$, 
\begin{eqnarray*}
\bL^p_{\cF_S}(\Omega) &= &\{ \xi : \Omega \to H , \ \xi \ \mbox{is } \cF_S-\mbox{measurable, } \bE (|\xi|^p) < +\infty\} ,\\
\bL^p_{\cF_S}(\Omega;\bL^q(0,T)) &= &\Bigg\{ \phi : (0,T)\times \Omega \to H , \ \cB([0,T])\otimes\cF_S-\mbox{measurable with } \\
&&\qquad \qquad \left. \bE \left( \int_0^T |\phi(t)|^q \dd t\right)^{\frac{p}{q}} < +\infty \right\},\\
\bL^q_{\cF_S}(0,T;\bL^p_{\cF_S}(\Omega)) &= &\Bigg\{ \phi : (0,T)\times \Omega \to H , \ \cB([0,T])\otimes\cF_S-\mbox{measurable with } \\
&&\qquad \qquad \left.\int_0^T  \left(\bE  |\phi(t)|^p \dd t\right)^{\frac{q}{p}} < +\infty \right\},
\end{eqnarray*}
We identify 
$$\bL^p_{\cF_S}(\Omega;\bL^p(0,T)) = \bL^p_{\cF_S}(0,T;\bL^p_{\cF_S}(\Omega)) = \bL^p_{\cF_S}(0,T).$$
For $p=2$, this space corresponds to $\bL^2_{0,\cF_S}(0,T)$ of Section \ref{sect:general_Type_I_BSVIE}. 
The cases where $p$ or $q$ are equal to $\infty$ can be defined in a similar way. When adaptiveness is required, we replace the subscript $\cF_S$ by $\bF$. The above spaces are for the free term $\Phi(\cdot)$ (for which the $\bF$-adaptiveness is not required) and for $Y(\cdot)$ (for which $\bF$-adaptiveness is required). Sometimes we also use the subscript $\cP$ if we require predictability. 

To control the martingale terms in the BSVIE, we need other spaces. 
We define for any $p,q\geq 1$
$$\bL^q(S,T;\cH^p(S,T))$$
the set of processes $M(\cdot,\cdot)$ such that for almost all $t\in [0,T]$, $M(t,\cdot)$ belongs to $\cH^p(S,T)$ and 
$$\int_S^T \left[ \bE \left( \langle M(t,\cdot)\rangle_{S,T} \right)^\frac{p}{2} \right]^{\frac{q}{p}} \dd t < +\infty.$$
For $p=q=2$, if $M$ is restricted to $\Delta^c(S,T)$, this space is denoted $\cH^2_{0\leq 0}(\Delta^c(S,T))$ in Section \ref{sect:general_Type_I_BSVIE}. 
In the particular case where $M(t,\cdot) = \int_S^\cdot Z(t,s) dW_s$, then $M \in \bL^q(S,T;\cH^p(S,T))$ is equivalent to 
$$Z \in \bL^q(S,T;\bL^p_\cP(\Omega;\bL^2(S,T)))=\bL^q(S,T;\bH^p(S,T)),$$ 
that is $Z$ belongs to the set of all processes $Z : [S,T]^2 \times \Omega \to \bR^k$ such that for almost all $t \in [S,T]$, $Z(t,\cdot) \in \bH^p(S,T)= \bL^p_{\cP}(\Omega;\bL^2(S,T))$ satisfying
$$\int_S^T \left[ \bE \left( \int_S^T |Z(t,s)|^2 \dd s \right)^{\frac{p}{2}} \right]^{\frac{q}{p}} \dd t < +\infty.$$
Again for $p=q=2$, we obtain the space $\cH^{2,\circ}_{0\leq 0}$ of Section \ref{sect:general_Type_I_BSVIE}.

Let us also consider the case: 
\begin{equation*} 
N(t,s) =   \int_S^s  \int_{\bR^m} \psi(t,u,x) \tpi(\dd x, \dd u), \ t\ge S, \ s\geq S.
\end{equation*}
The process $N \in \bL^q(S,T;\cH^p(S,T))$ if and only if $\psi \in \bL^q(S,T;\bL^p_\pi(S,T)),$
namely $\psi$ is in the set of all processes $\psi : [S,T]^2 \times \bR^m \to \bR^d$ such that for almost all $t \in [S,T]$, $\psi(t,\cdot,\cdot) \in \bL^p_\pi(S,T)$ verifies
$$\int_S^T \left[ \bE \left( \int_S^T  \int_{\bR^m} |\psi(t,s,x)|^2 \pi(\dd s, \dd x)\right)^{\frac{p}{2}} \right]^{\frac{q}{p}} \dd t < +\infty.$$
Let us emphasize that for $p=q=2$, this space corresponds to $\cH^{2,\natural}_{0\leq 0}$ in Section \ref{sect:general_Type_I_BSVIE}. 

If the martingale $\mart$ is defined by \eqref{eq:def_mart_sharp}, 
due to the orthogonality of the components of $\mart$, for any $p>1$, there exist two universal constants $c_p$ and $C_p$ such that 
\begin{eqnarray*}
&&c_p \bE \left[ \left( \int_S^T |Z(t,s)|^2 \dd  s \right)^{\frac{p}{2}} +  \left( \int_S^T  \int_{\bR^m} |U(t,s,x)|^2 \pi(\dd s, \dd x) \right)^{\frac{p}{2}} +\left( \left\langle M(t,\cdot) \right\rangle_{S,T} \right)^\frac{p}{2}\right] \\
&& \quad \leq \bE \left[ \left(  \left\langle \mart(t,\cdot)  \right\rangle_{S,T} \right)^\frac{p}{2} \right] \\
&&\quad \leq C_p \bE \left[ \left( \int_S^T |Z(t,s)|^2 \dd s \right)^{\frac{p}{2}} +  \left( \int_S^T  \int_{\bR^m} |U(t,s,x)|^2 \pi(\dd s, \dd x) \right)^{\frac{p}{2}} +\left(  \left\langle M(t,\cdot) \right\rangle_{S,T} \right)^\frac{p}{2}\right].
\end{eqnarray*}
And $\mart$ belongs to $\bL^q(S,T;\cH^p(S,T))$ if and only if the triplet $(Z,U,M)$ is in $\bL^q(0,T;\bH^p(0,T)) \times \bL^q(0,T;\bL^p_\pi(0,T)) \times\bL^q(0,T;\cH^{p,\perp}(0,T))$.

Finally we define
\begin{eqnarray*}
\cM^p(0,T)&=& \bL^p(0,T;\bH^p(0,T)) \times \bL^p(0,T;\bL^p_\pi(0,T)) \times\bL^p(0,T;\bM^{p,\perp}(0,T))\\
\mathfrak S^p(0,T) &=& \bL^p_\bF(0,T)\times \bL^p(0,T;\bH^p(0,T)) \times \bL^p(0,T;\bL^p_\pi(0,T)) \times\bL^p(0,T;\cH^{p,\perp}(0,T))\\
& =&  \bL^p_\bF(0,T)\times\cM^p(S,T).
\end{eqnarray*}
with the naturally induced norm.

The definitions \ref{def:BSVI_adapted_sol} (adapted solutions) and \ref{def:BSVI_M_sol} (M-solutions) remain unchanged, except that $(Y,Z,U,M)$ belongs to $\mathfrak S^p(0,T)$ and the condition \eqref{eq:M_sol_def} becomes:
\begin{equation*} 
Y(t) = \bE \left[ Y(t) |\cF_S\right] + \int_S^t Z(t,s) \dd W_s + \int_S^t \int_{\bR^m} U(t,s,x)\tpi(\dd s,\dd x) + \int_S^t \dd M(t,s). 
\end{equation*}
%
If $(Y,Z,U,M)$ is a solution of the BSDE \eqref{eq:Ito_BSDE} in $\bD^p(0,T)$, then by the martingale representation (\cite[Lemma III.4.24]{jaco:shir:03}), we have
$$Y(t) = \bE (Y(t)) + \int_0^t \zeta(t,s)\dd W_s + \int_0^t \int_{\bR^m} \upsilon(t,s,x) \tpi(\dd x,\dd s) + m(t,s),$$
where $(\zeta,\upsilon,m) \in \cM^p(0,T)$. Thus we define
$$(Z(t,s),U(t,s,x),M(t,s)) = \left\{ \begin{array}{ll}(\zeta(t,s),\upsilon(t,s,x),m(t,s)) , & (t,s) \in \Delta(0,T) \\ (Z(s),U(s,x),M(s)), & (t,s) \in \Delta^c(0,T) \end{array} \right.$$
and we get that $(Y,Z,U,M)$ is an adapted M-solution of the BSVIE \eqref{eq:Ito_general_BSVIE} on $[0,T]$, and in fact it is the unique solution (see Theorem \ref{thm:Hp_M_solution}).

\medskip
To complete this presentation, let us recall some facts concerning the Poisson integral. From the \textit{Burkholder-Davis-Gundy inequality} (see \cite[Theorem 48]{prot:04}), for all $p\in [1,\infty)$ there exist two universal constants $c_p$ and $C_p$ (not depending on $M$) such that for any c\`adl\`ag $\bF$-martingale $M(\cdot)$ and for any $T \geq 0$
\begin{equation} \label{eq:BDG_ineq}
c_p \bE \left( \langle M\rangle_T^{p/2} \right) \leq \bE \left[ \left(\sup_{t \in [0,T]} |M(t)| \right)^{p} \right] \leq C_p \bE \left( \langle M\rangle_T^{p/2} \right). 
\end{equation}
In particular \eqref{eq:BDG_ineq} means that the Poisson martingale $N$ is well-defined on $[0,T]$ (see Chapter II in \cite{jaco:79}) provided we can control the expectation of $\langle N\rangle_T^{p/2}$ for some $p\geq 1$.

From the {\it Bichteler-Jacod inequality} (see for example \cite{mari:rock:14}), we distinguish the two cases: $p\geq 2$ and $p< 2$. 
\begin{itemize}
\item Assume that $p\geq 2$. If $\bE (\langle N\rangle_T^{p/2}) < +\infty$, then $\bP \otimes \mbox{Leb}$-a.s. on $\Omega \times [0,T]$, $\psi(t,\cdot)$ is in $\bL^2_\mu$. Hence the generator of our BSVIE can be defined on $\bL^2_\mu$. 
\item But if $p<2$, $\bP \otimes \mbox{Leb}$-a.s. on $\Omega \times [0,T]$, $\psi(t,\cdot)$ is in $\bL^p_\mu+\bL^2_\mu$ if again $\bE (\langle N\rangle_T^{p/2}) < +\infty$. Moreover $\psi_t$ is also in $\bL^1_\mu+\bL^2_\mu$. Thereby for $p<2$, our generator is be defined on $\bL^1_\mu+\bL^2_\mu$ (for the definition of the sum of two Banach spaces, see for example \cite{krei:petu:seme:82}).
\end{itemize}
See \cite[Section 1]{krus:popi:17} for details on this point. In particular for $N$ defined by
\begin{equation*}
N_t =   \int_0^t  \int_{\bR^m} \psi_s(x) \tpi(\dd s, \dd x), t\ge 0,
\end{equation*}
if $p\geq 2$, there exist two universal constants $\kappa_p$ and $K_p$ such that 
\begin{equation} \label{eq:BJ_ineq_p_=_2}
\kappa_p\left[ \bE \left( \langle N \rangle^{p/2}_T \right) \right] \leq \bE \left( \int_0^T \|\psi_t\|^2_{\bL^2_\mu} \dd t\right)^{p/2} \leq K_p  \left[ \bE \left( \langle N\rangle^{p/2}_T \right) \right]. 
\end{equation}
But if $1< p < 2$, we only have the existence of a universal constant $K_{p,T}$ such that 
\begin{equation} \label{eq:norm_equiv}
\bE \left[  \int_0^T  \|\psi_s\|^p_{\bL^p_\mu+\bL^2_\mu} \dd s\right] \leq K_{p,T}   \bE \left( \langle N\rangle_T^{p/2} \right) .
\end{equation}
And it holds that $\bL^p_\mu+\bL^2_\mu \subset \bL^1_\mu+\bL^2_\mu$.

\subsection{Solution in $\mathfrak S^p$ ($p>1$) for Type-I BSVIEs}

Here we consider the Type-I BSVIE \eqref{eq:second_special_BSVIE}:
\begin{eqnarray*} \nonumber
Y(t) &= &\Phi(t) + \int_t^{T} f(t,s, Y(s), Z(t,s) ,U(t,s)) \dd s - \int_t^{T} Z(t,s) \dd W_s \\ 
&& \qquad - \int_t^{T}  \int_{\bR^m}  U(t,s,x) \tpi(\dd s,\dd x) -\int_t^{T} \dd M(t,s). 
\end{eqnarray*}
In the Brownian-Poisson $L^2$-setting, it was already studied in \cite{wang:zhan:07} and if there is only the Brownian component $W$, for $p\neq 2$, in \cite{wang:12}.



\begin{Thm} \label{thm:Hp_M_solution}
For $p >1$, assume that $\Phi \in \bL^p_{\cF_T}(0,T)$, that {\rm \ref{H2}} and {\rm \ref{H3star}} hold. Suppose that 
\begin{equation}\label{eq:int_cond_f_0}
\bE \int_0^T \left( \int_t^T |f^0(t,s)| \dd s \right)^p \dd t < +\infty.
\end{equation}
Then the BSVIE \eqref{eq:general_BSVIE_type_I} has a unique adapted M-solution $(Y,Z,U,M)$ in $\mathfrak S^p(0,T)$ on $[0,T]$. Moreover for any $S\in[0,T]$
\begin{eqnarray} \nonumber
 \|(Y,Z,U,M)\|_{\mathfrak S^p(S,T)}^p& = & \bE \left[ \int_S^T  |Y(t)|^p \dd t + \int_S^T \left(  \int_S^T  |Z(t,r)|^2 \dd r \right)^{p/2} \dd t  \right. \\ \nonumber
&& \quad \left.+ \int_S^T \left(  \|U(t,\cdot)\|^2_{\bL^2_\pi(S,T)}  \right)^{p/2}\dd  t +\int_S^T \left( \langle M(t,\cdot) \rangle_{S,T}\right)^{p/2} \dd t\right]  \\  \label{eq:Hp_estimate}
& \leq & C \bE \left[ \int_S^T |\Phi(t)|^p \dd t +\int_S^T \left( \int_t^T |f^0(t,r)| dr \right)^p \dd t \right].
\end{eqnarray}
\end{Thm}
Recall that 
$$ \|U(t,\cdot)\|^2_{\bL^2_\pi(S,T)} =   \int_S^T \int_{\bR^m}|U(t,r,x)|^2\pi(\dd r,\dd x).$$

The proof is based on intermediate results. We consider the BSVIE \eqref{eq:special_BSVIE} which becomes here
\begin{eqnarray} \nonumber
Y(t) &= &\Phi(t) + \int_t^{T} h(t,s, Z(t,s) ,U(t,s)) \dd s - \int_t^{T} Z(t,s) \dd W_s  \\ \label{eq:Ito_special_BSVIE}
&& \qquad - \int_t^{T}  \int_{\bR^m} U(t,s,x) \tpi(\dd s,\dd x) -\int_t^{T} \dd  M(t,s),
\end{eqnarray}
The first result is a modification of Lemmata \ref{lem:param_BSDE->BSVIE} and \ref{lem:stab_BSVIE_h_poss}, namely
\begin{Lemma} \label{lem:Ito_param_BSDE->BSVIE}
If 
\begin{equation}\label{eq:Ito_int_cond_h_0} 
\bE \int_S^T \left( \int_S^T |h(t,s,0,{\mathbf 0})| \dd s \right)^p \dd t < +\infty
\end{equation}
holds and if $\Phi \in \bL^p_{\cF_T}(S,T)$, then the BSVIE \eqref{eq:Ito_special_BSVIE} has a unique adapted M-solution in $\mathfrak S^p(S,T)$ and for $t \in [S,T]$:
\begin{eqnarray} \nonumber
&& \bE \left[  |Y(t)|^p +  \left( \int_S^T  |Z(t,r)|^2 \dd r \right)^{\frac{p}{2}} +  \left(  \int_S^T  \int_{\bR^m} |U(t,r,x)|^2 \pi(\dd r, \dd x) \right)^{\frac{p}{2}}  \right. \\ \label{eq:Hp-estimate}
&&\quad \left.+\left(  \langle M(t,\cdot) \rangle_{S,T} \right)^{\frac{p}{2}}\right]  \leq C \bE \left[|\Phi(t)|^p  + \left( \int_S^T |h(t,r,0,\mathbf{0})| \dd r \right)^p \right].
\end{eqnarray}
Moreover we have a stability result for this BSVIE. Let $(\bar \Phi,\bar h)$ be a couple of data each satisfying the above assumptions {\rm \ref{H2}} and {\rm \ref{H3star}} and the same integrability conditions. Let $(\bar Y,\bar Z,\bar U,\bar M)$ the solution of the BSVIE \eqref{eq:Ito_special_BSVIE} with data $(\bar \Phi, \bar h)$. Define
$$(\fd Y,\fd Z, \fd U, \fd M) = (Y-\bar Y, Z-\bar Z,U-\bar U,M-\bar M).$$
Then there exists a constant $C$ depending on $p$, $K$ and $T$, such that for $t \in [S,T]$
\begin{eqnarray} \nonumber
&& \bE \left[ |\fd Y(t)|^p +  \left( \int_S^T  |\fd Z(t,r)|^2 \dd r \right)^{p/2} + \left(  \langle \fd M(t,\cdot) \rangle_{S,T} \right)^{p/2} \right. \\ \nonumber
&&\quad \left.+  \left(  \int_S^T  \int_{\bR^m} |\fd U(t,r,x)|^2 \pi(\dd r, \dd x) \right)^{p/2}\right] \leq C  \bE \Bigg[ |\Phi(t)-\bar \Phi(t)|^p \\ \label{eq:tech_stab_BSVIE}
&& \qquad \left. + \left( \int_S^T |h(t,r,Z(t,r),U(t,r))-\bar h(t,r,Z(t,r),U(t,r)) |\dd r \right)^p   \right].
\end{eqnarray}
\end{Lemma}
This lemma is a consequence of Proposition \ref{prop:BSDE_Lp_sol_exist} applied to the parametrized BSDE \eqref{eq:parametrized_BSDE}. The arguments are similar to the ones used for Lemmata \ref{lem:param_BSDE->BSVIE} and \ref{lem:stab_BSVIE_h_poss} or for \cite[Corollary 3.6]{yong:08}. For completeness we recall the main ideas.
The parametrized BSDE \eqref{eq:parametrized_BSDE} becomes:
\begin{eqnarray*} \nonumber
\lambda(t,r) &= &\Phi(t) + \int_r^{T} h(t,s, z(t,s) ,u(t,s)) \dd s - \int_r^{T} z(t,s) \dd W_s  \\
&& \qquad - \int_r^{T}  \int_{\bR^m}  u(t,s,x) \tpi(\dd s,\dd x) -\int_r^{T} \dd m(t,s).
\end{eqnarray*}
From Proposition \ref{prop:BSDE_Lp_sol_exist}, for any $\Phi(\cdot) \in \bL^p_{\cF_T}(S,T)$ the previous BSDE has a unique solution $(\lambda(t,\cdot),z(t,\cdot),u(t,\cdot),m(t,\cdot))$ in $\bD^p(R,T)$ and for a.e. $t\in [S,T]$
\begin{eqnarray*} \nonumber
&& \bE \left[\sup_{r\in[R,T]}   |\lambda(t,r)|^p +  \left( \int_R^T  |z(t,r)|^2 \dd r \right)^{p/2} + \left(  
\langle m(t,\cdot) \rangle_{R,T} \right)^{p/2} \right. \\ 
&& \left.+  \left(  \int_R^T  \int_{\bR^m} |u(t,r,x)|^2 \pi(\dd r,\dd x) \right)^{p/2}\right]  \leq C \bE \left[|\Phi(t)|^p  + \left( \int_R^T |h(t,r,0,\mathbf{0})| \dd r \right)^p \right].
\end{eqnarray*}
Moreover we have a stability result for BSDEs (\cite[Lemma 5 and proof of Proposition 2]{krus:popi:14} for $p\geq 2$, \cite[Proposition 3]{krus:popi:17}). Let $(\bar \Phi,\bar h)$ be a couple of data each satisfying the above assumption \ref{H2}-\ref{H3star} and the required integrability conditions on the data. Let $(\bar \lambda(t,\cdot),\bar z(t,\cdot),\bar u(t,\cdot),\bar m(t,\cdot))$ the solution of the BSDE \eqref{eq:parametrized_BSDE} with data $(\bar \Phi, \bar h)$. Define
$$(\fd  \lambda(t,\cdot),\fd z(t,\cdot), \fd u(t,\cdot), \fd m(t,\cdot)) = (\lambda(t,\cdot)-\bar \lambda(t,\cdot),z(t,\cdot)-\bar z(t,\cdot),u(t,\cdot)-\bar u(t,\cdot),m(t,\cdot)-\bar m(t,\cdot)).$$
Then there exists a constant $C$ depending on $p$, $K$ and $T$, such that
\begin{eqnarray} \nonumber
&& \bE \left[\sup_{r\in[R,T]}   |\fd \lambda(t,r)|^p +  \left( \int_R^T  |\fd z(t,r)|^2 \dd r \right)^{p/2} + \left(  \langle \fd m(t,\cdot) \rangle_{R,T} \right)^{p/2} \right. \\ \label{eq:BSDE_stability}
&&\qquad \qquad \left.+  \left(  \int_R^T  \int_{\bR^m} |\fd u(t,r,x)|^2 \pi(\dd r,\dd x) \right)^{p/2}\right] \\ \nonumber 
&& \qquad  \leq C  \bE \left[ |\Phi(t)-\bar \Phi(t)|^p + \left( \int_R^T |h(t,r,z(t,r),u(t,r))-\bar h(t,r,z(t,r),u(t,r)) |\dd r \right)^p   \right].
\end{eqnarray}
If we define $Y(t) = \lambda(t,t)$, $t \in [S,T]$, $Z(t,s) = z(t,s)$, $U(t,s,e) = u(t,s,e)$, $M(t,s) = m(t,s)$ for $(t,s) \in \Delta^c(S,T)$, the equation \eqref{eq:parametrized_BSDE} becomes the BSVIE \eqref{eq:Ito_special_BSVIE}.

But we can also deduce one result for {\it stochastic Fredholm integral equation} (SFIE in abbreviated form). 
Indeed for the BSDE \eqref{eq:parametrized_BSDE}, let us fix $r=S \in [R,T)$ and define for $t\in [R,S]$ and $s \in [S,T]$:
$$\psi^S(t) = \lambda(t,S),\ Z(t,s) = z(t,s),\ U(t,s) = u(t,s), \ M(t,s)=m(t,s).$$
Then Equation \eqref{eq:parametrized_BSDE} becomes a SFIE: for $t \in [R,S]$
\begin{eqnarray} \nonumber
\psi^S(t) &= &\Phi(t) + \int_S^{T} h(t,s, Z(t,s) ,U(t,s)) \dd s - \int_S^{T} Z(t,s) \dd W_s  \\  \label{eq:SFIE}
&& \qquad - \int_S^{T}  \int_{\bR^m} U(t,s,x) \tpi(\dd s,\dd x) -\int_S^{T} \dd M(t,s).
\end{eqnarray}
\begin{Lemma} \label{lem:Ito_SFIE_sol}
If \eqref{eq:Ito_int_cond_h_0} holds and if $\Phi \in \bL^p_{\cF_T}(S,T)$, then the SFIE \eqref{eq:SFIE} has a unique solution such that $\psi^S$ belongs to $ \bL^p_{\cF_S}(R,S)$ and 
$$(Z,U,M) \in \times \bL^p(R,S;\bH^p(S,T)) \times \bL^p(R,S;\bL^p_\pi(S,T))\times \bL^p(R,S;\cH^{p,\perp}(S,T))$$
and for $t \in [R,S]$
\begin{eqnarray} \nonumber
&& \bE \left[  |\psi^S(t)|^p +  \left( \int_S^T  |Z(t,r)|^2 \dd r \right)^{\frac{p}{2}} + +  \left(  \int_S^T  \int_{\bR^m} |U(t,r,x)|^2 \pi(\dd r, \dd x) \right)^{\frac{p}{2}}  \right. \\  \label{eq:SFIE_Hp-estimate}
&&\qquad \left.\left(  \langle M(t,\cdot) \rangle_{S,T} \right)^{\frac{p}{2}}\ \right]  \leq C \bE \left[|\Phi(t)|^p  + \left( \int_S^T |h(t,r,0,\mathbf{0})| \dd r \right)^p \right].
\end{eqnarray}
\end{Lemma}
Note that here $\psi^S(t)$ is only required to be $\cF_S$-measurable for almost all $t$ and not $\bF$-adapted. Let us now come to the proof of Theorem \ref{thm:Hp_M_solution}. \\
\begin{proof}
We follow the scheme of the proof of \cite[Theorem 3.7]{yong:08}.

\noindent {\bf Step 1.} For any $S \in[0,T]$, let us consider the set $\widehat{\mathfrak S}^p(S,T)$, the space of all $(y,z,u,m)$ in $\mathfrak S^p(S,T)$ such that for a.e. $t \in [S,T]$ a.s. 
$$y(t) = \bE \left[ y(t)|\cF_S \right] + \int_S^t z(t,s) \dd W_s + \int_S^t \int_{\bR^m} u(t,s,x)\tpi(\dd s,\dd x) + \int_S^t \dd m(t,s).$$
From this representation, the Doob's martingale inequality and the Burkholder-Davis-Gundy inequality, for $t\in [S,T]$, the quantity
$$\bE \left[ \left( \int_S^t  |z(t,r)|^2 \dd r \right)^{\frac{p}{2}} + \left( \langle m(t,\cdot) \rangle_{S,t}  \right)^{\frac{p}{2}} + \left( \int_S^t  \int_{\bR^m} |u(t,r,x)|^2 \pi(\dd r, \dd x)\right)^{\frac{p}{2}} \right] $$
is bounded from above by $C \bE |y(t)|^p$. We take on $\widehat{\mathfrak S}^p(S,T)$ the following norm:
\begin{eqnarray*}
\|(y,z,u,m)\|_{\widehat{\mathfrak S}^p(S,T)}^p &= &\bE \left[ \int_S^T  |y(t)|^p \dd t + \int_S^T \left(  \int_t^T  |z(t,r)|^2 \dd r \right)^{p/2} dt \right. \\
&+&\left. \int_S^T \left(  \|u(t,r)\|^2_{\bL^2_\pi(t,T)}   \right)^{p/2}\dd t+\int_S^T  \left( \langle m(t,\cdot) \rangle_{t,T} \right)^{p/2}\dd t  \right] .
\end{eqnarray*}
The same arguments as \cite[Inequality (3.48)]{yong:08} show the norm equivalence:
\begin{eqnarray} \nonumber
&& \|(y,z,u,m)\|_{\widehat{\mathfrak S}^p(S,T)}^p  \leq \bE \left[ \int_S^T  |y(t)|^p \dd t + \int_S^T  \left( \int_S^T  |z(t,r)|^2 \dd r \right)^{p/2}\dd t \right. \\ \nonumber
&& \qquad \left.+ \int_S^T\left(   \|u(t,\cdot)\|^2_{\bL^2_\pi(S,T)}  \right)^{p/2} \dd t+\int_S^T\left(   \langle m(t,\cdot) \rangle_{S,T} \right)^{p/2}\dd t  \right]\\ \label{eq:norm_equivalence}
&& \qquad \leq  (C_p+1) \|(y,z,u,m)\|_{\widehat{\mathfrak S}^p(S,T)}^p.
\end{eqnarray}
Note that using \eqref{eq:BJ_ineq_p_=_2}, if $p\geq 2$, $ \|u(t,\cdot)\|^2_{\bL^2_\pi(S,T)}$ can be replaced by $ \int_S^T  \|u(t,r)\|^2_{\bL^2_\mu}  \dd r$ in the previous estimates (with the suitable modifications of the constant $C_p$). 
%

If $\Phi \in \bL^p_{\cF_T}(S,T)$ and $(y,\zeta,\nu,m) \in \widehat{\mathfrak S}^p(S,T)$, we consider the BSVIE on $[S,T]$: 
\begin{eqnarray} \nonumber
Y(t) &= &\Phi(t) + \int_t^{T} f(t,s, y(s),Z(t,s) , U(t,s)) \dd s - \int_t^{T} Z(t,s) \dd W_s  \\ \label{eq:fix_point_BSVIE_2}
&& \qquad - \int_t^{T}  \int_{\bR^m} U(t,s,x) \tpi(\dd s,\dd x) -\int_t^{T} \dd M(t,s).
\end{eqnarray}
From Conditions \ref{H2}, \ref{H3star}, we can easily check that the generator of this BSVIE satisfies all requirements of Lemma \ref{lem:Ito_param_BSDE->BSVIE}. Thus this BSVIE has a unique adapted M-solution $(Y,Z,U,M) \in \mathfrak S^p(S,T)$ and for any $t \in [S,T]$
\begin{eqnarray} \nonumber
&& \bE \left[  |Y(t)|^p +  \left(  \int_t^T  |Z(t,r)|^2 \dd r \right)^{\frac{p}{2}} +(  \langle M(t,\cdot) \rangle_{t,T} )^{\frac{p}{2}} + \left(  \int_t^T  \int_{\bR^m} |U(t,r,x)|^2 \pi(\dd r, \dd x) \right)^{\frac{p}{2}}\right]  \\ \label{eq:tech_estim_step_1}
&& \ \leq C \bE \left[|\Phi(t)|^p  + \left( \int_t^T |f^0(t,r)| \dd r \right)^p + \int_t^T |y(r)|^p \dd r\right].
\end{eqnarray}
Therefore $(Y,Z,U,M) \in \widehat{\mathfrak S}^p(S,T)$. In other words we have a map $\Theta$ from $\widehat{\mathfrak S}^p(S,T)$ to $ \widehat{\mathfrak S}^p(S,T)$. 
Moreover arguing as in \cite{yong:08}, we obtain that for $(y,\zeta,\nu,m)$ and $(\bar y,\bar \zeta,\bar \nu, \bar m)$ in $\widehat{\mathfrak{S}}^p(S,T)$, if $(Y,Z,U,M)$ and $(\bar Y ,\bar Z, \bar U, \bar M)$ are the solutions of the BSVIE \eqref{eq:fix_point_BSVIE_2}, then from \eqref{eq:tech_stab_BSVIE} the difference satisfies:
\begin{eqnarray} \nonumber
&& \bE \left[ |\fd Y(t)|^p +  \left( \int_S^T  |\fd Z(t,r)|^2 \dd r \right)^{p/2} + \left(  \langle \fd M(t,\cdot) \rangle_{S,T} \right)^{p/2} \right. \\ \nonumber
&&\qquad \left.+  \left(  \int_S^T  \int_{\bR^m} |\fd U(t,r,x)|^2 \pi(\dd r, \dd x) \right)^{p/2}\right]  \\ \label{eq:Theta_contraction}
&& \quad \leq C \bE \left[ \left( \int_S^T K |\fd y(r)|  \dd r \right)^p   \right]  \leq C K^p (T-S)^{p-1} \bE \left[ \left( \int_S^T  |\fd y(r)|^p \dd r \right) \right]. 
\end{eqnarray}
For $T-S$ sufficiently small, this map is a contraction and thus it admits a unique fixed point $(Y,Z,U,M) \in  \widehat{\mathfrak S}^p(S,T)$ which is the unique adapted M-solution of \eqref{eq:second_special_BSVIE} on $[S,T]$. Moreover Estimate \eqref{eq:Hp_estimate} holds. This step determines the values $(Y(t),Z(t,s),U(t,s),M(t,s))$ for $(t,s) \in [S,T]\times[S,T]$. Note that at this step we can replace \ref{H3star} by a weaker condition as in \cite{yong:08}. 

\medskip 
Let us make a short break in the proof to understand the trouble in the case of a Type-II BSVIE. The driver of the BSVIE \eqref{eq:fix_point_BSVIE_2} would be replaced by
$$\int_t^T f(t,s,y(s),Z(t,s),\zeta(s,t), U(t,s), \nu(s,t)) \dd s.$$
Hence in \eqref{eq:tech_estim_step_1}, we would have the additional terms 
$$\bE \left[  \int_S^T\left(  \int_t^T |\zeta(r,t)|^2 \dd r \right)^{p/2} \dd t +\int_S^T \left(  \int_t^T \|\nu(r,t)\|^2_{\bL^2_\pi} \dd r \right)^{p/2} \dd t \right].$$
There are also other terms in \eqref{eq:Theta_contraction}.
To circumvent this issue, we could add this term in the definition of the norm. However we cannot control this symmetrized version of the norm for $(Z,U)$ in this step, but also in the next one. In other words the map $\Theta$ is no more a contraction.

\noindent {\bf Step 2.} We use the martingale representation theorem to define $(Z,U,M)$ on $[S,T]\times [R,S]$ for any $R \in [0,S)$. Indeed since $\bE[Y(t) | \cF_S] \in \bL^p(S,T;\bL^p_{\cF_S}(\Omega))$, there exists a unique triple $(Z,U,M)$ in $\bL^p(S,T;\bH^p(R,S)) \times  \bL^p(S,T;\bL^p_\pi(R,S)) \times\bL^p(S,T;\cH^{p,\perp}(R,S))$ such that for $t \in [S,T]$:
$$\bE[Y(t) | \cF_S] = \bE \left[ Y(t)|\cF_R \right] + \int_R^S Z(t,s)\dd W_s + \int_R^S \int_{\bR^m} U(t,s,x)\tpi(\dd s,\dd x) + \int_R^S \dd M(t,s),$$
and
\begin{eqnarray*}
&& \bE \left[\left( \int_R^S |Z(t,r)|^2 \dd r \right)^{\frac{p}{2}} +\left(  \int_R^S  \int_{\bR^m} |U(t,r,x)|^2 \pi(\dd r, \dd x) \right)^{\frac{p}{2}}+\left(  \langle M(t,\cdot) \rangle_{R,S} \right)^{\frac{p}{2}}\right] \\
&& \qquad  \leq C \bE |Y(t)|^p .
\end{eqnarray*}
Thus together with the first step, we have defined $(Z,U,M)$ for $(t,s) \in [S,T]\times [R,T]$ and 
\begin{eqnarray}  \nonumber
&&\bE \left[\int_S^T \left(  \int_R^T |Z(t,r)|^2 \dd r\right)^{\frac{p}{2}} \dd t + \int_S^T \left(\int_R^T  \int_{\bR^m} |U(t,r,x)|^2 \pi(\dd r, \dd x) \right)^{\frac{p}{2}}\dd t \right. \\ \nonumber
&&\hspace{2cm} \left. + \int_S^T\left( \langle M(t,\cdot) \rangle_{R,T} \right)^{\frac{p}{2}} \dd t  \right]   \\ \label{eq:Hp_estimates_again}
&& \qquad \leq  C \bE \left[ \int_S^T |\Phi(t)|^p \dd t +\int_S^T \left( \int_t^T |f^0(t,r)| \dd r \right)^p \dd t \right].
\end{eqnarray}

\noindent {\bf Step 3.} From the two previous steps, for $(t,s) \in [R,S]\times [S,T]$, the values of $Y(s)$ and $(Z(s,t),U(s,t))$ are already defined. Thus let us consider
$$f^S(t,s,z,u) = f(t,s,Y(s),z,u), \quad (t,s,z,u) \in [R,S]\times [S,T] \times \bR^{d\times k} \times (\bL^1_\mu + \bL^2_\mu),$$
and from Lemma \ref{lem:Ito_SFIE_sol}, the SFIE
\begin{eqnarray*} 
&&\psi^S(t) = \Phi(t) + \int_S^T f^S(t,s,Z(t,s),U(t,s))\dd s - \int_S^{T} Z(t,s) \dd W_s  \\ 
&& \qquad - \int_S^{T}  \int_{\bR^m} U(t,s,x) \tpi(\dd s,\dd x) -\int_S^{T} \dd M(t,s)
\end{eqnarray*}
has a unique solution $(\psi^S,Z,U,M)$ such that for $t \in [R,S]$
\begin{eqnarray*} \nonumber
&& \bE \left[  |\psi^S(t)|^p + \left(  \int_S^T  |Z(t,r)|^2 \dd r  \right)^{\frac{p}{2}} \right. \\
&& \qquad \left.+ \left( \langle M(t,\cdot) \rangle_{S,T}  \right)^{\frac{p}{2}}+  \left(  \int_S^T  \int_{\bR^m} |U(t,r,x)|^2 \pi(\dd r, \dd x)\right)^{\frac{p}{2}} \right] \\
&& \quad \leq C \bE \left[|\Phi(t)|^p  + \left( \int_S^T |f^S(t,r,0,\mathbf{0})| \dd r \right)^p \right] \\
&& \quad = C \bE \left[|\Phi(t)|^p  + \left( \int_S^T |f(t,r,Y(r),0,\mathbf{0})| \dd r \right)^p \right] \\
&& \quad \leq C \bE \left[|\Phi(t)|^p  + \left( \int_S^T |f^0(t,r)| \dd r \right)^p + \int_S^T |Y(r)|^p \dd  r \right].
\end{eqnarray*}
Hence using \eqref{eq:BJ_ineq_p_=_2} for $p\geq 2$ or \eqref{eq:norm_equiv} for $p<2$, and \eqref{eq:tech_estim_step_1} and \eqref{eq:Hp_estimates_again}, we obtain:
\begin{eqnarray} \nonumber
&& \bE \left[ \int_R^S |\psi^S(t)|^p \dd t + \int_R^S \left(  \int_S^T  |Z(t,r)|^2 \dd r \right)^{\frac{p}{2}} \dd t + \int_R^S \left( \langle M(t,\cdot) \rangle_{S,T} \right)^{\frac{p}{2}} \dd t \right. \\  \nonumber
&&\qquad \left. + \int_R^S  \left(  \int_S^T \|U(t,r)\|^2_{\bL^1_\mu +\bL^2_\mu} \dd r \right)^{\frac{p}{2}} \dd t \right] \\ \nonumber
&& \quad \leq C \bE \left[\int_R^S |\Phi(t)|^p \dd t +\int_R^S \left( \int_S^T |f^0(t,r)| \dd r \right)^p dt+ \int_S^T |Y(r)|^p \dd r \right] \\ \label{eq:estimate_step3_bis}
&& \quad \leq C  \bE \left[ \int_R^T |\Phi(t)|^p \dd t +\int_R^T \left( \int_t^T |f^0(t,r)| \dd r \right)^p \dd t \right].
\end{eqnarray}
Hence we have defined $(Z,U,M)$ for $(t,s) \in [R,S]\times [S,T]$, and and by the definition of $f^S$, for $t \in [R,S]$
\begin{eqnarray} \nonumber
&&\psi^S(t) = \Phi(t) + \int_S^T f(t,s,Y(s),Z(t,s),U(t,s))\dd s - \int_S^{T} Z(t,s) \dd W_s  \\ \label{eq:intermediate_SFIE_bis}
&& \qquad - \int_S^{T}  \int_{\bR^m} U(t,s,x) \tpi(\dd s,\dd x) -\int_S^{T} \dd M(t,s).
\end{eqnarray}

\medskip
As in the first step and for the same reason, the general driver of the Type-II BSVIE \eqref{eq:general_BSVIE} can not be directly handled in Equation \eqref{eq:estimate_step3_bis}.

\noindent {\bf Step 4.} 
 Let us summarize what we have after these three steps. $Y$ is uniquely defined on $[S,T]$ (from Step 1) and $(Z,U,M)$ are uniquely determined on $[S,T]\times [R,T]$ (from Steps 1 and 2) and on $[R,S]\times [S,T]$ (from Steps 1 and 3). Let us now solve \eqref{eq:second_special_BSVIE} on $[R,S]^2$. Consider 
\begin{eqnarray*} 
&&Y(t) = \psi^S(t) + \int_t^S f(t,s,Y(s),Z(t,s),U(t,s))\dd s - \int_t^{S} Z(t,s) \dd W_s  \\ 
&& \qquad - \int_t^{S}  \int_{\bR^m} U(t,s,x) \tpi(\dd s,\dd x) -\int_t^{S} \dd M(t,s).
\end{eqnarray*}
It is a BSVIE with terminal condition $\psi^S \in \bL^p_{\cF_S}(R,S)$ and generator $f$. As in the first step, this BSVIE has a unique solution in $\mathfrak S^p(R,S)$ provided that $S-R > 0$ is small enough. Now for $t \in [R,S]$ from the expression \eqref{eq:intermediate_SFIE_bis} of $\psi^S$, we obtain that 
\begin{eqnarray*} 
&&Y(t) = \Phi(t) + \int_t^T f(t,s,Y(s),Z(t,s),U(t,s))\dd s - \int_t^{T} Z(t,s) \dd W_s  \\ 
&& \qquad - \int_t^{T}  \int_{\bR^m} U(t,s,x) \tpi(\dd s,\dd x) -\int_t^{T} \dd M(t,s).
\end{eqnarray*}
Moreover again by the arguments as in the first step
\begin{eqnarray*} \nonumber
&& \bE \left[ \int_R^S  |Y(t)|^p \dd t + \int_R^S \left(  \int_R^S  |Z(t,r)|^2 \dd r \right)^{\frac{p}{2}}\dd t +\int_R^S \left(   \langle M(t,\cdot) \rangle_{R,S} \right)^{\frac{p}{2}}\dd t \right. \\
&& \qquad \left.+ \int_R^S \left(   \|U(t,\cdot)\|^2_{\bL^2_\pi(R,S)}  \right)^{\frac{p}{2}}\dd t \right]  \\ 
&& \quad \leq C \bE \left[ \int_R^S |\Phi(t)|^p \dd t +\int_R^S \left( \int_t^S |f^0(t,r)|  \dd r \right)^p \dd t \right] \\
&& \quad \leq C \bE \left[ \int_R^T |\Phi(t)|^p \dd t +\int_R^T \left( \int_t^T |f^0(t,r)|  \dd r \right)^p \dd t \right].
\end{eqnarray*}
From this inequality together with \eqref{eq:Hp_estimate} on [S,T], \eqref{eq:Hp_estimates_again} and \eqref{eq:estimate_step3_bis}, we proved that the BSVIE \eqref{eq:second_special_BSVIE} has a unique adapted M-solution $(Y,Z,U,M)$ in $\mathfrak S^p(R,T)$ on $[R,T]$ with the estimate \eqref{eq:Hp_estimate} on $[R,T]$. 

\noindent {\bf Step 5.} The conclusion of the proof is done by induction since the time intervals $[S,T]$ (Step 1) and $[R,S]$ (Step 4) are determined by absolute constants depending only on the Lipschitz constant $K$ of $f$ in Conditions \ref{H2}--\ref{H3star} and on the time horizon $T$. 
\end{proof}

The stability result holds in our setting. Let $\bar \Phi \in L^p_{\cF_T}(0,T)$ and $\bar f :\Omega \times [0,T]\times \bR^{d+(d\times k)} \times (\bL^1_\mu+\bL^2_\mu)^2 \to \bR^d$ satisfy \ref{H2}--\ref{H3star} and \eqref{eq:int_cond_f_0}
\begin{equation*} 
\bE \int_0^T \left( \int_t^T |\bar f^0(t,s)| \dd s \right)^p \dd t < +\infty.
\end{equation*}
Let $(\bar Y,\bar Z,\bar U,\bar M)$ in $\mathfrak S^p(0,T)$ be the unique adapted M-solution of the BSVIE \eqref{eq:second_special_BSVIE} with data $\bar \Phi$ and $\bar f$ (Theorem \ref{thm:Hp_M_solution}). Then for any $S\in[0,T]$
\begin{eqnarray} \nonumber
&& \bE \left[ \int_S^T  |Y(t)-\bar Y(t)|^p \dd t + \int_S^T \left(  \int_S^T  |Z(t,r)- \bar Z(t,s)|^2 \dd r\right)^{\frac{p}{2}} \dd t  \right. \\ \nonumber
&& \qquad \left.+\int_S^T \left(  \langle M(t,\cdot) - \bar M(t,\cdot)\rangle_{S,T}\right)^{\frac{p}{2}} \dd t+ \int_S^T\left(  \int_S^T  \|U(t,r)-\bar U(t,r)\|^2_{\bL^2_\mu}  \dd r \right)^{\frac{p}{2}} dt \right]  \\ \nonumber
&& \quad \leq  C \bE \left[ \int_S^T |\Phi(t)-\bar \Phi(t)|^p \dd t \right. \\ \nonumber 
&& \qquad + \int_S^T \left( \int_t^T |f(t,r,Y(r),Z(t,r),U(t,r)) \right. \\  \label{eq:H_p_stability}
&& \hspace{3cm} -\left.\bar f(t,r,Y(r),Z(t,r),U(t,r)) | \dd r \bigg)^p \dd t \right].
\end{eqnarray}
The proof is based on the same arguments given in \cite{yong:08} (see Equation (3.71) in particular) and we skip it here.

\subsection{Time regularity} \label{ssect:time_regularity}

Let us emphasize that we do not require any regularity of the paths $t \mapsto Y(t)$, which are a priori not continuous nor c\`adl\`ag. The component $Y$ is only supposed to be in $\bL^p_\bF(0,T)$ (or $\bL^2_{\beta,\bF}(0,T)$ in Section \ref{sect:general_Type_I_BSVIE}). 
If $Y$ solves the BSDEs \eqref{eq:poss_BSDE} or \eqref{eq:Ito_BSDE}, then it has the same regularity as the martingale part, thus a.s. it is a c\`adl\`ag process. For a BSVIE it is more delicate. In \cite[Theorem 4.2]{yong:08}, the author shows that in the Brownian setting the BSVIE 
\begin{eqnarray} \label{eq:brownian-type_II_BSVIE}
Y(t) &= &\Phi(t) + \int_t^{T} f(t,s, Y(s), Z(t,s),Z(s,t)) \dd s - \int_t^{T} Z(t,s) \dd W_s
\end{eqnarray}
is continuous in $\bL^2(\Omega)$, which does not mean that $Y$ has a.s. continuous paths. Of course since $t$ appears in the generator $f$ and in the free term $\Phi$, we have to add some property on $t\mapsto \Phi(t)$ and $t \mapsto f(t,\cdots)$. To obtain the time regularity for the BSVIE \eqref{eq:brownian-type_II_BSVIE}, the author uses the Malliavin derivative to control the term $Z(s,t)$ in the generator (see \cite[Theorems 4.1 and 4.2]{yong:08}). Hence to apply the same arguments, we should use the Malliavin calculus in the presence of jumps (see e.g. \cite{boul:deni:15,dinu:okse:pros:09}). This point is left as future research and to avoid this machinery, let us study the BSVIE \eqref{eq:second_special_BSVIE}
\begin{eqnarray*} \nonumber
Y(t) &= &\Phi(t) + \int_t^{T} f(t,s, Y(s), Z(t,s) ,U(t,s)) \dd s - \int_t^{T} Z(t,s) \dd W_s \\ 
&& \qquad - \int_t^{T}  \int_{\bR^m}  U(t,s,x) \tpi(\dd s,\dd x) -\int_t^{T} \dd M(t,s). 
\end{eqnarray*}

Let us describe several sets for the process $Y$.
\begin{eqnarray*}
D([0,T];\bL^p_{\bF}(\Omega)) & = & \Bigg\{ \phi \in \bL^\infty(0,T;\bL^p_{\bF}(\Omega)), \ \phi(t) \ \mbox{is }\bF-\mbox{adapted} ,\\
&&\qquad \qquad \left. \phi(\cdot) \ \mbox{is c\`adl\`ag from } [0,T] \mbox{ to } \bL^p_{\bF}(\Omega). \right\},\\
D^\sharp([0,T];\bL^p_{\bF}(\Omega)) & =& \Bigg\{ \phi \in D([0,T];\bL^p_{\bF}(\Omega)) , \  \phi(\cdot) \ \mbox{is c\`adl\`ag paths a.s.} \Bigg\},\\
\bD^p_{\bF}(0,T) &= & \bD^p(0,T)= \Bigg\{ \phi \in D^\sharp([0,T];\bL^p_{\bF}(\Omega)), \  \bE \left[ \sup_{t\in [0,T} |\phi(t)|^p\right] < +\infty \Bigg\}.
\end{eqnarray*}
When only measurability is required, we replace the subscript $\bF$ by $\cF_S$. If we want to deal with continuity, then we replace $D$ (resp. $\bD$) by $C$ (resp. $\mathbb C$) (see \cite[Section 2.1]{yong:08}).
Coming back to a generic martingale $M(t,\cdot)$, we also define 
$$\bL^p(\Omega;D([S,T];\bM^p(S,T))$$ 
as the set of all $M \in \bL^\infty(S,T;\bM^p(S,T))$ such that 
$t \mapsto M(t,\cdot)$ is c\`adl\`ag from $[S,T]$ to $\bM^p(0,T)$ and
$$\bE \left( \sup_{t\in [S,T]} \langle M(t,\cdot) \rangle_{S,T} \right)^\frac{p}{2}  < +\infty.$$
Again if $M(t,\cdot)$ is a Brownian martingale, then $M \in \bL^p(\Omega;D([S,T];\bM^p(S,T))$ if and only if $Z \in  \bL^p(\Omega;D([S,T];\bH^p(S,T))$ and if $N(t,\cdot)$ is a Poisson martingale, then $N \in \bL^p(\Omega;D([S,T];\bM^p(S,T))$ is equivalent to $\psi  \in  \bL^p(\Omega;D([S,T];\bL^2_\pi(S,T))$.

%

\bigskip

The integrability condition \eqref{eq:int_cond_f_0} is replaced by the stronger one:
\begin{equation} \label{eq:stronger_int_cond_f_0}
\sup_{t\in [0,T]} \bE \left[ \left( \int_t^T |f^0(t,s)| \dd s \right)^p \right] < +\infty.
\end{equation}
Finally instead of $\Phi \in \bL^p_{\cF_T}(0,T)$, we assume also that 
\begin{equation} \label{eq:stronger_int_cond_Phi}
\sup_{t\in [0,T]} \bE \left[|\Phi(t)|^p \right]  < +\infty.
\end{equation}
Under these two hypotheses and if $f$ satisfies \ref{H2} and \ref{H3star}, it is possible to deduce that $Y$ belongs to $D([0,T];\bL^p_{\bF}(\Omega))$, provided that we have regularity assumption on $t\mapsto \Phi(t)$ and $t\mapsto f(t,s,y,z,\psi)$, as in \cite[Theorem 4.2]{yong:08} in the continuous setting (see Section \ref{ssect:Lp_continuity} in the appendix). Let us emphasize again that it does not mean that $Y$ is in $D^\sharp([0,T];\bL^p_{\bF}(\Omega)) $; in other words we do not deduce that a.s. the paths are c\`adl\`ag. 

In \cite[Theorem 2.4]{wang:zhan:07}, a.s. continuity of $Y$ is proved in the Brownian setting and if the generator of the BSVIE \eqref{eq:brownian-type_II_BSVIE} is of Type-I, namely for the BSVIE:
\begin{eqnarray*} \nonumber
Y(t) &= &\Phi(t) + \int_t^{T} f(t,s, Y(s), Z(t,s)) \dd s - \int_t^{T} Z(t,s) \dd W_s.
\end{eqnarray*}
Our aim now is to extend this property for the BSVIE \eqref{eq:second_special_BSVIE}, assuming that $\Phi$ and $f$ are H\"older continuous w.r.t. $t$. Before we state the next result, which is the same as \cite[Lemma 3.1]{wang:zhan:07}: 
\begin{Lemma} \label{lem:exp_ineq}
Let us assume that for $\Phi \in \bL^2_{\cF_T}(0,T)$, for $f=\{f(t,s), \ 0\leq t \leq s\leq T\}$ such that 
$$\bE \int_0^T\int_0^T  |f(t,s)|^2 \mathbf 1_{s\geq t} \dd s < +\infty,$$
and for some $(Z,U,M)$ is in $\cH^2(0,T)$, we have for almost all $t \in [0,T]$
\begin{eqnarray*}
Y(t) &=& \Phi(t) + \int_t^T f(t,s) \dd s  + \int_t^T Z(t,s)\dd W_s + \int_t^T \int_{\bR^m} U(t,s,x)\tpi(\dd s,\dd x) + \int_t^T \dd M(t,s).
\end{eqnarray*}
Then
\begin{eqnarray*}
e^{\beta t} |Y(t)|^2 & \leq & e^{\beta T} \bE^{\cF_t} |\Phi(t)|^2  + \frac{1}{\beta} \bE^{\cF_t} \int_t^T  e^{\beta s} |f(t,s)|^2 \dd s
\end{eqnarray*}
%
and
\begin{eqnarray*}
&&\bE^{\cF_t} \int_t^T e^{\beta s}  |Z(t,s)|^2  \dd s  + \bE^{\cF_t} \int_t^T e^{\beta s}  \dd \langle M(t,\cdot) \rangle_{t,s} + \bE^{\cF_t} \int_t^T e^{\beta s} \int_{\bR^m} |U(t,s,x)|^2\pi(\dd s,\dd x)  \\
 & & \quad \leq e^{\beta T} \bE^{\cF_t} |\Phi(t)|^2  +\frac{1}{\beta}  \bE^{\cF_t} \int_t^T  e^{\beta s} |f(t,s)|^2 \dd s.
\end{eqnarray*}
\end{Lemma}
\begin{proof}
The proof is an adaptation of the arguments of \cite{wang:zhan:07}, together with \cite{krus:popi:14,krus:popi:17}, and is postponed in the appendix (see Section  \ref{ssect:Lp_continuity}). 
\end{proof}

Now we have the next regularity result, which is the extension of \cite[Theorem 2.4]{wang:zhan:07}. The restriction $p\geq 2$ is due to the dependence of the generator $f$ on $U$. If it doesn't, the arguments of the proof lead to the same conclusion for $p>1$. Following \cite{krus:popi:17} it should be possible to extend the theorem for $p>1$, but this point is left for further research.
\begin{Thm} \label{thm:as_regularity_p_2}
In addition to {\rm \ref{H2}} and {\rm \ref{H3star}}, suppose that the generator satisfies for some $p \geq 2$ and $0 < \alpha <1$ such that $\alpha p > 1$ and with $\varrho >0$, uniformly in $(\omega,s,y,z,\psi)$:
$$|f(t,s,y,z,\psi)-f(t',s,y,z,\psi)|\leq \varrho |t-t'|^\alpha, \quad  \bE \left[  \left| \Phi(t) - \Phi(t')\right|^p  \right] \leq \varrho |t-t'|^{\alpha p}, $$
 for all $(y,z,\psi)$ and all $0\leq t,t' \leq s \leq T$. Moreover 
$$\bE \left[  \sup_{t\in [0,T]} \left(  \int_t^{T} |f^0(t,s)|^{p} \dd s \right) \ \right] <+\infty.$$
Then the solution $Y$ of the BSVIE \eqref{eq:second_special_BSVIE} has a c\`adl\`ag version, still denoted by $Y$, such that 
$$\bE \left[ \sup_{s\in [0,T]} \left| Y(s)\right|^p \right] \leq \infty.$$
\end{Thm}
Let us emphasize that our conditions imply that a.s. $t\mapsto \Phi(t)$ is continuous.
The irregularity $\alpha$ of $f$ should be compensated by more integrability $p$ on the data. If $\alpha$ is close or equal to 1, the condition $p\geq 2$ is too strong but our arguments are not sufficient in the proof. \\*
\begin{proof}
From our assumptions, the solution $(Y,Z,U,M)$ belongs to $\mathfrak S^p(0,T)$ with $p\geq 2$, thus in $\mathfrak S^2(0,T)$. We follow the scheme of the proof of \cite[Theorem 2.4]{wang:zhan:07}. 

\noindent {\bf Step 1.} We consider for a fixed $t$ in $[0,T]$:
$$X_t(u)=\bE \left[ \Phi(t) + \int_t^T f(t,s) \dd s \bigg| \cF_u\right], \quad u \in [0,T].$$
From our assumption, $u\mapsto X_t(u)$ is a c\`adl\`ag $\bL^p$-martingale. For $0\leq t \leq t' \leq T$, Doob's martingale inequality implies 
\begin{eqnarray*} \nonumber
\bE \left[ \sup_{u\in [0,T]} \left| X_t(u) - X_{t'}(u)\right|^p \right] & \leq  & C \bE \left[  \left| X_t(T) - X_{t'}(T)\right|^p \right] \\
& \leq & C \bE \left[  \left| \Phi(t) - \Phi(t')\right|^p  \right] \\
 &+ &C \bE \left[  \left| \int_t^{t'} f(t,s) \dd s  \right|^p  + \left| \int_{t'}^T| f(t,s) - f(t',s) | \dd s \right|^p \right] .
\end{eqnarray*}
The H\"older's inequality leads to:
$$\bE \left[  \left| \int_t^{t'} f(t,s) \dd s  \right|^p \right] \leq |t'-t|^{p-1} \bE \left[  \left| \int_t^{T} |f(t,s)|^p \dd s  \right| \right].$$
Hence from our setting we have:
$$\bE \left[ \sup_{u\in [0,T]} \left| X_t(u) - X_{t'}(u)\right|^p \right] \leq  C_\varrho |t-t'|^{\alpha p}  . $$
Since $\alpha p > 1$, if we consider $X=(X_t, \ t \in [0,T])$ as a process with values in the Skorohod space $D([0,T] ; \bR^d)$ equipped with the uniform norm, which is a complete metric space, then we can apply the Kolmogorov continuity criterion (see \cite[IV.Corollary 1]{prot:04} or \cite[Theorem I.2.1]{revu:yor:99}): there is a continuity version of $t \in [0,T] \mapsto X_t \in D([0,T] ; \bR^d)$. In particular a.s. $t \mapsto Y(t):=X_t(t)$ is c\`adl\`ag:
\begin{eqnarray*}
|Y(t)-Y(t')| & \leq & |X_t(t) - X_t(t')| + |X_t(t')-X_{t'}(t')| \\
& \leq & |X_t(t) - X_t(t')|  + \sup_{s\in [0,T]} |X_t(s)-X_{t'}(s)|.
\end{eqnarray*}
Note that 
\begin{eqnarray*}
X_t(u) &=& X_t(0) + \int_0^u Z(t,s) \dd W_s + \int_0^u \int_{\bR^m} U(t,s,x)\tpi(\dd s,\dd x) + M(t,u) \\
&=& X_t(0) + \mart(t,u).
\end{eqnarray*}
Using the BDG inequality ($p \geq 2$), we have
\begin{eqnarray*}
\bE \left[ \left( \left\langle \mart(t,\cdot)-\mart(t',\cdot)\right\rangle_{0,T} \right)^{p/2} \right] &\leq& C \bE \left[ \left| \mart(t,T)-\mart(t',T) \right|^{p} \right] \\
& \leq & C \left( \bE \left[ \left| X_t(0)-X_{t'}(0) \right|^{p} \right] + \bE \left[ \left| X_t(T)-X_{t'}(T) \right|^{p} \right]  \right) \\
& \leq & C |t-t'|^{\alpha p}.
\end{eqnarray*}
And
\begin{eqnarray*}
\bE \left[ \left( \left\langle \mart(0,\cdot)\right\rangle_{0,T} \right)^{p/2} \right] &\leq& C \bE \left[ \left| \mart(0,T) \right|^{p} \right] \leq  C \bE \left[ \left| X_0(T)-X_{0}(0) \right|^{p} \right]  \\
& \leq & C \bE \left( |\Phi(0)|^{p} + \left( \int_0^T |f(0,s)| \dd s\right)^{p} \right)  \leq C.
\end{eqnarray*}
Let us recall that the space $\cH^2$ is a Banach space (see \cite{dell:meye:80}, Section VII.3 (98.1)-(98.2) or \cite{prot:04}, Section V.2). If we consider $t \mapsto \mart(t,\cdot)$, this map defined on $[0,T]$ takes values in the space $\cH^2$. We can apply the Kolmogorov continuity criterion (see again \cite[Theorem I.2.1]{revu:yor:99}) in order to have:
\begin{equation} \label{eq:estim_cont_sup_mart}
\bE \left[ \left( \sup_{t\in [0,T]}  \left\langle \mart(t,\cdot)\right\rangle_{0,T} \right)^{p/2} \right] \leq C.
\end{equation}

\noindent {\bf Step 2.} Assume that $f$ depends only on $z$ and $\psi$. Let us define $Z_0(t,s)\equiv 0$, $U_0(t,s)\equiv 0$ and recursively for $n\geq 1$:
\begin{eqnarray*} \nonumber
Y_n(t) &= &\Phi(t) + \int_t^{T} f(t,s, Z_{n-1}(t,s) ,U_{n-1}(t,s)) \dd s - \int_t^{T} Z_{n}(t,s) \dd W_s \\ 
&& \qquad - \int_t^{T}  \int_{\bR^m}  U_{n}(t,s,x) \tpi(\dd s,\dd x) -\int_t^{T} \dd  M_{n}(t,s). 
\end{eqnarray*}
We can argue exactly as in \cite{wang:zhan:07} in order to prove that for any $n\geq 1$:
\begin{eqnarray*}
&&\bE \left[ \left( \left\langle \mart_n(t,\cdot)-\mart_n(t',\cdot)\right\rangle_{0,T} \right)^{p/2} \right] \leq C_n |t-t'|^{\alpha p}, \\
&&\bE \left[  \left(\sup_{t\in [0,T]} \left\langle \mart_n(t,\cdot)\right\rangle_{0,T} \right)^{p/2} \right] \leq C_n,\\
&&t \mapsto Y_n(t) \ \mbox{is c\`adl\`ag.}
\end{eqnarray*}
Let us now prove the convergence of $Y_n$. Using Lemma \ref{lem:exp_ineq}, Conditions \ref{H2} and \ref{H3star}, we obtain 
\begin{eqnarray*}
&&e^{\beta t} |Y_{n+1}(t)-Y_n(t)|^2 +\bE^{\cF_t} \int_t^T e^{\beta s}  |Z_{n+1}(t,s)-Z_n(t,s)|^2  \dd s\\
&&\qquad +\bE^{\cF_t} \int_t^T e^{\beta s}  \|U_{n+1}(t,s,\cdot)-U_n(t,s,\cdot)\|^2_{\bL^2_\pi}  \dd s\\
&&\qquad +\bE^{\cF_t} \int_t^T e^{\beta s}   \dd \langle M_{n+1}(t,\cdot)-M_n(t,\cdot)\rangle_{t,s} \\
&&\quad  \leq \frac{1}{\beta^2}  \bE^{\cF_t} \int_t^T  e^{\beta s} |f(t,s,Z_n(t,s),U_n(t,s)) -f(t,s,Z_{n-1}(t,s),U_{n-1}(t,s)) |^2 \dd s \\
&&\quad \leq \frac{2 K^2}{\beta^{2}}  \bE^{\cF_t} \int_t^T  e^{\beta s}\left(  |Z_n(t,s)-Z_{n-1}(t,s)|^2 +\|U_n(t,s,\cdot)-U_{n-1}(t,s,\cdot)\|^2_{\bL^2_\pi} \right) \dd s. 
\end{eqnarray*}
Using the Inequality \eqref{eq:BJ_ineq_p_=_2}, taking $\beta$ large enough (greater than $4 (K^2 K_2^2)$, where $K$ comes from \ref{H3star} and $K_2$ from \eqref{eq:BJ_ineq_p_=_2}) and iterating the previous inequality yield to:
\begin{eqnarray*}
&&e^{\beta t} |Y_{n+1}(t)-Y_n(t)|^2 +\bE^{\cF_t} \int_t^T e^{\beta s}  |Z_{n+1}(t,s)-Z_n(t,s)|^2  \dd s\\
&&\qquad +\bE^{\cF_t} \int_t^T e^{\beta s}  \|U_{n+1}(t,s,\cdot)-U_n(t,s,\cdot)\|^2_{\bL^2_\mu}  \dd s\\
&&\qquad +\bE^{\cF_t} \int_t^T e^{\beta s}   \dd \langle M_{n+1}(t,\cdot)-M_n(t,\cdot)\rangle_{t,s} \\
&&\quad \leq \frac{1}{2^n}  \bE^{\cF_t} \int_t^T  e^{\beta s}\left(  |Z_1(t,s)-Z_{0}(t,s)|^2 +\|U_1(t,s,\cdot)-U_{0}(t,s,\cdot)\|^2_{\bL^2_\mu} \right) \dd s. 
\end{eqnarray*}
First taking the expectation and integrating w.r.t $t \in [0,T]$, we deduce the convergence of $(Z_n,U_n,M_n)$ in $\cH^2$. 
Then 
\begin{eqnarray*}
&&\bE \left[ \sup_{t\in [0,T]} |Y_{n+1}(t)-Y_n(t)|^p \right]   \\
&&\quad \leq \frac{1}{2^{np /2}} \bE \left[ \sup_{t\in [0,T]} \left(  \bE^{\cF_t} \int_t^T  e^{\beta s}\left(  |Z_1(t,s)|^2 +\|U_1(t,s,\cdot)\|^2_{\bL^2_\mu}  \right)\dd s \right)^{p/2} \right] \\
&&\quad \leq \frac{1}{2^{np /2}} \bE \left[ \sup_{t\in [0,T]} \left(  \bE^{\cF_t} \xi \right)^{p/2} \right],
\end{eqnarray*}
where
$$\xi = \sup_{t\in [0,T]} \int_t^T  e^{\beta s}\left(  |Z_1(t,s)|^2 +\|U_1(t,s,\cdot)\|^2_{\bL^2_\mu} \right)\dd s.$$
From \eqref{eq:estim_cont_sup_mart}, $\bE (\xi^{p/2}) < +\infty$ and $t\mapsto \bE^{\cF_t} (\xi)$ is a martingale. By Doob's maximal inequality 
\begin{eqnarray*}
\bE \left[ \sup_{t\in [0,T]} |Y_{n+1}(t)-Y_n(t)|^p \right] & \leq& \frac{1}{2^{np /2}} \bE \left[ \sup_{t\in [0,T]} \left(  \bE^{\cF_t} \xi \right)^{p/2} \right] \\
&\leq &  C\frac{e^{\beta p T}}{2^{np /2}} \bE \left[ \left(   \xi \right)^{p/2} \right]  \leq \frac{C}{2^{np /2}},
\end{eqnarray*}
where the constant $C$ does not depend on $n$. Thus there exists a c\`adl\`ag adapted process $Y$ such that 
$$\lim_{n\to +\infty} \bE \left[ \sup_{t\in [0,T]} |Y_{n+1}(t)-Y(t)|^p \right] = 0.$$
And we deduce immediately that the limit is the unique solution in $\mathfrak S^2(0,T)$ of the BSVIE
\begin{eqnarray*} \nonumber
Y(t) &= &\Phi(t) + \int_t^{T} f(t,s, Z(t,s) ,U(t,s)) \dd s - \int_t^{T} Z(t,s) \dd W_s \\ 
&& \qquad - \int_t^{T}  \int_{\bR^m}  U(t,s,x) \tpi(\dd s,\dd x) -\int_t^{T} \dd  M(t,s). 
\end{eqnarray*}

\noindent {\bf Step 3.} Assume that $f$ depends now also on $y$. Let us define $Y_0(t)\equiv 0$ and for $n\geq 1$:
\begin{eqnarray*} \nonumber
Y_n(t) &= &\Phi(t) + \int_t^{T} f(t,s, Y_{n-1}(s) ,Z_n(t,s),U_{n}(t,s)) \dd s - \int_t^{T} Z_{n}(t,s) \dd W_s \\ 
&& \qquad - \int_t^{T} \int_{\bR^m}  U_n(t,s,x) \tpi(\dd s,\dd x) -\int_t^{T} \dd M_{n}(t,s). 
\end{eqnarray*}
We know that $t\mapsto Y_n(t)$ is c\`adl\`ag. Using again Lemma \ref{lem:exp_ineq}, we obtain:
\begin{eqnarray*}
&&e^{\beta t} |Y_n(t)|^2 + \bE^{\cF_t} \int_t^T e^{\beta s} |Z_n(t,s)|^2  \dd s+ \bE^{\cF_t} \int_t^T e^{\beta s}  \|U_n(t,s,\cdot)\|^2_{\bL^2_\pi}  \dd s  \\
&& \quad + \bE^{\cF_t} \int_t^T e^{\beta s}  \dd  \langle M_n(t,\cdot)\rangle_{t,s}  \leq e^{\beta T} \bE^{\cF_t} |\Phi(t)|^2 \\
& & \qquad + \frac{4K^2}{\beta}  \bE^{\cF_t} \int_t^T  e^{\beta s} \left( |f^0(t,s)|^2 + |Y_{n-1}(s)|^2 + |Z_n(t,s)|^2 +  \|U_n(t,s,\cdot)\|^2_{\bL^2_\mu}  \right)\dd s.
\end{eqnarray*}
Thus for $\beta = 8K^2 K_2^2$ (again $K_2$ coming from \eqref{eq:BJ_ineq_p_=_2}), we have:
\begin{eqnarray*}
&&e^{\beta t} |Y_n(t)|^2 +\frac{1}{2} \bE^{\cF_t} \int_t^T e^{\beta s} \left( |Z_n(t,s)|^2  +  \|U_n(t,s,\cdot)\|^2_{\bL^2_\mu} \right) \dd s  \\
&& \quad +\frac{1}{2} \bE^{\cF_t} \int_t^T e^{\beta s}  \dd \langle M_n(t,\cdot)\rangle_{t,s}  \\
&&\qquad  \leq e^{\beta T} \bE^{\cF_t} |\Phi(t)|^2 + \frac{1}{2}  \bE^{\cF_t} \int_t^T  e^{\beta s} \left( |f^0(t,s)|^2 + |Y_{n-1}(s)|^2  \right) \dd s.
\end{eqnarray*}
Set
$$h_n(t) = \sup_{1\leq k\leq n} \bE \left[ \sup_{s\in [t,T]} \left| Y_k(s)\right|^p \right].$$
Then
\begin{eqnarray*}
h_n(t) &\leq &  \sup_{1\leq k\leq n} \bE \left[ \left( \sup_{s\in [t,T]} e^{\beta s} \left| Y_k(s)\right|^2\right)^{p /2} \right] \\
& \leq & C  \bE \left[ \left( \sup_{s\in [t,T]} \bE^{\cF_s} |\Phi(s)|^2\right)^{p /2} \right] + C  \bE \left[ \left( \sup_{s\in [t,T]}  \bE^{\cF_s} \int_s^T  e^{\beta u}  |f^0(s,u)|^2 \dd u \right)^{p /2} \right]\\
&+&   C \sup_{1\leq k\leq n} \bE \left[ \left( \sup_{s\in [t,T]}  \bE^{\cF_s} \int_s^T  e^{\beta u} |Y_{k-1}(u)|^2   \dd u \right)^{p /2} \right] \\
&\leq & C  \bE \left[ \left( \sup_{s\in [t,T]} \bE^{\cF_s} \sup_{r\in [0,T]}|\Phi(r)|^2\right)^{p /2} \right] \\
&+& C  \bE \left[ \left( \sup_{s\in [t,T]}  \bE^{\cF_s}\sup_{r\in[t,T]} \int_r^T  e^{\beta u}  |f^0(r,u)|^2 \dd u \right)^{p /2} \right]\\
&+&   C \sup_{1\leq k\leq n} \bE \left[ \left( \sup_{s\in [t,T]}  \bE^{\cF_s} \int_t^T  e^{\beta u} |Y_{k-1}(u)|^2   \dd u \right)^{p /2} \right].
\end{eqnarray*}
By Doob's maximal inequality
\begin{eqnarray*}
h_n(t) &\leq &  C  \bE \left( \sup_{s\in [t,T]} |\Phi(s)|^p \right)  + C  \bE \left[ \left( \sup_{s\in [t,T]}  \int_s^T |f^0(s,u)|^2 \dd u \right)^{p /2} \right]\\
&+&   C \sup_{1\leq k\leq n} \bE \left[ \left( \int_t^T   |Y_{k-1}(u)|^2  \dd u \right)^{p /2} \right].
\end{eqnarray*}
Since $p \geq 2$, by Jensen's inequality, 
$$h_n(t) \leq   C+ C  \int_t^T   |h_{n}(u)| \dd u .$$
Gronwall's inequality leads to 
$$ \sup_{1\leq k\leq n} \bE \left[ \sup_{s\in [0,T]} \left| Y_k(s)\right|^p \right] \leq C$$
for any $n$, that is
$$ \sup_{ n \in \bN} \bE \left[ \sup_{s\in [0,T]} \left| Y_k(s)\right|^p \right] \leq C.$$

By Lemma \ref{lem:exp_ineq}, we also have for almost all $t\in [0,T]$
\begin{eqnarray*}
e^{\beta t} |Y_n(t)-Y_m(t) |^2& \leq &C  \bE^{\cF_t} \int_t^T  e^{\beta s}   |Y_{n-1}(s)-Y_{m-1}(s)|^2 \dd s.
\end{eqnarray*}
Define
$$h(t) = \limsup_{m,n \to +\infty} \bE \left[ \sup_{s\in [t,T]}  |Y_n(s)-Y_m(s) |^2 \right].$$
Arguing as above, with Fatou's lemma and the preceding uniform (in $n$ and $s$) estimate, we have
$$h(t) \leq C \int_t^T h(s) \dd s \Longrightarrow h(t) = 0.$$
Hence there is a c\`adl\`ag adapted process $Y$ such that 
$$\lim_{n \to +\infty} \bE \left[ \sup_{s\in [0,T]}  |Y_n(s)-Y(s) |^2 \right] = 0.$$
And from the above estimate, we have $\displaystyle \bE \left[ \sup_{s\in [0,T]} \left| Y(s)\right|^p \right] \leq \infty.$
This achieves the proof of this theorem.
\end{proof}


\subsection{Comparison principle in the It\^o setting}

The results of Propositions \ref{prop:comp_BSVIE_1} and \ref{prop:comp_BSVIE_2} remain true here. The condition \eqref{eq:cond_comp_BSVIE_5} holds in this case if the free terms $\Phi^i$ belong to $C([0,T],\bL^2(\Omega))$, as in \cite{wang:yong:15}.

\section{Extension for Type-II BSVIE.}

The aim of this section is the study of the Type-II BSVIE \eqref{eq:general_BSVIE}
\begin{eqnarray*} \nonumber
Y(t) &= &\Phi(t) + \int_t^{T} f(t,s, Y(s), Z(t,s) ,Z(s,t),U(t,s),U(s,t)) \dd B_s  \\ 
&-& \int_t^{T} Z(t,s) \dd X^\circ_s -  \int_t^{T}  \int_{\bR^m} U(t,s,x) \tpi^\natural(\dd s,\dd x) -\int_t^{T} \dd M(t,s)\ ,
\end{eqnarray*}
or in the It\^o setting, the BSVIE \eqref{eq:Ito_general_BSVIE}
\begin{eqnarray*} \nonumber
Y(t) &= &\Phi(t) + \int_t^{T} f(t,s, Y(s), Z(t,s) ,Z(s,t),U(t,s),U(s,t)) \dd s  \\
&-& \int_t^{T} Z(t,s) \dd W_s -  \int_t^{T}  \int_{\bR^m} U(t,s,x) \tpi(\dd s,\dd x) -\int_t^{T} \dd M(t,s).
\end{eqnarray*}


We adapt the definitions of the section \ref{sect:general_Type_I_BSVIE}. Here $B$ and $\alpha$ (see Hypothesis \ref{H3}) are supposed to be deterministic. For the free term $\Phi$ and the first component $Y$ of the solution, we keep the spaces $\bL^2_{\beta,\cF_T}(0,T)$ and $\bL^2_{\beta,\bF}(0,T)$. Here we need to control the martingale terms not only on $\Delta^c(0,T)$, but on the whole set $[0,T]^2$. Thereby we define 
$$\cH_{\delta\leq \gamma}^2(R,T)$$
the set of processes $M(\cdot,\cdot)$ such that for $\dd B$-a.e. $t\in [R,T]$, $M(t,\cdot)$ belongs to $\cH^2$ and 
$$ \bE \left[ \int_R^T e^{(\gamma-\delta) A_t} \int_R^T e^{\delta A_s} \dd  \trace \langle M(t,\cdot)\rangle_s    \dd A_t \right] <+\infty.$$
Note that this definition implies that $M(t,s)$ is $\cF_s$-measurable for any $s \in [R,T]$. 
In the particular cases where $M(t,\cdot) = \int_S^\cdot Z(t,s) \dd X^\circ_s$ (resp. $M(t,\cdot) = U(t,\cdot) \star \tpi^\natural$), then $M \in \cH_{\gamma,\delta}^2(R,T)$
if
$$  \bE \left[ \int_R^T e^{(\gamma-\delta) A_t} \int_R^T e^{\delta A_s} \left\| b_s Z(t,s)\right\|^2 \dd B_s \dd A_t  \right] < +\infty,$$
(resp. 
$$\bE \left[ \int_R^T e^{(\gamma-\delta) A_t} \int_R^T e^{\delta A_s} \vertiii{U(t,\cdot)}_s^2 \dd B_s \dd A_t  \right] < +\infty.)$$
In this case $Z \in \bH^{2,\circ}_{\delta\leq \gamma}(R,T)$ (resp. $U \in \bH^{2,\natural}_{\delta\leq \gamma}(R,T)$). Finally we define
\begin{eqnarray*}
\mathfrak{S}^2_{\delta\leq \gamma}(R,T)&=& \bL^2_{\gamma,\bF}(R,T) \times\bH^{2,\circ}_{\delta\leq \gamma}(R,T) \times \bH^{2,\circ}_{\delta\leq \gamma}(R,T) \times \cH^{2,\perp}_{\delta\leq \gamma}(R,T)
\end{eqnarray*}
with the naturally induced norm. If $\delta=\delta(\gamma)$ is a known function of $\gamma$, we denote $\mathfrak{S}^2_{\delta\leq \gamma}(R,T)$  by $\mathfrak{S}^2_{\gamma}(R,T)$.

\subsection{Type-II BSVIE for deterministic $A$.}

Therefore we restrict ourselves to the case where $B$ is deterministic. Note that from \cite[Theorem II.4.15]{jaco:shir:03}, the semimartingale $\overline{X}$ has deterministic characteristics (that is $B$ and the two others components) if and only if it is a process with independent increments. 

The conditions \ref{H2} and \ref{H3} are modified as follows.
\begin{enumerate}[label=\textbf{(H\arabic*')}]
\setcounter{enumi}{1}
\item \label{H2prime} The generator $f$ is defined on $\Omega \times \Delta^c(0,T) \times \bR^{d} \times (\bR^{d\times m})^2 \times (\mathfrak H)^2 \to \bR^d$ and we assume that for any fixed $(t,y,z,\zeta,u,\nu)$ the process $f(t,\cdot,y,z,\zeta,u,\nu)$ is progressively measurable. Moreover there exist
$$\varpi:(\Omega\times \Delta^c(0,T), \cP) \to \bR_+, \qquad \vartheta = (\theta^\circ,\theta^\natural) : (\Omega\times \Delta^c(0,T), \cP) \to (\bR_+)^2$$
such that for $\dd B \otimes \dd B \otimes \dd \bP$-a.e. $(t,s,\omega)$
\begin{eqnarray*}
&& |f(\omega,t,s,y,z,\zeta,u_s(\omega;\cdot),\nu_s(\omega;\cdot)) - f(\omega,t,s,y',z',\zeta',u'_s(\omega;\cdot),\nu'_s(\omega;\cdot)) |^2 \\
&&\quad \leq \varpi_{t,s}(\omega)|y-y'|^2 + \theta^\circ_{t,s}(\omega) \left( \|b_s(\omega)(z-z')\|^2 + \|b_t(\omega)(\zeta-\zeta')\|^2\right) \\
&&\qquad + \theta^\natural_{t,s}(\omega) \left[ (\vertiii{u_s(\omega;\cdot)-u_s(\omega;\cdot)}_s(\omega))^2 + (\vertiii{ \nu_s(\omega;) - \nu'_s(\omega;)}_{t}(\omega))^2 \right].
\end{eqnarray*}

\item \label{H3prime} There exists a deterministic and non-decreasing\footnote{If not, we can replace $\alpha$ by $\sup_{u\in [0,s]} \alpha_u$. This increases the size of $A$ and requires more integrability in the assumptions \ref{H1} and \ref{H4}.} $\alpha:[0,T]\to \bR_+$ such that a.s. for any $s\in [0,T]$ and $t \in [0,s]$,  
$$\alpha^2_s \geq \max(\sqrt{\varpi_t(\omega)},\theta^\circ_{t,s}(\omega),\theta^\natural_{t,s}(\omega))>0.$$

The rest of the condition \ref{H3} (or \ref{F3}) remains unchanged. 
\end{enumerate}
In particular \ref{H3star} implies \ref{H3prime}. If \ref{H3prime} holds, then instead of \ref{H5}, we have that $\dd B$-almost every $t \in [0,T]$ belongs to the set $\mathfrak T^{\Phi,f}_{\delta}$ defined by \eqref{eq:subset_times_square_int}. This assumption \ref{H3prime} is true for example for the setting of Example \ref{exp:brownian_poisson_setting} if the generator has bounded stochastic Lipschitz coefficients. 

\begin{Rem}
Let us emphasize that our assumption {\rm \ref{H2prime}} implies implicitly that the process $\nu_s$ is in fact $\cF_t$-measurable.
\end{Rem}

\begin{Rem}[On Condition \ref{H3prime}]
Denote $L(t,s)^2 =  \max(\theta^\circ(\omega,t,s),\theta^\natural(\omega,t,s))$. 
Assuming that $\alpha$ does not depend on $t$ implies that the $\sup_{t\in [0,s]} L(t,s)^2 \leq \alpha_s^2$ is integrable. If we compare with the conditions imposed in \cite{wang:12,yong:08}, this assumption is stronger:
$$\sup_{t\in [0,T]} \int_t^T L(t,s)^2 \dd B_s \leq  \int_t^T \sup_{t\in [0,T]} L(t,s)^2 \dd B_s \leq \int_t^T \alpha_s^2 \dd B_s.$$
\end{Rem}

Let us study \ref{H1} and \ref{H4} under the preceding conditions. We can rewrite \ref{H1} as follows:
$$A_T < +\infty, \quad \mbox{and} \quad \bE \int_0^T e^{\beta A_t} |\Phi(t)|^2 \dd B_t < +\infty.$$
Moreover if $A_T < +\infty$, the estimate 
$$\bE \int_0^T \left( |\Phi(t)|^2 + \int_t^T \dfrac{|f^0(t,s)|^2}{\alpha_s^2} \dd B_s \right) \dd B_t < +\infty$$
leads to \ref{H1} and \ref{H4}. Let us recall again that $\delta=\delta^*(\beta)$ is defined in Lemma \ref{lem:infimum}. 
\begin{Thm} \label{thm:type_II_general_BSVIE}
If {\rm \ref{H1}}, {\rm \ref{H2prime}}, {\rm \ref{H3prime}} and {\rm \ref{H4}} hold and if the constants
$\kappa^\mathfrak f(\delta)$, $M^{\mathfrak f}(\beta)$ and $ \Sigma^\mathfrak f(\beta)$ defined by \eqref{eq:cond_exist_beta_poss}, in Lemma \ref{lem:infimum} and  by \eqref{eq:op_constants_beta_f} verify
\begin{equation} \label{eq:cond_beta_Type_II_BSVIE}
\kappa^\mathfrak f(\delta) < \dfrac{1}{2},\qquad M^{\mathfrak f}(\beta) < \dfrac{1}{2}, \qquad \left( 16 + \dfrac{e^{(\beta-\delta) \mathfrak f}}{\beta-\delta}\right) \Sigma^\mathfrak f(\beta) < 1,
\end{equation}
then the Type-II BSVIE \eqref{eq:general_BSVIE} has a unique adapted M-solution $(Y,Z,U,M)$ in $ \mathfrak{S}^2_{\beta}(0,T)$. 
\end{Thm}

The condition \eqref{eq:cond_beta_Type_II_BSVIE} is much stronger than the assumption \eqref{eq:cond_beta_Type_I_BSVIE}. Indeed for large values of $\beta$, \eqref{eq:cond_beta_Type_II_BSVIE} holds if 
$$e \mathfrak f < \dfrac{1}{3 (51 + \sqrt{2603})}\approx 0,0033,$$
whereas $ e\mathfrak f < 3 (\sqrt{11}-3)/18 \approx 0,052$ is sufficient for $\eqref{eq:cond_beta_Type_I_BSVIE}$ (see Remark \eqref{rem:large_values_beta}). In other words $A$ may be discontinuous, but with very small jumps. 


To give an idea of the value of $\beta$, let us assume that $\mathfrak f = 0$, that is $A$ is continuous. Then $\delta = \delta^*(\beta)= \dfrac{\sqrt{11}}{3+\sqrt{11}}\beta $. The condition \eqref{eq:cond_exist_beta_poss} becomes
$$\frac{9}{\delta} + \dfrac{4(2+9\delta)}{\delta^2} < \dfrac{1}{2}.$$
And 
$$M^0(\beta) = \frac{(\sqrt{11}+3)^2}{\beta}\approx \frac{40}{\beta},\quad \widetilde \Sigma^0(\beta) = \dfrac{2(\sqrt{11}+3)^3}{3\beta (\beta - 2 (\sqrt{11}+3)^2)} .$$
Some tedious computations show that $\beta > 174$ is sufficient for our condition \eqref{eq:cond_beta_Type_I_BSVIE}.
Nonetheless
$$ \left( 16 + \dfrac{1}{\beta-\delta}\right) \Sigma^0(\beta)  =\left( 16 + \dfrac{\sqrt{11}+3}{3\beta}\right) \dfrac{2(\sqrt{11}+3)^2}{\beta - 2 (\sqrt{11}+3)^2} .$$
This leads to $\beta >1357$ in order to satisfy \eqref{eq:cond_beta_Type_II_BSVIE}.
%

\subsubsection*{Proof of Theorem \ref{thm:type_II_general_BSVIE}}

From now on, $A$ is deterministic (Hypothesis \ref{H3prime}). If for some $\gamma \in \bR$, $\bE \int_S^T e^{\gamma A_s}|Y(s)|^2 \alpha_s^2 \dd B_s< +\infty$, then for $\dd B$-almost every $t \in [S,T]$, we can define $\mart$ by \eqref{eq:def_mart_Delta_S_T}:
$$Y(t)  = \bE \left[ Y(t)  \bigg| \cF_S \right]  + \int_S^t \dd \mart(t,s).$$
%
Since
\begin{eqnarray*}
\int_S^t e^{\delta A_r} \dd \trace \langle \mart(t,\cdot)\rangle_r & \leq & \delta \int_S^t e^{\delta A_u} \int_u^t \dd \trace \langle \mart(t,\cdot)\rangle_r \dd A_u \\
&+& e^{\delta A_S} \left( \trace \langle \mart(t,\cdot)\rangle_t - \trace \langle \mart(t,\cdot)\rangle_S \right) ,
\end{eqnarray*}
we have
\begin{eqnarray*}
&& \bE \left[ \int_S^T e^{(\gamma-\delta)A_t}  \int_S^t e^{\delta A_u}\int_u^t \dd \trace \langle \mart(t,\cdot)\rangle_s \dd A_u  \dd A_t \right] \\
&& \quad =  \bE \left[  \int_S^T e^{(\gamma-\delta)A_t}   \int_S^t  e^{\delta A_u} \bE \left[ (Y(t) - \bE[Y(t)|\cF_u])^2 |\cF_u\right]  \dd A_u\dd A_t\right] \\
&&\quad \leq 2 \bE \left[  \int_S^T e^{(\gamma-\delta)A_t}   \int_S^t  e^{\delta A_u} Y(t)^2   \dd A_u\dd A_t\right] \\
&&\quad \leq \frac{2}{\delta}  \bE \left[  \int_S^T e^{\gamma A_t}  Y(t)^2  \dd A_t\right] .
\end{eqnarray*}
Let us emphasize that the first equality holds only because $A_t$ is $\cF_u$-measurable even if $t > u$. Similarly since $A_S \leq A_t$
\begin{eqnarray*}
&& \bE \left[ \int_S^T e^{(\gamma-\delta)A_t} e^{\delta A_S} \left( \trace \langle \mart(t,\cdot)\rangle_t - \trace \langle \mart(t,\cdot)\rangle_S \right) \dd A_t  \right] \\
&& \quad =  \bE \left[  \int_S^T e^{(\gamma-\delta)A_t}    e^{\delta A_S} \bE \left[ (Y(t) - \bE[Y(t)|\cF_S])^2 |\cF_S\right] \dd A_t\right] \\
&&\quad \leq 2 \bE \left[  \int_S^T e^{(\gamma-\delta)A_t}e^{\delta A_S} Y(t)^2 \dd A_t\right]  \leq 2 \bE \left[  \int_S^T e^{\gamma A_t}  Y(t)^2  \dd A_t\right] .
\end{eqnarray*}
Hence we obtain that 
\begin{equation} \label{eq:estim_mart_Delta_S_T}
 \bE \left[ \int_S^T e^{(\gamma-\delta)A_t}\int_S^t e^{\delta A_r} \dd \trace \langle \mart(t,\cdot)\rangle_r   \dd A_t \right]  \leq 4 \bE \left[  \int_S^T e^{\gamma A_t}  Y(t)^2  \dd A_t\right].
\end{equation}

\bigskip 
Let us consider the space $\widehat{\mathfrak{S}}^2_{\delta\leq \gamma}(R,T)$ be the space of all $(y,z,u,m)$ in $\mathfrak{S}^2_{\delta\leq \gamma}(R,T)$ such that for $\dd B$-a.e. $t \in [R,T]$ a.s. 
\begin{eqnarray*}
\int_R^t y(r) \dd B_r & = & \bE \left[ \int_R^t y(r)\dd B_r |\cF_R \right] \\
& + & \int_R^t z(t,r)\dd X^\circ_r + \int_R^t \int_{\bR^m} u(t,r,x)\tpi^\natural(\dd r,\dd x) + \int_R^t \dd m(t,r) .
\end{eqnarray*}
From \eqref{eq:estim_mart_Delta_S_T}, we have for any $\delta \in (0,\beta]$ 
\begin{eqnarray*}
&& \bE \left[ \int_R^T e^{(\beta-\delta) A_t}   \int_R^t e^{\delta A_r} |b_r z(t,r)|^2 \dd B_r \dd A_t + \int_R^T e^{(\beta-\delta) A_t}   \int_R^t e^{\delta A_r} \vertiii{u(t,r)}_r \dd B_r \dd A_t \right. \\
&&\quad \left.+ \int_R^T e^{(\beta-\delta) A_t} \int_S^t e^{\delta A_r} \dd \trace \langle m(t,\cdot)\rangle_r  \dd A_t \right]  \leq 4 \bE \left[ \int_R^T e^{\beta A_t} |y(t)|^2 \dd A_t\right].
\end{eqnarray*}
As in \cite{yong:08} we take on $\widehat{\mathfrak{S}}^2_{\delta\leq \beta}(0,T)$ the norm of $\mathfrak{S}^2_{\delta\leq \beta}(\Delta^c(0,T))$
\begin{eqnarray*}
\|(y,z,u,m)\|_{\widehat{\mathfrak{S}}^2_{\delta\leq \beta}}^2 &= &\bE \left[ \int_0^T e^{\beta A_t} |y(t)|^2 \dd A_t \right. \\
&+& \int_0^Te^{(\beta-\delta)A_t } \left(  \int_t^T e^{\delta A_r} \dd \trace [\langle z(t,\cdot)\cdot X^\circ\rangle_r ] \right) \dd A_t  \\
&+&\int_0^T e^{(\beta-\delta)A_t }\left( \int_t^T e^{\delta A_r} \dd \trace [\langle u(t,\cdot)\star \tpi^\natural \rangle_r ]  \right) \dd A_t\\
&+& \left. \int_0^T  e^{(\beta-\delta)A_t } \left( \int_t^T e^{\delta A_r} \dd \trace [\langle m(t,\cdot) \rangle_r ] \right) \dd A_t  \right] .
\end{eqnarray*}
We have proved that 
\begin{eqnarray} \nonumber
 \|(y,z,u,m)\|_{\widehat{\mathfrak{S}}^2_{\delta\leq \beta}}^2 & \leq & \bE \left[ \int_0^T e^{\beta A_t} |y(t)|^2 \dd A_t \right. \\ \nonumber
&+& \int_0^Te^{(\beta-\delta)A_t } \left(  \int_0^T e^{\delta A_r} \dd \trace [\langle z(t,\cdot)\cdot X^\circ\rangle_r ] \right) \dd A_t  \\ \nonumber
&+&\int_0^T e^{(\beta-\delta)A_t }\left( \int_0^T e^{\delta A_r} \dd \trace [\langle u(t,\cdot)\star \tpi^\natural \rangle_r ]  \right) \dd A_t\\ \nonumber
&+& \left. \int_0^T  e^{(\beta-\delta)A_t } \left( \int_0^T e^{\delta A_r} \dd \trace [\langle m(t,\cdot) \rangle_r ] \right) \dd A_t  \right]\\ \label{eq:norm_equivalence_p_2}
& \leq & 5 \|(y,z,u,m)\|_{\widehat{\mathfrak{S}}^2_{\delta\leq \beta}}^2.
\end{eqnarray}
Hence we have an equivalent norm for $\widehat{\mathfrak{S}}^2_{\delta\leq \beta}(0,T)$.

Fix $\Phi \in\bL^2_{\beta,\cF_T}(0,T)$ and $(y,\zeta,\nu,m) \in \widehat{\mathfrak{S}}^2_{\delta\leq \beta}(0,T)$ and consider the BSVIE
\begin{eqnarray} \nonumber
Y(t) &= &\Phi(t) + \int_t^{T} f(t,s, y(s),Z(t,s) , \zeta(s,t),U(t,s),\nu(s,t)) \dd B_s - \int_t^{T} Z(t,s) \dd X^\circ_s  \\ \label{eq:fix_point_type_II_BSVIE}
&& \qquad - \int_t^{T}  \int_{\bR^m} U(t,s,x) \tpi^\natural(\dd s,\dd x) -\int_t^{T} \dd M(t,s),
\end{eqnarray}
This is a particular case of the BSVIE \eqref{eq:special_BSVIE} (and of the Type-I BSVIE \eqref{eq:general_BSVIE_type_I}). 
We want to apply Lemma \ref{lem:param_BSDE->BSVIE}. If 
$$h(t,s,z,\psi)= f(t,s, y(s),z, \zeta(s,t),\psi,\nu(s,t)),$$
we need to check the condition \eqref{eq:int_cond_h_0}:
$$\bE \int_0^T e^{(\beta-\delta) A_t} \left( \int_0^T e^{\delta A_s} \dfrac{|h(t,s,0,\mathbf 0)|^2}{\alpha_s^2} \dd B_s \right) \dd B_t < +\infty.
$$
The Lipschitz property \ref{H2prime} of $f$ leads to
$$ |h(t,s,0,\mathbf 0)|^2  \leq 2|f^0(t,s)|^2 + 2\varpi_s |y(s)|^2 + 2\theta^\circ_s |b_t \zeta(s,t)|^2 + 2 \theta^\natural_s \vertiii{ \nu(s,t) }_{t}$$
Note that for $s\geq t$, $\zeta(s,t)$ and $\nu(s,t)$ are $\cF_t$-measurable. Moreover
\begin{eqnarray*}
&& \bE \left[ \int_0^T e^{(\beta-\delta) A_t} \int_t^T e^{\delta A_s}  \frac{1}{\alpha_s^2} \theta^\circ_s |b_t \zeta(s,t)|^2 \dd B_s \dd A_t \right] \\
&&\quad \leq \bE \left[ \int_0^T e^{(\beta-\delta) A_t} \int_t^T e^{\delta A_s}  |b_t \zeta(s,t)|^2 \dd B_s \alpha_t^2 \dd B_t \right] \\
&&\quad \leq \bE \left[ \int_0^T e^{\delta A_s} \int_0^s e^{(\beta-\delta) A_t}   \alpha_t^2 |b_t \zeta(s,t)|^2 \dd B_t \dd B_s \right] \\
&&\quad  \leq \bE \left[ \int_0^T e^{\delta A_s}   \alpha_s^2  \int_0^s e^{(\beta-\delta) A_t} |b_t \zeta(s,t)|^2 \dd B_t \dd B_s \right] \\
&&\quad  \leq \bE \left[ \int_0^T e^{\delta A_s}    \int_0^s e^{(\beta-\delta) A_t}\dd \trace [\langle \zeta(s,\cdot)\cdot X^\circ\rangle_t ]  \dd A_s \right]  \leq 4 \bE \left[  \int_S^T e^{\beta A_t}  y(t)^2  \dd A_t\right].
\end{eqnarray*}
since $\alpha$ is supposed to be non-decreasing. Similarly
\begin{eqnarray*}
&& \bE \left[ \int_0^T e^{(\beta-\delta) A_t} \int_t^T e^{\delta A_s}  \frac{1}{\alpha_s^2} \theta^\natural_{t,s} \vertiii{ \nu(s,t) }_{t}^2 \dd B_s \dd A_t \right] \\
&&\quad \leq \bE \left[ \int_0^T e^{\delta A_s} \int_0^t e^{(\beta-\delta) A_t}   \alpha_t^2 \vertiii{ \nu(s,t) }_{t}^2 \dd B_t \dd B_s \right] \\
&&\quad  \leq \bE \left[ \int_0^T e^{\delta A_s}   \alpha_s^2  \int_0^s e^{(\beta-\delta) A_t} \vertiii{ \nu(s,t) }_{t}^2 \dd B_t \dd B_s \right] \\
&&\quad  \leq \bE \left[ \int_0^T e^{\delta A_s}    \int_0^s e^{(\beta-\delta) A_t}\dd \trace [\langle \nu(s,\cdot)\star \tpi^\natural \rangle_t ]  \dd A_s \right]  \leq 4 \bE \left[  \int_0^T e^{\beta A_t}  y(t)^2  \dd A_t\right].
\end{eqnarray*}
Thus the BSVIE \eqref{eq:fix_point_BSVIE} has a unique adapted M-solution $(Y,Z,U,M) \in \mathfrak{S}^2_{\delta\leq \beta}(\Delta^c)$. 
Moreover from \eqref{eq:H2-estimate} and the preceding estimates
\begin{eqnarray*} \nonumber
&& \bE \left[ \int_0^T e^{ \beta A_t} |Y(t)|^2 \dd A_t +  \int_0^T e^{(\beta-\delta) A_t}  \int_t^T e^{\delta A_s} \dd  \trace \langle \mart(t,\cdot)\rangle_s    \dd A_t \right] \\  \nonumber
&& \quad  \leq\dfrac{\delta}{2}  \Sigma^\mathfrak f( \beta) \bE \left[ \int_0^T e^{\beta A_t} |\Phi(t)|^2  \dd A_t\right] \\ 
&&\qquad + \Sigma^\mathfrak f(\beta)  \bE \left[ \int_0^T  e^{(\beta-\delta) A_t}  \int_t^T e^{\delta A_s}\dfrac{ |h^0(t,s)|^2}{\alpha_s^2} \dd B_s  \dd A_t \right] \\
&& \quad  \leq\dfrac{\delta}{2}  \Sigma^\mathfrak f( \beta) \bE \left[ \int_0^T e^{\beta A_t} |\Phi(t)|^2  \dd A_t\right] \\ 
&&\qquad + \Sigma^\mathfrak f(\beta)  \bE \left[ \int_0^T  e^{(\beta-\delta) A_t}  \int_t^T e^{\delta A_s}\dfrac{ |f^0(t,s)|^2}{\alpha_s^2} \dd B_s  \dd A_t \right]\\
&&\qquad + 18 \Sigma^\mathfrak f(\beta)  \bE \left[  \int_0^T e^{\beta A_t}  y(t)^2  \dd A_t\right].
\end{eqnarray*}
Then using \eqref{eq:def_mart_Delta_S_T} and \eqref{eq:estim_mart_Delta_S_T} we define and control $Z$, $U$ and $M$ on $\Delta$. Therefore $(Y,Z,U,M) \in  \mathfrak{S}^2_{\delta\leq \beta}(0,T)$. In other words we have a map $\Theta$ from $ \mathfrak{S}^2_{\delta\leq \beta}(0,T)$ into itself. 

Thanks to Lemma \ref{lem:stab_BSVIE_h_poss}, if $(\bar Y,\bar Z,\bar U,\bar M)$ is another solution of the BSVIE \eqref{eq:fix_point_BSVIE} with data $(\bar y, \bar \zeta, \bar \nu, \bar m)$, then with
$$(\what Y,\what Z, \what Y, \what M) = (Y-\bar Y, Z-\bar Z,U-\bar U,M-\bar M),$$
the estimate \eqref{eq:stability_H2-estimate} becomes
\begin{eqnarray*} \nonumber
&& \bE \left[ \int_0^T e^{ \beta A_t} |\what Y(t)|^2 \dd A_t +  \int_0^T e^{(\beta-\delta) A_t}  \int_t^T e^{\delta A_s} \dd  \trace \langle \what \mart(t,\cdot)\rangle_s    \dd A_t \right] \\  \nonumber
&& \quad  \leq \Sigma^\mathfrak f(\beta)  \bE \left[ \int_0^T  e^{(\beta-\delta) A_t}  \int_t^T e^{\delta A_s}\dfrac{ |\what h(t,s)|^2}{\alpha_s^2} \dd B_s  \dd A_t \right] 
\end{eqnarray*}
where using \ref{H2prime}
\begin{eqnarray*} 
|\what h(t,r)|^2 &=& |f(t,r,y(r),Z(t,r),\zeta(r,t),U(t,r),\nu(r,t))-f(t,r,\bar y(r),Z(t,r),\bar \zeta(r,t),U(t,r),\bar \nu(r,t))|^2 \\
& \leq & \varpi_r |\what y(r)|^2 + \theta^\circ_r| b_t \what \zeta(r,t)|^2 + \theta^\natural_r \vertiii{ \nu(r,t) }^2_{t}.
\end{eqnarray*}
The preceding computations show that 
\begin{eqnarray*} \nonumber
&& \bE \left[ \int_0^T e^{ \beta A_t} |\what Y(t)|^2 \dd A_t +  \int_0^T e^{(\beta-\delta) A_t}  \int_t^T e^{\delta A_s} \dd  \trace \langle \what \mart(t,\cdot)\rangle_s    \dd A_t \right] \\  \nonumber
&& \quad  \leq \left( 16 + \dfrac{e^{(\beta-\delta)\mathfrak f}}{\beta-\delta} \right) \Sigma^\mathfrak f(\beta)  \bE \left[ \int_0^T  e^{\beta A_t}  |\what y(t)|^2 \dd A_t \right] .
\end{eqnarray*}
Hence $\Theta$ is a contraction if 
$$ \left( 16 + \dfrac{e^{(\beta-\delta)\mathfrak f}}{\beta-\delta} \right) \Sigma^\mathfrak f(\beta) < 1,$$
and we obtain the existence and uniqueness of a solution $(Y,Z,U,M) \in \mathfrak{S}^2_{\beta}(0,T)$ to the Type-II BSVIE \eqref{eq:general_BSVIE}.

\subsection{Existence and uniqueness for the Type-II BSVIE \eqref{eq:Ito_general_BSVIE}}

Now we come back to the BSVIE \eqref{eq:Ito_general_BSVIE}. Here $B_t =t$ and if \ref{H3star} holds, then $1 \leq e^{\beta A_t} \leq e^{\beta KT}$. Therey if $\Phi \in \bL^2_{\cF_T}(0,T)$ and if 
\begin{equation}\label{eq:int_cond_f_0_p_2}
\bE \int_0^T \left( \int_t^T |f^0(t,s)| \dd s \right)^2 \dd t < +\infty,
\end{equation}
then \ref{H1}, \ref{H4} (and \ref{H5}) hold. The next result is a corollary of Theorem \ref{thm:type_II_general_BSVIE}. 
\begin{Prop} \label{thm:H2_M_solution}
Assume that $\Phi \in \bL^2_{\cF_T}(0,T)$, that {\rm \ref{H2}}, {\rm \ref{H3star}} and \eqref{eq:int_cond_f_0_p_2}
 hold\footnote{The space $\mathfrak H$ in \ref{H2} is replaced by $\bL^2_\mu$ in this case.}. Then the BSVIE \eqref{eq:Ito_general_BSVIE} has a unique adapted M-solution $(Y,Z,U,M)$ in $\mathfrak S^2(0,T)$ on $[0,T]$. Moreover for any $S\in[0,T]$
\begin{eqnarray} \nonumber
 \|(Y,Z,U,M)\|_{\mathfrak S^2(S,T)}^2& = & \bE \left[ \int_S^T  |Y(t)|^2  \dd t + \int_S^T \left(  \int_S^T  |Z(t,r)|^2 \dd r\right)  \dd t  \right. \\ \nonumber
&& \quad \left.+ \int_S^T \left(  \|U(t,\cdot)\|^2_{\bL^2_\pi(S,T)}  \right)  \dd t +\int_S^T \left( \langle M(t,\cdot) \rangle_{S,T}\right) \dd t\right]  \\  \label{eq:H2_estimate}
& \leq & C \bE \left[ \int_S^T |\Phi(t)|^2  \dd t +\int_S^T \left( \int_t^T |f^0(t,r)|  \dd r \right)^2  \dd t \right].
\end{eqnarray}
\end{Prop}
Recall that 
$$ \|U(t,\cdot)\|^2_{\bL^2_\pi(S,T)} =   \int_S^T \int_{\bR^m} |U(t,r,x)|^2\pi( \dd r, \dd x).$$
\begin{proof}
Again this proposition is a particular case of Theorem \ref{thm:type_II_general_BSVIE}. Nonetheless we could also give a direct proof, following the scheme of the proof of Theorem \ref{thm:Hp_M_solution}, that is of \cite[Theorem 3.7]{yong:08}. The modifications are quite obvious. In Step 1, 
we fix $\Phi \in \bL^2_{\cF_T}(S,T)$ and $(y,\zeta,\nu,m) \in \widehat{\mathfrak S}^2(S,T)$ and consider the BSVIE on $[S,T]$
\begin{eqnarray} \nonumber
Y(t) &= &\Phi(t) + \int_t^{T} f(t,s, y(s),Z(t,s) , \zeta(s,t),U(t,s),\nu(s,t)) \dd s - \int_t^{T} Z(t,s) \dd W_s  \\ \label{eq:Ito_fix_point_BSVIE}
&& \qquad - \int_t^{T}  \int_{\bR^m} U(t,s,x) \tpi(\dd s,\dd x) -\int_t^{T} \dd  M(t,s). 
\end{eqnarray}
We apply Lemma \ref{lem:Ito_param_BSDE->BSVIE} and the inequalities \eqref{eq:tech_estim_step_1} and 
\eqref{eq:Theta_contraction} become
\begin{eqnarray*} \nonumber
&& \bE \left[ \int_S^T  |Y(t)|^2 \dd t + \int_S^T \left(  \int_t^T  |Z(t,r)|^2 \dd r \right) \dd t +\int_S^T \left(  \langle M(t,\cdot) \rangle_{t,T} \right) \dd t \right. \\ \nonumber
&& \qquad \left.+ \int_S^T \left( \int_t^T  \|U(t,r)\|^2_{\bL^2_\mu}  \dd r \right) \dd t \right]  \\  \nonumber
&& \quad \leq C \bE \left[ \int_S^T |\Phi(t)|^2 dt +\int_S^T \left( \int_t^T |f^0(t,r)| \dd r \right)^2 \dd t+ \int_S^T |y(r)|^2 \dd r  \right. \\ \label{eq:cM_p_estim_p_2}
&& \hspace{1.5cm}  \left.+ \int_S^T\left(  \int_t^T |\zeta(r,t)|^2 \dd r \right) \dd t +\int_S^T \left(  \int_t^T \|\nu(r,t)\|^2_{\bL^2_\mu} \dd r \right) \dd t \right]
\end{eqnarray*}
and
\begin{eqnarray*} \nonumber
&& \bE \left[ |\fd Y(t)|^2 +   \int_S^T  |\fd Z(t,r)|^2 \dd r +  \langle \fd M(t,\cdot) \rangle_{S,T}+   \int_S^T  \int_{\bR^m} |\fd U(t,r,x)|^2 \pi(\dd r, \dd x) \right]  \\
&& \quad   \leq C K^2 (T-S) \bE \left[ \left( \int_S^T  |\fd y(r)|^2 \dd r \right)  +  \left( \int_S^T \left( |\fd \zeta(t,r)|^2 + \|\what \nu(t,r) \|^2_{\bL^2_\pi} \right) \dd r  \right) \right].
\end{eqnarray*}

The second step remains unchanged, whereas in the third step, we apply
Lemma \ref{lem:Ito_SFIE_sol}, with
$$f^S(t,s,z,u) = f(t,s,Y(s),z,Z(s,t),u,U(s,t)), \quad (t,s,z,u) \in [R,S]\times [S,T] \times \bR^k \times \bL^2_\mu.$$
The last two steps are almost the same; the modifications are straightforward. 

\end{proof}

Note that from \eqref{eq:BJ_ineq_p_=_2}, concerning $U$, Estimate \eqref{eq:H2_estimate} is completely equivalent to: for any $S\in [0,T]$
\begin{eqnarray} \nonumber
&& \bE \left[  \int_S^T \left( \int_S^T  \|U(t,r)\|^2_{\bL^2_\mu}  \dd r \right) \dd t \right] \leq C \bE \left[ \int_S^T |\Phi(t)|^2 dt +\int_S^T \left( \int_t^T |f^0(t,r)| \dd r \right)^2 \dd t \right].
\end{eqnarray}

\subsection{A duality result in the It\^o setting} \label{ssect:duality_princ}

The {\it duality principle} of linear stochastic integral equations (\cite[Section 4]{yong:06}) plays an important role for comparison principle or optimal control problem (see \cite[Section 5]{yong:08}). This result is based on the notion of FSVIE (see among many others \cite{berg:mize:80a,berg:mize:80b,kolo:83,kolo:84,pard:prot:90,prot:85}). In \cite{yong:06,yong:08}, the next FSVIE is considered:
 for $t \in [0,T]$
\begin{eqnarray*}
X(t) & = & \Psi(t) + \int_0^t A_0(t,s) X(s) \dd s + \int_0^t \sum_{i=1}^k A_i(t,s)X(s)\dd W_i(s),
\end{eqnarray*}
where $A_i(\cdot,\cdot) \in \bL^\infty([0,T]; \bL^\infty_\bF(0,T;\bR^{d\times d}))$ for $i=0,1,\ldots,k$. It means that $A_i : \Omega \times [0,T]^2  \to \bR^{d \times d}$ is bounded, $\cF_T \otimes \cB([0,T]^2)$-measurable and for almost all $t\in [0,T]$, $A_i(t,\cdot)$ is $\cF$-adapted. Then for any $\Psi \in \bL^2_\cF(0,T;\bR^d)$, there exists a unique solution $X$ in $\bL^2_\cF(0,T;\bR^d)$. 

Here we consider the extension: for $t \in [0,T]$
\begin{eqnarray} \nonumber
X(t) & = & \Psi(t) + \int_0^t A_0(t,s) X(s) \dd s + \int_0^t \sum_{i=1}^k A_i(t,s)X(s-)\dd W_i(s) \\ \label{eq:linear_FSVIE}
&+&\int_0^t \int_{\bR^m} \cA(t,s,x) X(s-) \tpi(\dd s,\dd x).
\end{eqnarray}
We keep the same conditions on the $A_i$. We assume that $\cA : \Omega \times [0,T]^2 \times \bR^m \to \bR^{d}$ is bounded  and such that for almost all $t \in [0,T]$, $B(t,\cdot,\cdot)X(\cdot)\star \tpi$ is well defined. Since we are interesting by c\`adl\`ag processes $X$, we use the setting of \cite[Condition 4.1]{prot:85}. Hence we also suppose that $A_i$ and $\cA$ are differentiable w.r.t. $t$ with a bounded derivative (uniformly in $(\omega,t,s)$). Thus we can apply \cite[Theorem 4.3]{prot:85}: if $\Psi$ is a c\`adl\`ag process, then there exists a unique c\`adl\`ag solution $X$ of the preceding FSVIE. 
The key point is that $X$ is a c\`adl\`ag process, hence for a.e. $t\in [0,T]$, $X(t) = X(t-)$. 
\begin{Lemma}
Let $\Psi(\cdot) \in \bL^2_\bF(0,T;\bR^d) \cap \bD^2(0,T)$ and $\Phi(\cdot) \in \bL^2((0,T)\times \Omega; \bR^d)$. Let $X \in \bL^2(0,T;\bR^d)$ be the c\`adl\`ag solution of the linear FSVIE \eqref{eq:linear_FSVIE}. We also consider the BSVIE:
\begin{eqnarray*} \nonumber
Y(t) &= &\Phi(t) -\int_t^T Z(t,s) \dd W_s - \int_t^{T}  \int_{\bR^m}  U(t,s,x) \tpi(\dd s,\dd x) -\int_t^{T} \dd M(t,s)\\
&+& \int_t^{T} \left[ A_0(s,t)^{\top} Y(s) + \sum_{i=1}^k A_i(s,t)^\top Z_i(s,t) + \int_{\bR^m} \cA(s,t,x)^\top U(s,t,x) \mu(\dd x)\right]\dd s  . 
\end{eqnarray*}
Then 
$$ \bE \int_0^T \left\langle \Psi(t), Y(t) \right\rangle \dd t  =  \bE \int_0^T \left\langle X(t), \Phi(t)   \right\rangle \dd t.$$
\end{Lemma}
\begin{proof}
The arguments are the same as \cite{over:roed:18,yong:08} and are based on the orthogonality of $W$, $\tpi$ and $M$. We skip the details here.  
\end{proof}

Let us emphasize that the role of the c\`adl\`ag property of $X$ is important here. Thus it should be possible to relax the regularity assumption on the coeffcients $A_i$ or $\cA$ of the FSVIE. But as for a BSVIE, the regularity of the paths of $X$ is not a direct property nor an easy stuff. 

Note that the extension of the duality result to the setting of Section \ref{sect:general_Type_I_BSVIE} is an issue. Indeed we first need a solution for the Type-II BSVIE \eqref{eq:general_BSVIE}. But the orthogonality between $B$, $X^\circ$, $\tpi^\natural$ and $M$ is much more delicate and several simplifications are not true anymore with these processes.

With the previous duality result, it is possible to extend the comparison principle for M-solution of a Type-II BSVIE of the form:
\begin{eqnarray*} \nonumber
Y(t) &= &\Phi(t) + \int_t^{T} (g(t,s, Y(s)) + C(s)Z(s,t) ) \dd s - \int_t^{T}  \dd \mart(t,s). 
\end{eqnarray*}
Up to some technical conditions, one can follow the scheme of \cite[Theorems 3.12 and 3.13]{wang:yong:15}. 

\section{Appendix}\label{sect:appendix}



\subsection{Proof of Lemma \ref{lem:infimum}}

Here we want to prove Lemma \ref{lem:infimum}. Recall that for $\delta < \gamma \leq \beta$:
$$ \Pi^{\mathfrak f}(\gamma,\delta) =   \frac{11}{\delta} + 9 \frac{e^{(\gamma-\delta)\mathfrak f}}{(\gamma-\delta)}.$$ 
The lemma states that the infimum of $\Pi^{\mathfrak f}(\gamma,\delta) $ over all $\delta < \gamma \leq \beta$ is given by 
$$ M^{\mathfrak f}(\beta)=\widetilde \Pi^{\mathfrak f}(\beta,\delta^*(\beta)),$$
where $\delta^*(\beta)$ is the unique solution on $(0,\beta)$ of the equation:
$$11(\beta-x)^2 - 9 e^{(\beta-x)\mathfrak f} x^2 (\mathfrak f(\beta-x)-1)=0.$$
For any fixed $\gamma \in (0,\beta]$, since 
$$\lim_{\delta \to 0}  \Pi^{\mathfrak f}(\gamma,\delta) = \lim_{\delta \to \gamma}  \Pi^{\mathfrak f}(\gamma,\delta) = +\infty,$$
there exists a $\delta^*(\gamma) \in (0,\gamma)$ such that
$$ \Pi^{\mathfrak f}(\delta^*(\gamma),\delta) = \inf_{0< \delta < \gamma}  \Pi^{\mathfrak f}(\gamma,\delta),$$
and $\delta^*(\gamma)$ is the critical point of 
$$\dfrac{\partial  \Pi^{\mathfrak f}}{\partial \delta} (\gamma,\delta)= -\frac{11}{\delta^2} + 9\dfrac{e^{(\gamma-\delta)\mathfrak f} }{(\gamma-\delta)^2}\left(- \mathfrak f (\gamma-\delta)+1\right),$$
that is 
$$-\frac{11}{(\delta^*(\gamma))^2} = 9\dfrac{e^{(\gamma-\delta^*(\gamma))\mathfrak f} }{(\gamma-\delta^*(\gamma))^2}\left( \mathfrak f (\gamma-\delta^*(\gamma))-1\right) < 0.$$
Note that we have necessarily 
$$\left(\gamma - \dfrac{1}{\mathfrak f} \right) \vee 0 < \delta^*(\gamma) < \gamma.$$

Let us differentiate w.r.t. $\gamma$:
$$\dfrac{\partial  \Pi^{\mathfrak f} }{\partial \gamma}  (\gamma,\delta)= 9e^{(\gamma-\delta)\mathfrak f} \dfrac{1}{\gamma-\delta}\left( \mathfrak f - \frac{1}{\gamma-\delta}\right)=9\dfrac{e^{(\gamma-\delta)\mathfrak f} }{(\gamma-\delta)^2}\left( \mathfrak f (\gamma-\delta) -1\right).$$
Hence a critical point should satisfy:
$$\gamma^*(\delta) = \delta + \dfrac{1}{\mathfrak f}.$$
It is admissible if and only $\beta > 1/\mathfrak f$. 
\begin{itemize}
\item Assume that $\beta > 1/\mathfrak f$. Then 
$$ \Pi^{\mathfrak f}(\gamma^*(\delta),\delta) = \dfrac{11}{\delta} +  9e\mathfrak f.$$
This quantity has no critical point on $0< \delta < \gamma < \beta$. Hence the infimum is attained on its boundary. 
\item If $\beta  \mathfrak f \leq 1$, the partial derivative w.r.t. $\gamma$ remains non positive. There is no critical point in this case. And again the infimum is attained at its boundary. 
\end{itemize}
The cases where one among $\delta$ or $\gamma$ goes to zero or where their difference goes to 0, lead to the value $+\infty$. The only remaining case is therefore $0< \delta < \gamma=\beta$. Then the minimum is attained at $\delta^*(\beta)$ and the infimum is equal to
$$ M^{\mathfrak f}(\beta) = \frac{11}{\delta^*(\beta)} + 9 \frac{e^{(\beta-\delta^*(\beta))\mathfrak f}}{(\beta-\delta^*(\beta))} .$$
Let us recall that $\delta^*(\beta)$ is the unique solution in the interval $((\beta - 1/\mathfrak f) \vee 0, \beta)$ of the equation
$$11(\beta-x)^2 - 9 e^{(\beta-x)\mathfrak f} x^2 (\mathfrak f(\beta-x)-1)=0.$$
Now let us compute the limit as $\beta$ goes to $\infty$. Since $\delta^* (\beta) \geq \beta - 1/\mathfrak f$, 
$$\lim_{\beta \to +\infty} \delta^*(\beta) = +\infty.$$
Hence since $\delta^*(\beta)$ solves the preceding equation, dividing by $x^2$, we obtain that 
$$\dfrac{11}{9}\left(\dfrac{\beta-\delta^*(\beta)}{\delta^*(\beta)}\right)^2e^{-(\beta-\delta^*(\beta))\mathfrak f} +1 =   \mathfrak f(\beta-\delta^*(\beta)).$$
Thus 
$$\lim_{\beta\to +\infty}   (\beta-\delta^*(\beta)) = \dfrac{1}{\mathfrak f}.$$
Thereby
$$\lim_{\beta\to +\infty}  M^{\mathfrak f}(\beta) = 9e\mathfrak f.$$

%
%

\subsection{Comparison principle (Section \ref{ssect:comp_prin_BSVIE})}

\subsubsection*{Proof of Proposition \ref{prop:comp_pinc_BSDE}}

This result is a comparison principle for the BSDE \eqref{eq:poss_BSDE}, where $X^\circ$ and $B$ are continuous (Assumption \ref{P1}). We follow the arguments of \cite[Theorem 3.25]{papa:poss:sapl:18}. Nonetheless since $B$ is supposed to be continuous, there are some important simplifications. 

Let us define
 $$(\fd Y,\fd Z, \fd  Y, \fd M, \fd  \xi) = (Y^1-Y^2, Z^1-Z^2,U^1-U^2,M^1- M^2,\xi^1-\xi^2)$$
 and 
$$\fd  f^{1}_s =  f^1(s,Y^1_{s-},Z_s^1,U^1_s) - f^1(s,Y^2_{s-},Z_s^2,U^2_s),  \quad \fd  f^{1,2}_s =  f^1(s,Y^2_s,Z_s^2,U^2_s) - f^2(s,Y^2_s,Z_s^2,U^2_s).$$
Let us stress that in $\fd f^1$, we consider $Y^i_{s-}$ and not $Y^i_s$. For a nonnegative predictable process $\gamma$, define $v= \int_0^\cdot \gamma_s \dd B_s$. Define $\cE(v)_t = \exp(v_t)$. The It\^o formula gives:
\begin{eqnarray*}
\cE(v)_t \fd Y_t & = & \cE(v)_T \fd \xi + \int_t^T \left[ \cE(v)_{s-} (\fd f^{1,2}_s + \fd f^1_s) -\gamma_s  \fd  Y_{s-} \right] \dd B_s - \int_t^T \cE(v)_{s-} \fd Z_s \dd X^{\circ}_s \\
& - &  \int_t^T  \int_{\bR^m} \cE(v)_{s-} \fd  U_s(x) \dd \tpi^{\natural}(\dd s, \dd x) - \int_t^T  \cE(v)_{s-} \dd \fd  M_s.
\end{eqnarray*}
See \cite[Theorem 3.20]{papa:poss:sapl:18} and note the modification due to the continuity of $B$. Let us also emphasize that the continuity of $B$ also leads to:
$$\int_t^T f^1(s,Y^1_s,Z^1_s,U^1_s) \dd B_s = \int_t^T f^1(s,Y^1_{s-},Z^1_s,U^1_s) \dd B_s.$$
Namely we can use this predictable version of the driver in the BSDE \eqref{eq:poss_BSDE}. This property fails in general and this is the reason of the \cite[section 3.6]{papa:poss:sapl:18}. Anyway the linearization procedure for $\delta f^1$ implies that
$$\fd  f^1_s = \lambda_s \fd  Y_{s-} + \eta_s b_s \fd  Z_s^\top + f^1(s,Y^2_{s-},Z_s^2,U^1_s) - f^1(s,Y^2_{s-},Z_s^2,U^2_s)$$
where $\lambda$ is a one-dimensional predictable process s.t. $|\lambda_s(\omega)|\leq \varpi^1_s(\omega)$ and $\eta$ is a $m$-dimensional predictable process such that $|\eta|^2\leq \theta^{1,\circ}$, $\dd \bP \otimes \dd B$-a.e. (using \ref{H2} for $f^1$). Let us choose $\gamma=\lambda$ and from \ref{P2}
\begin{eqnarray*}
\cE(v)_t\fd Y_t & \leq & \cE(v)_T\fd  \xi + \int_t^T \left[ \cE(v)_{s-}\fd  f^{1,2}_s  \right] \dd B_s \\
&+& \int_t^T \cE(v)_{s-} \eta_s b_s\fd  Z_s^\top \dd B_s - \int_t^T \cE(v)_{s-} \fd  Z_s \dd X^{\circ}_s \\
& + &\int_t^T \cE(v)_{s-} \widehat K_s (\fd  U_s(\cdot) \kappa_s(\cdot)) \dd B_s -  \int_t^T  \int_{\bR^m} \cE(v)_{s-} \fd U_s(x) \dd \tpi^{\natural}(\dd s, \dd x)\\
& -& \int_t^T  \cE(v)_{s-} \dd\fd  M_s.
\end{eqnarray*}
Now by the Girsanov's transform, if $\bQ$ is the probability measure defined by:
$$\frac{\dd \bQ}{\dd \bP} = \cE \left( \eta \cdot X^\circ + \kappa \star \tpi^\natural \right),$$
where $\cE(v)$ stands for the stochastic exponential operator, then the assumptions on $\kappa$ imply that $\bQ$ is equivalent to $\bP$. Taking the conditional expectation under $\bQ$ in the previous inequality yields to
\begin{eqnarray*}
\cE(v)_t\fd Y_t & \leq & \bE^{\bQ} \left[ \cE(v)_T\fd  \xi + \int_t^T \left[ \cE(v)_{s-}\fd  f^{1,2}_s  \right] \dd B_s \bigg| \cF_t \right] \leq 0. 
\end{eqnarray*}
The details concerning the disappeared martingale terms can be found in the proof of \cite[Theorem 3.25]{papa:poss:sapl:18}. Hence the conclusion of the Proposition follows. 
%

\subsubsection*{Proof of Proposition \ref{prop:comp_BSVIE_1}}

The arguments are the same as for \cite[Theorem 3.4]{wang:yong:15}. Let us consider the unique solution in $ \mathfrak{S}^2_{\beta}(\Delta^c)$ of the BSVIE:
 \begin{eqnarray} \nonumber
\bar Y(t) &= &\overline{\Phi}(t) + \int_t^{T} \bar f(t,s, \bar Y(s), \bar Z(t,s) ,\bar U(t,s)) \dd B_s - \int_t^{T} \bar Z(t,s) \dd X^\circ_s \\  \label{eq:intermediate_BSVIE_comp_princ}
&& \qquad - \int_t^{T}  \int_{\bR^m}  \bar U(t,s,x) \tpi(\dd x, \dd s) -\int_t^{T} \dd \bar M(t,s),
\end{eqnarray}
together with the solution of:
 \begin{eqnarray*} \nonumber
\widetilde Y_1(t) &= &\overline{\Phi}(t) + \int_t^{T} \bar f(t,s, \widetilde Y_0(s), \widetilde Z_1(t,s) ,\widetilde U_1(t,s)) \dd B_s - \int_t^{T} \widetilde Z_1(t,s) \dd X^\circ_s \\ 
&& \qquad - \int_t^{T}  \int_{\bR^m}  \widetilde U_1(t,s,x) \tpi(\dd x, \dd s) -\int_t^{T} \dd \widetilde M_1(t,s),
\end{eqnarray*}
where $ \widetilde Y_0 = Y^2$. Since 
$$\bar f(t,s,\widetilde Y_0(s),z,u) \leq f^2(t,s,\widetilde Y_0(s),z,u),\qquad \overline{\Phi}(t)  \leq \Phi^2(t),$$
from Proposition \ref{prop:comp_BSVIE_without_Y}, there exists a measurable set $\Omega^1_t$ verifying $\bP(\Omega^1_t) = 0$ such that 
$$\widetilde Y_1(\omega,t) \leq \widetilde Y_0(\omega,t),\quad \omega \in \Omega\setminus \Omega^1_t, \quad t \in [0,T].$$
Then we consider the BSVIE 
 \begin{eqnarray*} \nonumber
\widetilde Y_2(t) &= &\overline{\Phi}(t) + \int_t^{T} \bar f(t,s, \widetilde Y_1(s), \widetilde Z_2(t,s) ,\widetilde U_2(t,s)) \dd B_s - \int_t^{T} \widetilde Z_2(t,s) \dd X^\circ_s \\ 
&& \qquad - \int_t^{T}  \int_{\bR^m}  \widetilde U_2(t,s,x) \tpi(\dd x, \dd s) -\int_t^{T} \dd \widetilde M_2(t,s).
\end{eqnarray*}
The preceding arguments show that for any $t\in [0,T]$, there exists a measurable set $\Omega^2_t$ verifying $\bP(\Omega^2_t) = 0$ such that 
$$\widetilde Y_2(\omega,t) \leq \widetilde Y_1(\omega,t),\quad \omega \in \Omega\setminus \Omega^2_t, \quad t \in [0,T].$$
By induction we can construct a sequence $( \widetilde Y_k, \widetilde Z_k, \widetilde U_k, \widetilde Y_k) \in \mathfrak{S}^2_{\beta}(\Delta^c)$ and $\Omega^k_t$ satisfying $\bP(\Omega^k_t) = 0$,
 \begin{eqnarray*} \nonumber
\widetilde Y_k(t) &= &\overline{\Phi}(t) + \int_t^{T} \bar f(t,s, \widetilde Y_{k-1}(s), \widetilde Z_k(t,s) ,\widetilde U_k(t,s)) \dd B_s - \int_t^{T} \widetilde Z_k(t,s) \dd X^\circ_s \\ 
&& \qquad - \int_t^{T}  \int_{\bR^m}  \widetilde U_k(t,s,x) \tpi(\dd x, \dd s) -\int_t^{T} \dd \widetilde M_k(t,s),
\end{eqnarray*}
and for any $t\in [0,T]$ and $\omega \in \Omega \setminus\left(  \bigcup_{k\geq 1} \Omega^k_t \right)$:
$$Y^2(\omega,t) \geq \widetilde Y_1(\omega,t) \geq \widetilde Y_2(\omega,t) \geq \cdots$$
The set $  \bigcup_{k\geq 1} \Omega^k_t $ is a $\bP$-null set. And using Lemma \ref{lem:stab_BSVIE_h_poss} and Inequality \eqref{eq:stability_H2-estimate} yields to:
\begin{eqnarray*} 
&& \bE \left[ \int_0^T e^{ \beta A_t} |\widetilde Y_k(t)-\widetilde Y_\ell(t)|^2 \dd A_t +  \int_S^T e^{(\beta-\delta) A_t}  \int_t^T e^{\delta A_s} \dd  \trace \langle \what \mart(t,\cdot)\rangle_s    \dd A_t \right] \\
&& \quad \leq \Sigma^\mathfrak f(\beta) \bE \left[ \int_0^T  e^{(\beta-\delta) A_t}  \int_t^T e^{\delta A_s}\dfrac{ |\what f(t,s)|^2}{\alpha_s^2} \dd B_s  \dd A_t \right] ,
\end{eqnarray*}
with 
$$\what f(t,s) = \bar f(t,s, \widetilde Y_{k-1}(s), \widetilde Z_k(t,s) ,\widetilde U_k(t,s))-\bar f(t,s, \widetilde Y_{\ell-1}(s), \widetilde Z_k(t,s) ,\widetilde U_k(t,s)).$$
The condition \ref{H2} on $\bar f$ leads to
$$ \bE \left[ \int_0^T e^{ \beta A_t} |\widetilde Y_k(t)-\widetilde Y_\ell(t)|^2 \dd A_t \right] \leq \widetilde \Sigma^\mathfrak f(\beta) \bE \left[ \int_0^T  e^{\beta A_t}  |\widetilde Y_{k-1}(t)-\widetilde Y_{\ell-1}(t)|^2  \dd A_t \right] .
$$
From Hypothesis \ref{eq:cond_beta_Type_I_BSVIE} of Theorem \ref{thm:type_I_general_BSVIE}, the sequence  $( \widetilde Y_k, \widetilde Z_k, \widetilde U_k, \widetilde Y_k)$ is a Cauchy sequence in $\mathfrak{S}^2_{\beta}(\Delta^c)$, converging to the solution $(\bar Y,\bar Z,\bar U,\bar M)$ of the BSVIE \eqref{eq:intermediate_BSVIE_comp_princ}. Thereby we prove that: for any $t\in[0,T]$
$$\bar Y(t) \leq Y^2(t), \quad \mbox{a.s.}.$$
Similar arguments imply that $Y^1(t) \leq \bar Y(t)$ and achieve the proof of the Proposition.

\subsection{$L^p$-continuity } \label{ssect:Lp_continuity}
If $(Y,Z,U,M)$ solves the BSVIE \eqref{eq:second_special_BSVIE}, then taking $h(t,s,z,\psi)=f(t,s,Y(s),z,\psi)$ and using the estimate \eqref{eq:Hp-estimate} of Lemma \ref{lem:param_BSDE->BSVIE}, we have:
\begin{eqnarray*} \nonumber
&& \bE \left[  |Y(t)|^p +  \left( \int_S^T  |Z(t,r)|^2 \dd r \right)^{\frac{p}{2}} + \left(  \langle M(t,\cdot) \rangle_{S,T} \right)^{\frac{p}{2}} \right. \\ 
&&\left.+  \left(  \int_S^T  \int_{\bR^m} |U(t,r,x)|^2 \pi(\dd x, \dd r) \right)^{\frac{p}{2}}\right]  \leq C \bE \left[|\Phi(t)|^p  + \left( \int_S^T |h(t,r,0,0)| \dd r \right)^p \right].
\end{eqnarray*}
Since $f$ is Lipschitz continuous, taking $S=t$, the Gronwall inequality leads to
\begin{eqnarray*} \nonumber
&& \bE \left[  |Y(t)|^p +  \left( \int_t^T  |Z(t,r)|^2 \dd r \right)^{\frac{p}{2}} + \left(  \langle M(t,\cdot) \rangle_{t,T} \right)^{\frac{p}{2}} \right. \\ 
&&\qquad \left.+  \left(  \int_t^T  \int_{\bR^m} |U(t,r,x)|^2 \pi(\dd x, \dd r) \right)^{\frac{p}{2}}\right] \\
&& \quad  \leq C \bE \left[|\Phi(t)|^p  +  \left( \int_t^T |\Phi(r)| \dd r \right)^p + \left( \int_t^T |f^0(t,r)| \dd r \right)^p \right].
\end{eqnarray*}
Under our stronger integrability conditions on $\Phi$ and $f^0$, we obtain an stronger estimate on $(Z,U,M)$:
$$\sup_{t\in [0,T]} \bE \left[ \left( \int_t^T  |Z(t,r)|^2 \dd r \right)^{\frac{p}{2}} + \left(  \langle M(t,\cdot) \rangle_{t,T} \right)^{\frac{p}{2}} + \left(  \|U(t,\cdot)\|^2_{\bL^2_\pi(t,T)} \right)^{\frac{p}{2}} \right] <+\infty.$$
This property is important to get the c\`adl\`ag in mean property of $Y$. 

\begin{Lemma} \label{lem:L2_continuity_Y}
Assume that {\rm \ref{H2}} and {\rm \ref{H3star}} hold. Then the solution of the BSVIE \eqref{eq:second_special_BSVIE} satisfies: for any $(t,t')\in [S,T]$ and if $t_\star=t \wedge t'$ and $t^\star=t\vee t'$:
\begin{eqnarray*}
 && \bE \left[  | Y(t)- Y(t') |^p  \right]\\
 &&\qquad + \bE \left[  \left( \int_S^{t_\star} |Z(t,r)-Z(t',r)|^2 dr \right)^{p/2} + \left( \int_{t^\star}^T  |Z(t,r)-Z(t',r) |^2 dr \right)^{p/2} \right] \\
 &&\qquad  + \bE \left[  \left( \left[ M(t,\cdot)-M(t',\cdot)\right]_{S,{t_\star}} \right)^{p/2} + \left(  [  M(t,\cdot)-M(t',\cdot) ]_{t^\star,T} \right)^{p/2}\right]  \\
&& \qquad +  \bE \left[  \left( \|U(t,\cdot)-U(t',\cdot)\|^2_{\bL^2_\pi(S,t_\star)} \right)^{p/2} +  \left(  \|U(t,\cdot)-U(t',\cdot)\|^2_{\bL^2_\pi(t^\star,T)} \right)^{p/2} \right]\\
&&\quad   \leq C  \bE \left[ |\Phi(t)-\Phi(t')|^p \right] + C \bE\left[ \left( \int_{t_\star}^{t^\star} \left| h(t,r,Z(t,r),\psi(t,r)) \right| dr \right)^p \right]  \\ 
 && \qquad +C \bE\left[ \left( \int_{t_\star}^{t^\star} \left| Z(t,r) \right|^2 dr \right)^{p/2} \right]+ C \bE\left[ \left(  \left\| U(t,\cdot) \right\|_{\bL^2_\pi(t_\star,t^\star)}^2  \right)^{p/2} \right] \\
 && \qquad + C  \bE\left[\left(  \left[ M(t,\cdot) \right]_{t_\star,t^\star} \right)^{p/2} \right] \\
 &&\qquad +  C  \bE \left[ \left( \int_{t^\star}^T |h(t,r,Z(t,r),\psi(t,r))-h(t',r,Z(t,r),\psi(t,r)) |dr \right)^p   \right].
 \end{eqnarray*}

\end{Lemma}
\begin{proof}
We first consider the BSVIE \eqref{eq:special_BSVIE}
\begin{eqnarray*} \nonumber
Y(t) &= &\Phi(t) + \int_t^{T} h(t,s, Z(t,s) ,U(t,s)) ds - \int_t^{T} Z(t,s) dW_s  \\
&& \qquad - \int_t^{T}  \int_\cE U(t,s,e) \tpi(de,ds) -\int_t^{T} d M(t,s).
\end{eqnarray*}
We take $t,t'$ in $[S,T]$ and w.l.o.g. let $S\leq t \leq t' \leq T$. Applying \eqref{eq:BSDE_stability} to the solution of the BSDE with parameter $t$, we obtain:
\begin{eqnarray*} \nonumber
&& \bE \left[\sup_{r\in[t',T]}   | \lambda(t,r)- \lambda(t',r) |^p +  \left( \int_{t'}^T  |z(t,r)-z(t',r) |^2 dr \right)^{p/2}  \right. \\ 
&&\qquad \qquad \left. + \left(  [  m(t,\cdot)-m(t',\cdot) ]_{t',T} \right)^{p/2}+  \left(  \int_{t'}^T  \int_{\cE} |u(t,r,e)-u(t',r,e)|^2 \pi(de, dr) \right)^{p/2}\right] \\ \nonumber 
&& \quad  \leq C  \bE \left[ |\Phi(t)-\Phi(t')|^p + \left( \int_{t'}^T |h(t,r,z(t,r),u(t,r))-h(t',r,z(t,r),u(t,r)) |dr \right)^p   \right] \\
&& \quad = C  \bE \left[ |\Phi(t)-\Phi(t')|^p + \left( \int_{t'}^T |h(t,r,Z(t,r),\psi(t,r))-h(t',r,Z(t,r),\psi(t,r)) |dr \right)^p   \right]
\end{eqnarray*}
Remark that 
\begin{eqnarray*}
 && \bE \left[  | Y(t)- Y(t') |^p  \right] = \bE \left[  | \lambda(t,t)- \lambda(t',t') |^p  \right]  \\
 && \quad \leq  C \bE \left[  | \lambda(t,t)- \lambda(t,t') |^p  \right] + C \bE \left[  \sup_{r\in[t',T]}   | \lambda(t,r)- \lambda(t',r) |^p  \right] \\
 && \quad \leq  C \bE\left[ \left( \int_t^{t'} \left| h(t,r,Z(t,r),\psi(t,r)) \right| dr \right)^p \right] +C \bE\left[ \left( \int_t^{t'} \left| Z(t,r) \right|^2 dr \right)^{p/2} \right] \\ 
 && \qquad + C \bE\left[ \left( \int_t^{t'} \left| \psi(t,r) \right|_{\bL^2_\pi}^2 dr \right)^{p/2} \right] + C  \bE\left[\left(  \left[ M(t,\cdot) \right]_{t,t'} \right)^{p/2} \right] \\
 && \qquad + C \bE \left[  \sup_{r\in[t',T]}   | \lambda(t,r)- \lambda(t',r) |^p  \right]. 
 \end{eqnarray*}
Moreover the notion of M-solution (Equation \eqref{eq:M_sol_def}) implies that 
\begin{eqnarray*} 
&& Y(t)-Y(t') - \bE \left[ Y(t)-Y(t') |\cF_S\right] = \int_S^t (Z(t,r)-Z(t',r)) dW_r \\
&& \qquad + \int_S^t \int_{\cE} (U(t,r,e)-U(t',r,e)) \tpi(de,dr) + \int_S^t d(M(t,r)-M(t',r)) \\
&&\qquad + \int_t^{t'}  Z(t',r) dW_r + \int_t^{t'} \int_{\cE} U(t',r,e) \tpi(de,dr) + \int_t^{t'} dM(t',r).
\end{eqnarray*}
Using BDG's inequality, we get that 
\begin{eqnarray*} 
&& \bE \left[  \left( \int_S^t |Z(t,r)-Z(t',r)|^2 dr \right)^{p/2} \right] + \bE \left[  \left( \left[ M(t,\cdot)-M(t',\cdot)\right]_{S,t} \right)^{p/2} \right]  \\
&& \qquad +  \bE \left[  \left( \int_S^t \int_{\cE} |U(t,r,e)-U(t',r,e)|^2 \pi(de,dr) \right)^{p/2} \right] \\
&&\qquad + \bE \left[  \left( \int_t^{t'} |Z(t',r)|^2 dr \right)^{p/2} \right] + \bE \left[  \left( \left[M(t',\cdot)\right]_{t,t'} \right)^{p/2} \right]  \\
&& \qquad +  \bE \left[  \left( \int_t^{t'} \int_{\cE} |U(t',r,e)|^2 \pi(de,dr) \right)^{p/2} \right]  \\
&& \leq C \bE \left[ \left| Y(t)-Y(t') - \bE \left[ Y(t)-Y(t') |\cF_S\right] \right|^p \right] \leq C \bE \left[ \left| Y(t)-Y(t') \right|^p \right].
\end{eqnarray*}
Combining the previous inequalities, we obtain the desired result for the BSVIE \eqref{eq:special_BSVIE}. For the BSVIE \eqref{eq:second_special_BSVIE}, we apply the preceding arguments using the generator $h(t,s,z,\psi)=f(t,s,Y(s),z,\psi)$. 
\end{proof}

From this lemma, it is possible to deduce that $Y$ belongs to $D([0,T];\bL^p_{\bF}(\Omega))$, provided that we have regularity assumption on $t\mapsto \Phi(t)$ and $t\mapsto f(t,s,y,z,\psi)$, as in \cite[Theorem 4.2]{yong:08} in the continuous setting. Note that the estimate on $(Z,U,M)$ is crucial here. Let us emphasize again that it does not mean that $Y$ is in $D^\sharp([0,T];\bL^p_{\bF}(\Omega)) $; in other words we do not deduce that a.s. the paths are c\`adl\`ag. 

\bigskip
Let us now prove Lemma \ref{lem:exp_ineq}. \\
\noindent \begin{proof}
Fix one $t \in [0,T]$ such that the equation is satisfied and define on $[t,T]$
\begin{eqnarray*}
X_t(u) & = & Y(t) - \int_t^u f(t,s) \dd s - (\mart(t,u)-\mart(t,t)) \\
&=& \bE \left[ Y(t) + \int_t^T f(t,s) \dd s \bigg| \cF_u\right] + \int_u^T f(t,s) \dd s.
\end{eqnarray*}
The process $X_t=\{X_t(u),\ u \in [t,T]\}$ is a c\`adl\`ag semimartingale. And by Doob's martingale inequality, we have
\begin{eqnarray} \nonumber
\bE \left[ \sup_{u\in [t,T]} |X_t(u)|^2 \right] & \leq & C \bE \left[|X_t(T)|^2 + \left( \int_t^T |f(t,s)|\dd s \right)^2  \right] \\ \label{eq:estim_X_p}
&\leq & C  \bE \left[ |Y(t)|^2 + \left( \int_t^T |f(t,s)|\dd s \right)^2  \right]. 
\end{eqnarray}
Using It\^o's formula for $u \mapsto |X_t(u)|^2 e^{\beta(u-t)}$, on $[t,T]$ we obtain that:
\begin{eqnarray} \nonumber
& & |Y(t)|^2+  \int_t^T e^{\beta(s-t)}  |Z(t,s)|^2  \dd s \\ \nonumber
&&\qquad +  \int_{t}^{T} e^{\beta(s-t)}  \int_{\bR^m} \left| U(t,s,x)\right|^2   \pi(\dd x,\dd  s)+\int_t^T e^{\beta(s-t)}   \dd \langle M(t,\cdot) \rangle_{0,s}  \\ \nonumber
&&\quad \leq |X_t(T)|^2 e^{\beta(T-t)} + 2 \int_t^T e^{\beta(s-t)} X_t(s) f(t,s) \dd s  - \beta  \int_t^T |X_t(s)|^2 e^{\beta(s-t)} \dd s\\ \label{eq:Ito_formula_p=2}
& &\qquad + 2 \int_t^T e^{\beta(s-t)} X_t(s)  \dd \mart(t,s).
\end{eqnarray}
From our hypotheses and the control of $u\mapsto X_t(u)$, the martingale terms are true martingales. By the Young's inequality we obtain:
$$2 \int_t^T e^{\beta(s-t)} X_t(s) f(t,s) \dd s \leq \beta \int_t^T e^{\beta(s-t)} |X_t(s)|^2 \dd s +\frac{1}{\beta} \int_t^T e^{\beta(s-t)} |f(t,s)|^2 \dd s $$
Thus, since $X_t(T) = \Phi(t)$, taking the conditional expectation w.r.t. $\cF_t$ in \eqref{eq:Ito_formula_p=2} gives the desired control.
\end{proof}

\bibliography{biblio}

\end{document}